\documentclass[10pt,reqno]{amsart}
\usepackage{amsfonts}
\usepackage{amsmath}
\usepackage{amssymb}
\usepackage{latexsym, bm}
\usepackage[RGB]{xcolor}
\usepackage{mathrsfs}
\usepackage{amsthm}
\usepackage{hyperref}
\numberwithin{equation}{section}
\newtheorem{theorem}{Theorem}[section]
\newtheorem{assumption}[theorem]{Assumption}
\newtheorem{lemma}[theorem]{Lemma}
\newtheorem{remark}[theorem]{Remark}
\newtheorem{corollary}[theorem]{Corollary}
\newtheorem{proposition}[theorem]{Proposition}
\newcommand{\dif}{\mathrm{d}}

\newcommand{\SUM}[3]{\sum\limits_{{#1}={#2}}^{#3}}

\begin{document}
\title[IBVP for relaxed  compressible Navier-Stokes system]{Global well-posedness and relaxation limit for relaxed  compressible Navier-Stokes-Fourier equations in bounded domain}
\author{Yuxi Hu and Xiaoning Zhao}
\thanks{\noindent Yuxi Hu, Department of Mathematics, China University of Mining and Technology, Beijing, 100083, P.R. China, yxhu86@163.com\\
\indent Xiaoning Zhao, Department of Mathematics, China University of Mining and Technology, Beijing, 100083, P.R. China}
\maketitle
\allowdisplaybreaks

 , 

\begin{abstract}
This paper investigates an initial boundary value problem for the relaxed one-dimensional compressible Navier-Stokes-Fourier equations. By transforming the system into Lagrangian coordinates, the resulting formulation exhibits a uniform characteristic boundary structure. We first construct an approximate system with non-characteristic boundaries and establish its local well-posedness by verifying the maximal nonnegative boundary conditions. Subsequently, through the construction of a suitable weighted energy functional and careful treatment of boundary terms, we derive uniform a priori estimates, thereby proving the global well-posedness of smooth solutions for the approximate system. Utilizing these uniform estimates and standard compactness arguments, we further obtain the existence and uniqueness of global solutions for the original system. In addition, the global relaxation limit is established. The analysis is fundamentally based on energy estimates.
{\bf AMS classification code}: 35L50, 35A01, 35B25
\end{abstract}

\section{Introduction}

The fundamental equations governing the dynamic behavior of a one-dimensional compressible fluid are given by:
\begin{align}\label{1.1}
	\begin{cases}
		\rho_t+(\rho u)_x=0,\\
		\rho  u_t+\rho u u_x+p_x=S_x,\\
		(\rho \mathcal E)_t+(\rho u \mathcal E+pu+q-Su)_x=0, 
	\end{cases}
\end{align}
with $ (t,x)\in \mathbb{R}^+\times\Omega$ and $\Omega$ is a bounded domain in $\mathbb{R}$. Here, $\rho, u, p, \mathcal E, S,$ and $q $ denote fluid density, velocity, pressure, total energy per unit mass, stress, and heat flux, respectively. 

To close the system $\eqref{1.1}$, the stress $S$ and the heat flux $q$ must be specified. Instead of employing the classical relations  
$$
q = -\kappa(\theta) \theta_x, \qquad\quad S = \mu u_x,  
$$  
where $\theta$ denotes the temperature, $\kappa(\theta)$ represents the thermal conductivity coefficient as a smooth function of $\theta$, and $\mu$ is a positive constant denoting the viscosity coefficient, we consider the relaxed versions in the form of the (nonlinear) Cattaneo law for heat conduction  
\begin{align}\label{1.2}  
\tau_1   (\theta) (\rho q_t + \rho u \cdot q_x) + q + \kappa(\theta) \theta_x = 0,  
\end{align}  
and the Maxwell-type constitutive relation for the stress tensor  
\begin{align}\label{1.3}  
\tau_2   (\rho S_t + \rho u \cdot S_x) + S = \mu u_x.  
\end{align}  
Here, \( \tau_1   (\theta) = \tau    g(\theta) > 0 \) and \( \tau_1   > 0 \) are relaxation parameters, where \( g(\theta) \) is a smooth function of \( \theta \), and \( \tau_1    \) and \( \tau_2   \) are constants.

The constitutive relation \eqref{1.2}, in both its high dimensional and linear forms, was first proposed by Cattaneo \cite{CA48} to resolve the paradox of infinite propagation speed that arises when using Fourier's law to describe heat conduction. Later, Christov \cite{CHR09} introduced an objective version of the linear Cattaneo’s law in the following form:  
\begin{align}\label{CC}  
\hat \tau (q_t+u\cdot \nabla q-q\cdot \nabla u+(\nabla \cdot u) q)+q+\kappa \nabla \theta=0.  
\end{align}  
Notably, various physical experiments have shown that the parameters \( \hat \tau \) and \( \kappa \) may depend on \( \theta \) and \( \rho \) (see \cite{CHE63, CW03}). By assuming \( \hat \tau = \tau_1   (\theta) \rho \) and \( \kappa = \kappa(\theta) \), the objective constitutive relation \eqref{CC} simplifies to \eqref{1.2}.

On the other hand, the high-dimensional form of equation \eqref{1.3} (without the factor \(\rho\)) was first introduced by Maxwell \cite{Max1867} to describe the dynamic behavior of viscoelastic fluids and other complex fluids. An objective version of Maxwell’s law was later presented by Oldroyd \cite{OLD50} in the following three-dimensional form:  
\begin{align}\label{OL}  
\tilde{\tau} (S_t + u \cdot \nabla S + g_a(S, \nabla u)) + S = \mu \left(\nabla u + \nabla u^T - \frac{2}{3} \operatorname{div} u \, I_3 \right) + \lambda \operatorname{div} u \, I_3,  
\end{align}  
where \( \lambda \) and \( \mu \) denote the volume and shear viscosity coefficients, respectively. The function \( g_a(S, \nabla u) \) is given by  
\[
g_a(S, \nabla u) := S W(u) - W(u) S - a (D(u) S + S D(u)), \quad -1 \leq a \leq 1,
\]  
with  
\[
D(u) = \frac{1}{2} (\nabla u + (\nabla u)^T), \quad W(u) = \frac{1}{2} (\nabla u - (\nabla u)^T).
\]  
The parameter \( a \) determines different variants of the Maxwell model: \( a = 1 \) corresponds to the upper-convected Maxwell model (UCM), \( a = -1 \) to the lower-convected Maxwell model (LCM), and \( a = 0 \) to the corotational Maxwell model.  
By assuming \( \tilde{\tau} = \tau_2   \rho \) and setting \( a = 0 \), Oldroyd’s constitutive equation \eqref{OL} simplifies to \eqref{1.3}. Note that the assumption \( \tilde{\tau} = \tau_2   \rho \) was recently proposed by Freistühler \cite{Frei1} in the isentropic fluid dynamic equations, with similar ideas dating back to Ruggeri \cite{Rug83} and M\"uller \cite{Mu67} for the non-isentropic case.  

We further remark that even for simple fluids, the small relaxation parameter \(  \tilde \tau  \) (typically ranging from \( 10^{-9} \) to \( 10^{-12} \) seconds) cannot be neglected. This was demonstrated by Pelton et al. \cite{Peetal} in experiments on the high-frequency vibration of nanoscale mechanical devices immersed in water-glycerol mixtures. Notably, as also discussed in \cite{ChSa015}, equation \eqref{1.3} provides a general framework for characterizing fluid-structure interactions in nanoscale mechanical devices vibrating in simple fluids.

We note that, compared to other relaxed constitutive equations (see \cite{HR20, HRW22, Yong14, Peng-Zhao-22}), the relations \eqref{1.2} and \eqref{1.3} offer several advantages. First, they are objective, ensuring physical consistency. Second, their conservation form facilitates the definition of weak solutions. Most importantly, with the given constitutive relations \eqref{1.2} and \eqref{1.3}, the system \eqref{1.1} possesses a physical entropy. This entropy structure allows for immediate \( L^2 \) estimates without requiring any smallness conditions and provides a framework for analyzing potential blow-up phenomena in the presence of large initial data. For further details, see \cite{HR24}.

Now, let us return to the system \eqref{1.1}–\eqref{1.3}. Due to the use of different constitutive relations compared to the classical system, the energy formulation may be altered. In this paper, we assume that the total energy is given by  
\begin{equation} \label{eq1.3a}
	\mathcal{E} = \frac{1}{2} u^2 + \frac{\tau_2  }{2\mu} S^2 + e(\theta, q),
\end{equation}  
which can be interpreted as the sum of kinetic energy \( \frac{1}{2} u^2 \), relaxed potential energy \( \frac{\tau_2  }{2\mu} S^2 \), and internal energy \( e(\theta, q) \). Specifically, the internal energy \( e \) is assumed to depend on \( q \) and is given by  
\begin{align}\label{1.4}
	e(\theta, q) = C_v \theta + a(\theta) q^2, 
\end{align}  
where  
\[
a(\theta) = \frac{Z(\theta)}{\theta} - \frac{1}{2} Z'(\theta),  
\quad \text{with} \quad Z(\theta) = \frac{\tau_1   (\theta)}{\kappa(\theta)}.
\]  
Here, \( C_v > 0 \) represents the heat capacity at constant volume. Additionally, we assume the pressure \( p \) is given by \( p(\rho, \theta) = R \rho \theta \), where \( R > 0 \) denotes the gas constant. It is straightforward to verify that \( p \) and \( e \) satisfy the standard thermodynamic relation:  
\[
\rho^2 e_\rho = p - \theta p_\theta.
\]  
The dependence of internal energy on \( q^2 \) was first highlighted by Coleman et al. \cite{CFO}, who rigorously proved that for heat equations with a Cattaneo-type law, the formulation \eqref{1.4} is consistent with the second law of thermodynamics. For further discussions, see also \cite{CG, CHO, TA}.

We shall consider the initial boundary value problem for the functions in system \eqref{1.1}-\eqref{1.3}:
\[
(\rho,u,\theta,q,S):   [0,+\infty)\times\Omega \rightarrow \mathbb R^+\times \mathbb R \times \mathbb R^+\times \mathbb R\times \mathbb R
\]
with initial conditions
\begin{align}\label{1.6}
	(\rho(x,0),u(x,0),\theta(x,0),q(x,0),S(x,0))=(\rho_0,u_0,\theta_0,q_0,S_0)
\end{align}
and boundary conditions
\begin{align}\label{1.7}
	u\big|_{\partial\Omega}=q\big|_{\partial\Omega}=0.
\end{align}

We note that the Cauchy problem for the system \eqref{1.1}–\eqref{1.3} and its multi-dimensional form has been extensively studied. In particular, Yong \cite{Yong14} first investigated the isentropic Navier-Stokes equations with a revised Maxwell law (in its linear form), establishing a local well-posedness theory and a local relaxation limit. These results were later extended by Hu and Racke \cite{HR17, HR23} to the non-isentropic case and by Peng \cite{PY2020} to a more general setting.  
The full system \eqref{1.1}–\eqref{1.3} in its one-dimensional form (without the factor \( \rho \) in \eqref{1.2}–\eqref{1.3}) was studied in \cite{HR20}, where the global well-posedness of smooth solutions and the local relaxation limit were established. This work was further extended by Peng and Zhao \cite{Peng-Zhao-22, Peng-Zhao-24} to obtain the global relaxation limit.  
The blow-up phenomenon for this system has also been extensively investigated in \cite{HW2019, HW2020, HRW22, HR24}. For related results in more general systems, see also \cite{Bae023}.  

On the other hand, there are very few results concerning the initial boundary value problem for system \eqref{1.1}-\eqref{1.3}. To our knowledge, only isentropic case was  studied very recently by the first author and Li in \cite{HL24}.  In this paper, we are aim to study the non-isentropic system with more physical coefficients.  Due to the strong coupling between the momentum and energy equations, this extension from isentropic to non-isentropic system  is highly non-trivial. In particular, beyond addressing the challenges arising from the increased complexity of the non-isentropic case, we identify a novel structural property of the system: certain uncontrolled high-order terms involving $u$ and $\theta$ naturally cancel each other. This key observation allows us to close the necessary estimates and establish global solutions in the usual sense.

For  computational convenience,  using \eqref{eq1.3a} and \eqref{1.4}, we rewrite the system \eqref{1.1}-\eqref{1.3} in Lagrangian coordinates as follows:
\begin{align}\label{yuan}
	\begin{cases}
		v_t=u_x,\\
		u_t+{p(v,\theta)}_x=S_x,\\
		e_{\theta}\theta_t-\frac{2a(\theta)}{Z(\theta)}q\theta_x+\frac{R\theta u_x}{v}+q_x=\frac{2a(\theta)}{\tau   (\theta)}q^2v+\frac{S^2v}{\mu},\\
		\tau_1   (\theta)q_t+vq+\kappa(\theta)\theta_x=0,\\
		\tau_2  S_t+vS=\mu u_x,
	\end{cases}
\end{align}
where $v=\frac{1}{\rho}$ is the specific volume per unit mass. The initial and boundary conditions \eqref{1.6}-\eqref{1.7} turn to
\begin{align}\label{initial}
	(v,u,\theta,S,q)(0,x)=(v_0,u_0,\theta_0,S_0,q_0)(x), \quad x\in\Omega
\end{align}
and
\begin{align}\label{boundary}
	u(t,x)\big|_{\partial\Omega}=0,q(t,x)\big|_{\partial\Omega}=0, \quad t>0.
\end{align}

To illustrate the complexity of the problem \eqref{yuan}-\eqref{boundary}, we reformulate it as a first-order hyperbolic system. Let $U=(v,u,\theta,q,S)$, then the system \eqref{yuan} can be written as  
\begin{align}
	A^0(U)U_t+A^1(U)U_x+B(U)U=F(U),
\end{align}
where
\begin{align*}
	&A^0(U)=\mathrm{diag}\{1,1,e_\theta,\tau_1   (\theta),\tau_2  \},
	A^1(U)=
	\begin{pmatrix}
		0& -1& 0& 0& 0\\
		-\frac{R\theta}{v^2}& 0& \frac{R}{v}& 0& -1\\
		0& \frac{R\theta}{v}& -\frac{2a(\theta)}{Z(\theta)}q& 1& 0\\
		0& 0& \kappa(\theta)& 0& 0\\
		0& -\mu& 0& 0& 0\\
	\end{pmatrix},\\
	&B(U)=\mathrm{diag}\{0,0,0, v, v\}, F(U)=\mathrm{diag}\left\{0, 0, -\left(\frac{2a(\theta)}{\tau   (\theta)}q^2+\frac{S^2}{\mu}\right), 0, 0\right\}
\end{align*}
Note that for any \( U \in G := \{\mathbb{R}^+, \mathbb{R}, \mathbb{R}^+, \mathbb{R}, \mathbb{R}\} \), we have \( \det(A^0)^{-1} A^1(U) = 0 \), which implies that \eqref{boundary} is a uniformly characteristic boundary condition.  
It is important to note that for first-order hyperbolic systems with a uniform characteristic boundary, no general well-posedness results exist, even in the local setting \cite{Chen2007, SCH86}. The main difficulty arises from the inability to directly estimate \( U_x \) or \( U_{xx} \) using the estimates of \( U_t \) and \( U_{tt} \). As a result, a loss of derivatives may occur near the boundary, posing significant analytical challenges.

Note that when $\tau_1   =\tau_2  =0$,  system \eqref{1.1} is reduced to classical non-isentropic Navier-Stokes equations as follows:
\begin{align}\label{classical}
\begin{cases}
v_t=u_x,\\
u_t+p(v,\theta)_x=\left(\frac{\mu u_x}{v}\right)_x,\\
C_v\theta_t+\frac{R\theta u_x}{v}
=\left(\frac{\kappa(\theta)\theta_x}{v}\right)_x+\frac{\mu u^2_x}{v},
\end{cases}
\end{align}
for which the large global solutions (without vacuum) was already known, see Kazhikhov \cite{KA}. But the methods there can not be applied to the relaxed system due to the essential change of structure, i.e., from hyperbolic-parabolic to pure hyperbolic system. On the other hand, it has been shown that, see \cite{HR24}, solutions to the relaxed system may blowup in finite time for some large data. 
Therefore, a global defined smooth solutions of system \eqref{yuan}-\eqref{boundary} should not be expected for large data.

We outline the key ideas of our approach as follows. First, following the methodology in \cite{HL24} for the isentropic case and drawing inspiration from \cite{SCH86} on compressible Euler equations, we construct an approximate system with a non-characteristic boundary. By verifying that the boundary condition is maximally nonnegative, we establish local well-posedness using classical results. Next, we derive uniform a priori estimates and obtain global, uniformly smooth solutions for the approximate system. During this process, suitable weighted energy functionals are introduced to obtain uniform estimates with respect to the relaxation parameters. Furthermore, second-order boundary terms naturally emerge when performing integration by parts. By introducing appropriate auxiliary functions and leveraging the intrinsic structure of the system, these second-order boundary terms are transformed into interior integrals, although they involve third-order terms. Fortunately, due to the system's built-in cancellation mechanism, these high-order terms vanish. With the uniform estimates in hand, we then apply a standard compactness argument to obtain a global smooth solution and establish the global relaxation limit for the original system \eqref{yuan}--\eqref{boundary}.

Now, we introduce some notations. $W^{m,p}=W^{m,p}(\Omega),0\le m\le \infty,1\le p\le \infty$, denotes the usual Sobolev space with norm $\left\| \cdot \right\|_{W^{m,p}}$, $H^m$ and $L^p$ stand for $ W^{m,2}$ resp. $W^{0,p}$.  $D^\alpha =\partial_t^{\alpha_1} \partial_x^{\alpha_2}, \, \alpha=(\alpha_1, \alpha_2), \, |\alpha|=\alpha_1+\alpha_2$. 

The following assumptions are needed throughout the paper:
\begin{assumption}\label{assu}
\mbox{}\\
(1) $\tau_1   $ is the same order as $\tau_2  $, i.e.,  \( \tau_1    = O(\tau_2  ) \) as \( \tau_2   \to 0 \).  Without loss of generality, let $\tau:=\tau_1   =\tau_2  $.\\
(2) The initial and boundary data satisfy the usual compatibility condition
\begin{align}\label{compatibility condition}
	\partial^k_tu(0,x)\big|_{\partial\Omega}=\partial^k_tq(0,x)\big|_{\partial\Omega}=0,
	k=0,1
\end{align}
where $\partial_tu(0,x)$ and $\partial_tq(0,x)$ are defined recursively by equation $\eqref{yuan}_2$ and $\eqref{yuan}_4$, respectively.\\
(3) The initial data is well-prepared in the following sense:
\begin{align}\label{new-hu4-1}
\|v_0 q_0+\kappa (\theta_0) (\theta_0)_x\|_{H^1}=O(\sqrt{\tau}), \, \| v_0S_0-\mu (u_0)_x \|_{H^1} =O(\sqrt{\tau}),\, \mathrm{as}\, \tau \rightarrow 0.
\end{align}
(4) $V_0^k \in H^{2-k} $ for $k=0, 1, 2$, where 
\begin{align*}
V_0^k:= \partial_t^k\left(v-1, u ,\theta-1, \sqrt{\tau} q, \sqrt{\tau} S\right)(t=0, \cdot), \, k=0, 1\\
V_0^2:= \tau \partial_t^2 (v-1, u, \theta-1, \sqrt{\tau} q, \sqrt{\tau} S)(t=0, \cdot),
\end{align*}
where $V_0^1, V_0^2$ are defined recursively by equations \eqref{yuan}. 
\end{assumption}

Our main results are stated as follows.
\begin{theorem}\label{global}
	Let Assumption \ref{assu} be satisfied. Then, there exists a constant $\epsilon_0>0$ such that if
	\begin{align}
		E_0:=\SUM k 0 2 \|V_0^k\|_{H^{2-k}}<\epsilon_0,
	\end{align}
	the system  \eqref{yuan}-\eqref{boundary} has a unique global solution $(v,u,\theta,q,S)\in C^1([0,\infty)\times\Omega)$ with 
	\begin{align}
		(v-1,u,\theta-1,q,S) \in C([0,\infty),H^{2-\delta_0}(\Omega)) \cap C^1([0,\infty),H^{1-\delta_0}(\Omega)),
	\end{align}
	for any $\delta_0>0$ and satisfying 
	\begin{align}
		&\sup_{0\le t<\infty} \left(\SUM k 0 1 \left\|\partial_t^k\left(v-1, u, \theta-1, \sqrt{\tau   }q, \sqrt{\tau  }S\right)(t, \cdot)\right\|_{H^{2-k}}^{2}+\tau^2 \left\| \partial_t^2 \left(v, u, \theta, \sqrt{\tau} q, \sqrt{\tau} S\right(t, \cdot))\right\|_{L^2}^2\right),\nonumber\\
		&+\int_{0}^{\infty}\left( \sum_{|\alpha|=1}^{2}\left\|D^{\alpha}(v,u,\theta)(t, \cdot)\right\|_{L^{2}}^{2}+ \SUM k 0 1 \left\|\partial_t^k(q,S)(t ,\cdot)\right\|_{H^{2-k}}^{2}+\tau^2 \|(q_{tt}, S_{tt})(t, \cdot)\|_{L^2}^2\right)\dif t\le CE_0,  \label{new-hu5-1}
	\end{align}
where  $C$ is a universal constant independent of $\tau  $.
\end{theorem}

\begin{remark}
The assumption \( \tau_1 = \tau_2 \) or \( \tau_1 = O(\tau_2) \) as \( \tau_2 \to 0 \) is necessary to obtain global solutions with estimates that are independent of \( \tau_1 \) and \( \tau_2 \). However, if uniform boundedness with respect to \( \tau_1 \) and \( \tau_2 \) is not required, then global solutions can be obtained for any fixed \( \tau_1 > 0 \) and \( \tau_2 > 0 \).
\end{remark}

\begin{remark}
Note that the norm $\|V_0^k\|_{H^{2-k}}=O(1)$ as $\tau\rightarrow 0$, due to the assumption of well-prepared initial data \eqref{new-hu4-1}. To see this, we calculate
\begin{align}
\|\sqrt{\tau} q_t (t=0, \cdot)\|_{H^1}=\left\| \frac{-\kappa (\theta_0)(\theta_0)_x-v_0 q_0}{\sqrt{\tau} g(\theta_0)}\right\|_{H^1}=O(1),\, \mathrm{as}\, \tau \rightarrow 0. 
\end{align}
and
\begin{align}
\|\tau^\frac{3}{2} q_{tt}(t=0, \cdot)\|_{L^2}=\left\| \sqrt{\tau} \left(\frac{-\kappa (\theta)\theta_x -vq }{g(\theta)}\right)_t(t=0, \cdot)\right\|_{L^2}=O(1), \, \mathrm{as}\, \tau \rightarrow 0,
\end{align}
since $\|\sqrt{\tau} q_t(t=0, \cdot)\|_{L^2}=\left\| \frac{-\kappa (\theta_0) (\theta_0)_x -v_0 q_0}{ \sqrt{\tau} g(\theta_0)} \right\|_{L^2}= O(1)$.  The other terms can be done in a similar way.
\end{remark}

Based on the uniform estimates of solutions, we have the following convergence theorem.
\begin{theorem}\label{weak convergence}
	(Global weak convergence). Let  $(v^\tau,u^\tau,\theta^\tau,q^\tau,S^\tau)$ be the global solutions obtained in Theorem \ref{global}, then there exists functions $(v^0,u^0,\theta^0)\in L^\infty(\mathbb{R}^+;H^2(\Omega))$ and $(q^0,S^0) \in L^2(\mathbb{R}^+;H^2(\Omega))$, such that, as $\tau   \rightarrow 0$ up to subsequence,
	\begin{align}
		(v^\tau,u^\tau,\theta^\tau)\rightharpoonup (v^0,u^0,\theta^0)\quad
		weakly - \ast\quad in \quad L^\infty(\mathbb{R}^+;H^2(\Omega)),\\
		(q^\tau,S^\tau)\rightharpoonup (q^0,S^0)\quad
		weakly \quad in \quad L^2(\mathbb{R}^+;H^2(\Omega)),
	\end{align}
	where $(v^0,u^0,\theta^0)$ is the classical solution to the one dimensional non-isentropic compressible Navier-Stokes equations \eqref{classical} satisfying the boundary condition \eqref{boundary}, with initial value $(v^0,u^0,\theta^0 (t=0, x))=(v_0, u_0, \theta_0)(x)$. Moreover,
	$$q^0=-\frac{\kappa(\theta^0)(\theta^0)_x}{v^0},\quad S^0=\mu\frac{(u^0)_x}{v^0}.$$
\end{theorem}

The paper is organized as follows. In Section 2 we show the local existence theorem of approximating system. The uniform a prior estimates is established in Section 3. In Section 4, we justify the limit and prove our main theorem.

\section{approximate system and local existence}
In this part, we construct an approximate system with non-characteristic boundary, for which a unique local defined smooth solutions is given. Without loss of generality, we take $\Omega=[0,1]$. 

Our approximate system is as follows:
\begin{align}\label{approximate}
\begin{cases}
v^\epsilon_t=u^\epsilon_x,\\
u^\epsilon_t+{p(v^\epsilon,\theta^\epsilon)}_x=S^\epsilon_x,\\
e_{\theta}\theta^\epsilon_t-\frac{2a(\theta^\epsilon)}{Z(\theta^\epsilon)}q^\epsilon\theta^\epsilon_x+\frac{R\theta^\epsilon u^\epsilon_x}{v^\epsilon}+q^\epsilon_x=\frac{2a(\theta^\epsilon)}{\tau   (\theta^\epsilon)}(q^\epsilon)^2v^\epsilon+\frac{(S^\epsilon)^2v^\epsilon}{\mu},\\
\tau   (\theta^\epsilon)q^\epsilon_t+v^\epsilon q^\epsilon+\kappa(\theta^\epsilon)\theta^\epsilon_x=0,\\
\tau  (S^\epsilon_t+\epsilon b(x)S^\epsilon_x)+v^\epsilon S^\epsilon=\mu u^\epsilon_x,
\end{cases}
\end{align}	
where $\tau(\theta^\epsilon)=\tau g(\theta^\epsilon)$,  $b(x)=2x-1$  and initial and boundary conditions are given by 
\begin{align}\label{ainitial}
(v^\epsilon,u^\epsilon,\theta^\epsilon,q^\epsilon,S^\epsilon)(0,x)=(v_0,u_0,\theta_0,q_0,S_0)(x),
\end{align}
and
\begin{align}\label{aboundary}
u^\epsilon(t,0)=u^\epsilon(t,1)=q^\epsilon(t,0)=q^\epsilon(t,1)=0.
\end{align}

Let $U^\epsilon=(v^\epsilon,u^\epsilon,\theta^\epsilon,q^\epsilon,S^\epsilon)$, then we have
$$A^0(U^\epsilon)U^\epsilon_t+A^1(U^\epsilon)U^\epsilon_x+B(U^\epsilon)U^\epsilon=F^\epsilon(t,x),$$
where
\begin{align*}
&A^0(U^\epsilon)=\mathrm{diag}\left\{\frac{R\theta^\epsilon}{(v^\epsilon)^2},1,\frac{e_\theta(\theta^\epsilon,q^\epsilon)}{\theta^\epsilon},\frac{Z(\theta^\epsilon)}{\theta^\epsilon},\frac{\tau  }{\mu}\right\},\\
&A^1(U^\epsilon)=
\begin{pmatrix}
0& -\frac{R\theta^\epsilon}{(v^\epsilon)^2}& 0& 0& 0\\
-\frac{R\theta^\epsilon}{(v^\epsilon)^2}& 0& \frac{R}{v^\epsilon}& 0& -1\\
0& \frac{R}{v^\epsilon}& -\frac{2a(\theta^\epsilon)q^\epsilon}{\theta^\epsilon Z(\theta^\epsilon)}& \frac{1}{\theta^\epsilon}& 0\\
0& 0& \frac{1}{\theta^\epsilon}& 0& 0\\
0& -1& 0& 0& \frac{\tau  \epsilon}{\mu}b(x)\\
\end{pmatrix},\\
&B(U^\epsilon) =\mathrm{diag}\left\{0, 0, 0,  \frac{v^\epsilon}{\theta^\epsilon\kappa(\theta^\epsilon)},  \frac{v^\epsilon}{\mu}\right\},
F^\epsilon(t,x)=\mathrm{diag}\left\{0,0,-\frac{1}{\theta^\epsilon}\left(\frac{2a(\theta^\epsilon)(q^\epsilon)^2}{\tau   (\theta^\epsilon)}+\frac{(S^\epsilon)^2}{\mu}\right), 0,  0\right\}.
\end{align*}
with initial condition $U^\epsilon(0,x)=U_0:=(v_0,u_0,\theta_0,q_0,S_0)(x)$ and boundary condition
\begin{align*}
MU^\epsilon\big|_{\partial\Omega}=0,\mathrm{with}\
M=
\begin{pmatrix}
0& 0& 0& 0& 0\\
0& 1& 0& 0& 0\\
0& 0& 0& 0& 0\\
0& 0& 0& 1& 0\\
0& 0& 0& 0& 0\\
\end{pmatrix}.
\end{align*}
Note that $$\det (A^0)^{-1}A^1(U^\epsilon)\big|_{\partial\Omega}=\epsilon Rb(x)\frac{\theta^\epsilon}{(v^\epsilon)^2Z(\theta^\epsilon)e_\theta(\theta^\epsilon,q^\epsilon)}|_{\partial\Omega} \neq 0$$ for any $U\in G:=\{\mathbb{R}^+,\mathbb{R},\mathbb{R}^+,\mathbb{R},\mathbb{R}\}$. So, the boundary condition $\eqref{aboundary}$ is a non-characteristic boundary for any $\epsilon>0$.

Now, we show the boundary condition is maximally nonnegative, i.e., the matrix $A^1(U^\epsilon)\cdot \nu|_{\partial\Omega}$ is positive semidefinite on the null space $N$ of $M$ but not on any space containing $N$. Here $\nu(0)=-1, \nu(1)=1$. Let $\xi=(\xi_1,0,\xi_2,0,\xi_3)^T \in \mathrm{ker} M = \mathrm{span} \{(1,0,0,0,0)^T,(0,0,1,0,0)^T,(0,0,0,0,1)^T\}$, then,
\begin{align*}
\xi^TA^1\cdot\nu|_{\partial\Omega}\xi=\left(-\frac{2a(\theta^\epsilon)q^\epsilon}{\theta^\epsilon Z(\theta^\epsilon)}(\xi_2)^2+\frac{\tau  }{\mu}\epsilon b(x)(\xi_3)^2\right)\cdot\nu|_{\partial\Omega}
=\frac{\tau  \epsilon}{\mu}(\xi_3)^2
\ge 0.
\end{align*}
On the other hand, there are exactly four spaces containing ker $M$ as a proper subspace, they are $$\mathrm{span}\{(1,0,0,0,0)^T,(0,0,1,0,0)^T,(0,0,0,0,1)^T,(0,1,0,0,0)^T\},$$$$\mathrm{span}\{(1,0,0,0,0)^T,(0,0,1,0,0)^T,(0,0,0,0,1)^T,(0,0,0,1,0)^T\},$$$$\mathrm{span}\{(1,0,0,0,0)^T,(0,0,1,0,0)^T,(0,0,0,0,1)^T,(0,1,0,1,0)^T\} \quad\mathrm{and}\quad\mathbb{R}^5. $$
So, we take $\psi_1=(1,1,0,0,0)^T$, $\psi_2=(0,0,1,-1,0)^T$ and $\psi_3=(1,1,0,1,0)^T$ respectively,  and calculate
\begin{align*}
&\psi_1^TA^1\cdot\nu|_{\partial\Omega}\psi_1=-\frac{2R\theta^\epsilon}{(v^\epsilon)^2}<0,\\
&\psi_2^TA^1\cdot \nu \psi_2=\left(-\frac{2a(\theta)q}{\theta Z(\theta)}-\frac{2}{\theta}\right) \cdot \nu 
=-\frac{2}{\theta} \cdot \nu<0, \, \mathrm{for}\, x=1\\
&\psi_3^TA^1\cdot\nu|_{\partial\Omega}\psi_3=-\frac{2R\theta^\epsilon}{(v^\epsilon)^2}<0.
\end{align*}
Thus, the maximally nonnegative property is satisfied. Therefore, classical results implies local well-posedness theory, see \cite{SCH86}.
\begin{lemma}
Suppose $(v_0,u_0,\theta_0,q_0,S_0)\in H^2$ satisfying the compatibility condition $\eqref{compatibility condition}$ and
\begin{align}\label{new-hu-1}
\min_{x \in [0,1]}v_0(x),\min_{x \in [0,1]}\theta_0(x)>0, \min_{x\in[0,1]} (C_v+a^\prime(\theta_0) q_0^2) >0.
\end{align}
Then there exists a unique local solution $(v^\epsilon,u^\epsilon,\theta^\epsilon,q^\epsilon,S^\epsilon)$ to initial boundary value problem $\eqref{approximate}-\eqref{aboundary}$ on some time interval $[0,T]$ with
\begin{align*}
(v^\epsilon,u^\epsilon,\theta^\epsilon,q^\epsilon,S^\epsilon)\in C^0([0,T],H^2)\cap C^1([0,T],H^1),\\
\min_{x \in [0,1]}v(t,x)>0,\min_{x \in [0,1]}\theta(t,x)>0, \min_{x\in[0,1]} (C_v+a^\prime(\theta(t,x)) q^2(t,x)) >0, \quad \forall t>0. 
\end{align*}
\end{lemma}
 \begin{remark}
 Note that the initial assumption \eqref{new-hu-1} is to ensure the hyperbolicity of system \eqref{approximate}. In fact,
	\begin{align*}
		e_\theta&=C_v+a^\prime(\theta)q^2
		=C_v+\left(\frac{\tau   h^\prime(\theta)}{\theta}-\frac{\tau   h(\theta)}{\theta^2}-\frac{\tau   h^{\prime\prime}(\theta)}{2}\right)q^2\\
		&=C_v+\tau   q^2\left(\frac{h^\prime(\theta)}{\theta}-\frac{h(\theta)}{\theta^2}-\frac{h^{\prime\prime}(\theta)}{2}\right)
		= C_v+\tau   q^2l(\theta),
	\end{align*}
	where $l(\theta)=\frac{h^\prime(\theta)}{\theta}-\frac{h(\theta)}{\theta^2}-\frac{h^{\prime\prime}(\theta)}{2}$.  Therefore, for sufficiently small of $\theta-1$ and $\sqrt{\tau   } q$ (the situation for global solutions), it can be deduced that $e_\theta\ge \frac{C_v}{2}>0$.
\end{remark}

\section{uniform a prior estimates}

In this section, we establish uniform a priori estimates with respect to \( \tau   \) and \( \epsilon \), which enable us to obtain uniform global solutions. By taking the limit  \( \epsilon \to 0 \), we ultimately obtain the global well-posedness of smooth solutions for the original system \eqref{yuan}-\eqref{boundary}. 

For simplicity, we still denote $(v^\epsilon,u^\epsilon,\theta^\epsilon,q^\epsilon,S^\epsilon)$ by $(v,u,\theta,q,S)$ without of confusion.
First, define the weighed  energy
$$
E(t):=\sup_{0\leq s \leq t} \left(\SUM k 0 1 \left\|\partial_t^k(v-1, u, \theta-1, \sqrt{\tau   }q, \sqrt{\tau  }S) (s, \cdot)\right\|_{H^{2-k}}^{2}+\tau^2 \| \partial_t^2 (v, u, \theta, \sqrt{\tau} q, \sqrt{\tau} S)(s, \cdot)\|_{L^2}^2\right)
$$
and dissipation
$$
\mathcal{D}(t):=\sum_{|\alpha|=1}^{2}\left\|D^{\alpha}(v,u,\theta) (t, \cdot)\right\|_{L^{2}}^{2}+ \SUM k 0 1 \left\|\partial_t^k(q,S)(t, \cdot)\right\|_{H^{2-k}}^{2}+\tau^2 \|(q_{tt}, S_{tt})(t, \cdot)\|_{L^2}^2.
$$
We are aiming to show the following a prior estimates.

\begin{proposition}\label{zong}
Let  Assumption \ref{assu} hold and $(v,u,\theta,q,S)\in C^0([0,T],H^2)$ be local solutions to system $\eqref{approximate}-\eqref{aboundary}$. Then, there exists a small $\delta$ such that if  $E(t)\le \delta$,  we have
$$
E(t)+\int_{0}^{t}\mathcal{D}(s)\dif s\le C\left(E_0+E^\frac{1}{2} \int_0^t \mathcal D(s)\dif s+E^\frac{3}{2}(t)\right),
$$
where $C$ is a constant independent of $\tau   ,\tau  $ and $\epsilon$.
\end{proposition}

Without loss of generality, we assume $g(1)=\kappa(1)=1$, $\tau \le 1$ and $\epsilon<<1$.  Moreover, we assume $\delta$ is small enough such that 
\begin{align}
&\frac{3}{4}\le v(t,x),\theta(t,x)\le\frac{5}{4}, \label{new-hu2-1}\\
&2C_v\ge e_{\theta}\ge \frac{C_v}{2},\, 2\tau \ge Z(\theta)=\frac{\tau   (\theta)}{\kappa(\theta)}\ge \frac{\tau   }{2}, |e_{\theta\theta}|+|e_{\theta q}|+|e_{\theta \theta\theta}|+
|e_{\theta \theta q}|+|e_{\theta q q}| \le C.  \label{new-hu4-3}\\
&a(\theta)+\left(\frac{Z(\theta)}{2\theta}\right)^\prime=\frac{1}{\theta}(1-\frac{1}{2\theta}) Z(\theta)+(\frac{1}{\theta}-1) Z^\prime(\theta) \ge \frac{1}{4} \tau    \label{new-hu-2}\\
&|q(t,x)|_{L^\infty}=\left| \frac{\tau   (\theta) q_t+\kappa (\theta) \theta_x}{v}\right|_{L^\infty}\le C E^\frac{1}{2}(t),\,  |S(t,x)|_{L^\infty}= \left| \frac{ -\tau S_t+\mu u_x}{v}\right|_{L^\infty} \le CE^\frac{1}{2}(t). \label{new-hu4-4}
\end{align}
$ C$ denotes a universal constant which is independent of $\tau  $ and $\epsilon$.

First, we have the following $L^2$-estimates of solutions.
\begin{lemma}\label{0j}
There exists some constant $C$ such that
\begin{align}	\label{0jgjx}
\int_{0}^{1}\left((v-1)^2+u^2+ (\theta-1)^2+\tau   q^2+\tau  S^2\right)\dif x+\int_{0}^{t}\int_{0}^{1}(q^2+S^2)\dif x\dif t
\le CE_0.
\end{align}
\end{lemma}
\begin{proof}
From $\eqref{approximate}_3$, we get $$e_t+pu_x+q_x=\frac{v}{\mu}S^2.$$
Dividing the above equation by $\theta$, and using formula \eqref{1.4}, one has
$$\frac{1}{\theta}(C_v\theta+a(\theta)q^2)_t+\frac{R}{v}u_x+\frac{q_x}{\theta}=\frac{v}{\mu\theta}S^2.$$
First, we have
\begin{align*}
\frac{1}{\theta}C_v\theta_t=C_v(\ln\theta)_t,
\end{align*}
and
\begin{align*}
\frac{1}{\theta}(a(\theta)q^2)_t
&=(\frac{a(\theta)q^2}{\theta})_t+\frac{a(\theta)q^2}{\theta^2}\theta_t
=(\frac{a(\theta)q^2}{\theta})_t-(\frac{Z(\theta)}{2\theta^2}q^2)_t+\frac{Z(\theta)}{\theta^2}qq_t,
\end{align*}
where we have used the identity 
$$
(\frac{Z(\theta)}{2\theta^2})_t=-\frac{a(\theta)}{\theta^2}\theta_t.
$$
Notice that $\frac{a(\theta)}{\theta}-\frac{Z(\theta)}{2\theta^2}=-\left(\frac{Z(\theta)}{2\theta}\right)^\prime,$ and use the equation $\eqref{approximate}_4$, we have
\begin{align*}
\frac{1}{\theta}(a(\theta)q^2)_t
&=\left(-\left(\frac{Z(\theta)}{2\theta}\right)^\prime q^2\right)_t+\frac{Z(\theta)}{\theta^2}qq_t\\
&=\left(-\left(\frac{Z(\theta)}{2\theta}\right)^\prime q^2\right)_t+\frac{Z(\theta)}{\theta^2\tau   (\theta)}q(-vq-\kappa(\theta)\theta_x)\\
&=\left(-\left(\frac{Z(\theta)}{2\theta}\right)^\prime q^2\right)_t-\frac{vq^2}{\theta^2\kappa(\theta)}-\frac{\theta_x}{\theta^2}q,
\end{align*}
and
\begin{align*}
\frac{R}{v}u_x=\frac{R}{v}v_t=R(\ln v)_t.
\end{align*}
Combining the above estimates, we have
$$C_v(\ln \theta)_t+R(\ln v)_t+\left(-\left(\frac{Z(\theta)}{2\theta}\right)^\prime q^2\right)_t+\left(\frac{q}{\theta}\right)_x=\frac{v}{\theta^2\kappa(\theta)}q^2+\frac{v}{\mu\theta}S^2.$$
Combining the above equation with $\eqref{approximate}$, we derive that
\begin{align*}
\left[C_v(\theta-\ln\theta-1)+R(v-\ln v-1)+\left(a(\theta)+\frac{1}{2}\left(\frac{Z(\theta)}{\theta}\right)^\prime\right)q^2+\frac{\tau  }{2\mu}S^2+\frac{1}{2}u^2\right]_t\\
+\left[-Ru+pu+q-Su-\frac{q}{\theta}\right]_x+\frac{\tau  \epsilon}{\mu}b(x)SS_x+\frac{v}{\theta^2\kappa(\theta)}q^2+\frac{v}{\mu\theta}S^2=0.
\end{align*}
Notice that
\begin{align*}
\frac{\tau  \epsilon}{\mu}b(x)SS_x=\left(\frac{\tau  \epsilon}{2\mu}b(x)S^2\right)_x-\frac{\tau  \epsilon}{\mu}S^2,
\end{align*}
thus, choosing $\epsilon$ small such that $\epsilon\le \frac{1}{4}$, we have
\begin{align}\label{0jyuan}
\left[C_v(\theta-\ln\theta-1)+R(v-\ln v-1)+\left(a(\theta)+\frac{1}{2}\left(\frac{Z(\theta)}{\theta}\right)^\prime\right)q^2+\frac{\tau  }{2\mu}S^2+\frac{1}{2}u^2\right]_t\nonumber\\
+\left[-Ru+pu+q-Su-\frac{q}{\theta}+\frac{\tau  \epsilon}{2\mu}b(x)S^2\right]_x+\frac{vq^2}{\theta^2\kappa(\theta)}+\frac{1}{10\mu}S^2
\le0.
\end{align}
Moreover, using Taylor expansions, we get
$$\theta-\ln \theta-1=\frac{1}{2\zeta^2}(\theta-1)^2,$$
$$v-\ln v-1=\frac{1}{2\iota^2}(v-1)^2,$$
where $\zeta\in (1,\theta),\iota\in (1,v)$.

Therefore, integrating $\eqref{0jyuan}$ with respect to $t$ and $x$, using the boundary condition \eqref{aboundary}, the inequalities \eqref{new-hu2-1} and \eqref{new-hu-2}, and noting  the fact 
$$
\int_{0}^{t}\frac{\tau  \epsilon}{2\mu}b(x)S^2\Big|_0^1\dif t\ge0,
$$
we get the $L^2$ estimates \eqref{0jgjx} immediately.

\end{proof}

Before we do higher-order estimates which require the independence of the parameters $\tau  $, the following facts should be noted. That is, the quantities 
\begin{align*}
&\frac{a(\theta)}{Z(\theta)}=\frac{1}{\theta}-\frac{Z^\prime(\theta)}{2Z(\theta)}=\frac{1}{\theta}-\frac{h^\prime(\theta)}{2h(\theta)},\\
&\frac{a^\prime(\theta)}{Z(\theta)}=-\frac{1}{\theta^2}+\frac{Z^\prime(\theta)}{\theta Z(\theta)}-\frac{Z^{\prime\prime}(\theta)}{2Z(\theta)}
=-\frac{1}{\theta^2}+\frac{h^\prime(\theta)}{\theta h(\theta)}-\frac{h^{\prime\prime}(\theta)}{2h(\theta)},\\
&\frac{a^{\prime\prime}(\theta)}{Z(\theta)}
=\frac{2}{\theta^3}-\frac{2Z^\prime(\theta)}{\theta^2 Z(\theta)}+\frac{Z^{\prime\prime}(\theta)}{\theta Z(\theta)}-\frac{Z^{\prime\prime\prime}(\theta)}{2 Z(\theta)}
=\frac{2}{\theta^3}-\frac{2h^\prime(\theta)}{\theta^2 h(\theta)}+\frac{h^{\prime\prime}(\theta)}{\theta h(\theta)}-\frac{h^{\prime\prime\prime}(\theta)}{2 h(\theta)}
\end{align*}
are  independent of $\tau   $, where  $Z(\theta)=\frac{\tau    g(\theta)}{\kappa(\theta)}=\tau   h(\theta)$ and  $h(\theta)$ is independent of $\tau   $.

Now, we do the first-order estimates.
\begin{lemma}\label{yijiex}
There exists some constant $C$ such that
\begin{align}
\int_{0}^{1}(u^2_x+v^2_x+ \theta^2_x+\tau   q^2_x+\tau  S^2_x)\dif x+\int_{0}^{t}\int_{0}^{1}(q^2_x+S^2_x)\dif x\dif t
\le C\left(E_0+E^\frac{1}{2}(t)\int_{0}^{t}\mathcal{D}(s)\dif s\right).
\label{yjgjx}
\end{align}
\end{lemma}
	
\begin{proof}
Taking derivatives with respect to $x$ to the equations \eqref{approximate}, we get
\begin{align}\label{x}
\begin{cases}
v_{tx}=u_{xx},\\
u_{tx}+{p(v,\theta)}_{xx}=S_{xx},\\
(e_{\theta}\theta_t)_x-(\frac{2a(\theta)}{Z(\theta)}q\theta_x)_x+(\frac{R\theta u_x}{v})_x+q_{xx}=(\frac{2a(\theta)}{\tau   (\theta)}q^2v)_x+(\frac{S^2v}{\mu})_x,\\
\left(\tau   (\theta)q_t\right)_x+(vq)_x+(\kappa(\theta)\theta_x)_x=0,\\
\tau  S_{tx}+\tau  \epsilon(b(x)S_x)_x+(vS)_x=\mu u_{xx}.
\end{cases}
\end{align} 

Multiplying the equation $\eqref{x}_2$ by $\theta u_x$ and integrating over $[0,1]$ with respect to $x$, we get
$$\int_{0}^{1} \theta u_xu_{tx}\dif x+\int_{0}^{1} \theta \left(p(v,\theta)_x-S_x\right)_xu_x\dif x=0,$$
where
\begin{align*}
\int_{0}^{1}\theta u_xu_{tx}\dif x=\frac{1}{2}\frac{\dif}{\dif t}\int_{0}^{1} \theta u^2_x\dif x-\frac{1}{2}\int_{0}^{1} \theta_t u^2_x\dif x \ge \frac{1}{2}\frac{\dif}{\dif t}\int_{0}^{1} \theta u^2_x\dif x-CE^{\frac{1}{2}}(t)\mathcal{D}(t).
\end{align*}
Note that $u_t(t,0)=u_t(t,1)=0$, by the momentum equation, we get 
$$
\left(p(v,\theta)_x-S_x\right)|_{\partial \Omega}=0.
$$
Therefore, we have
\begin{align*}
&\int_{0}^{1} \theta \left(p(v,\theta)_x-S_x\right)_xu_x\dif x\\
&=\theta \left(p(v,\theta)_x-S_x\right)u_x\big|_{0}^{1}-\int_{0}^{1} \theta \left(p(v,\theta)_x-S_x\right)u_{xx}\dif x-\int_{0}^{1}\theta_x\left(p_vv_x+p_\theta\theta_x-S_x\right)u_x\dif x\\
&\ge-\int_{0}^{1} \theta p(v,\theta)_xu_{xx}\dif x+\int_{0}^{1} \theta S_xu_{xx}\dif x-CE^{\frac{1}{2}}(t)\mathcal{D}(t)\\
&=-\int_{0}^{1} \theta \left(\frac{R\theta_x}{v}-\frac{R\theta}{v^2}v_x\right) u_{xx}\dif x+\int_{0}^{1} \theta S_xu_{xx}\dif x-CE^{\frac{1}{2}}(t)\mathcal{D}(t)\\
&=-\int_{0}^{1}\frac{R\theta}{v}\theta_xu_{xx}\dif x+\int_{0}^{1}\frac{R\theta^2}{v^2}v_xv_{tx}\dif x+\int_{0}^{1}\theta S_xu_{xx}\dif x-CE^{\frac{1}{2}}(t)\mathcal{D}(t)\\
&=-\int_{0}^{1}\frac{R\theta}{v}\theta_xu_{xx}\dif x+\int_{0}^{1}\theta S_xu_{xx}\dif x+\frac{1}{2}\frac{d}{\dif t}\int_{0}^{1} \frac{R\theta^2}{v^2}v^2_x\dif x-\frac{1}{2}\int_{0}^{1}\left(\frac{R\theta^2}{v^2}\right)_tv^2_x\dif x-CE^{\frac{1}{2}}(t)\mathcal{D}(t)\\
&\ge-\int_{0}^{1}\frac{R\theta}{v}\theta_xu_{xx}\dif x+\int_{0}^{1}\theta S_xu_{xx}\dif x+\frac{1}{2}\frac{\dif}{\dif t}\int_{0}^{1} \frac{R\theta^2}{v^2}v^2_x\dif x-CE^{\frac{1}{2}}(t)\mathcal{D}(t).
\end{align*}
Combining the above estimates, we get
\begin{align}\label{u_{x1}}
\frac{1}{2}\frac{\dif}{\dif t}\int_{0}^{1} (\theta u^2_x+\frac{R\theta^2}{v^2}v^2_x)\dif x-\int_{0}^{1}\frac{R\theta}{v}\theta_xu_{xx}\dif x+\int_{0}^{1}\theta S_xu_{xx}\dif x\le CE^{\frac{1}{2}}(t)\mathcal{D}(t).
\end{align}
Multiplying the equation $\eqref{x}_3$ by $\theta_x$ and integrating over $[0,1]$ with respect to $x$, we get
\begin{align}
&\int_{0}^{1}(e_{\theta}\theta_t)_x\theta_x\dif x-\int_{0}^{1}\left(\frac{2a(\theta)}{Z(\theta)}q\theta_x\right)_x\theta_x\dif x+\int_{0}^{1}\left(\frac{R\theta u_x}{v}\right)_x\theta_x\dif x+\int_{0}^{1}q_{xx}\theta_x\dif x \nonumber\\
&=\int_{0}^{1}\left(\frac{2a(\theta)}{\tau   (\theta)}q^2v\right)_x\theta_x\dif x+\int_{0}^{1}\left(\frac{S^2v}{\mu}\right)_x\theta_x\dif x. \label{new-hu2-2}
\end{align}
We estimate each term of the above equation separately. For the left-hand of the above equation, we have
\begin{align*}
&\int_{0}^{1}(e_{\theta}\theta_t)_x\theta_x\dif x\\
&=\int_{0}^{1}e_{\theta}\theta_{tx}\theta_x\dif x+\int_{0}^{1}\left(e_{\theta \theta}\theta_x+e_{\theta q}q_x\right)\theta_t\theta_x\dif x\\
&=\frac{1}{2}\frac{d}{\dif t}\int_{0}^{1}e_\theta \theta^2_x\dif x-\frac{1}{2}\int_{0}^{1}\left(e_{\theta \theta}\theta_t+e_{\theta q}q_t\right)\theta^2_x\dif x+\int_{0}^{1}\left(a^{\prime \prime}(\theta)q^2\theta_x+2a^\prime(\theta)qq_x\right)\theta_t\theta_x\dif x\\
&\ge\frac{1}{2}\frac{d}{\dif t}\int_{0}^{1}e_\theta \theta^2_x\dif x-CE^{\frac{1}{2}}(t)\mathcal{D}(t),
\end{align*}
and
\begin{align*}
-\int_{0}^{1}\left(\frac{2a(\theta)}{Z(\theta)}q\theta_x\right)_x\theta_x\dif x&=-\int_{0}^{1}\frac{2a(\theta)}{Z(\theta)}q\theta_{xx}\theta_x\dif x-\int_{0}^{1}\left(\frac{2a(\theta)}{Z(\theta)}q\right)_x\theta^2_x\dif x\\
&=-\frac{a(\theta)}{Z(\theta)}q\theta^2_x\big|_{0}^{1}-\int_{0}^{1}\left(\frac{a(\theta)}{Z(\theta)}q\right)_x\theta^2_x\dif x\\
&=-\int_{0}^{1}\left(\frac{a(\theta)}{Z(\theta)}q_x+\frac{a^\prime(\theta)}{Z(\theta)}\theta_xq-\frac{Z^\prime(\theta)a(\theta)}{Z^2(\theta)}\theta_xq\right)\theta^2_x\dif x\\
&\ge-CE^{\frac{1}{2}}(t)\mathcal{D}(t).
\end{align*}
Similarly, one has
\begin{align*}
\int_{0}^{1}\left(\frac{R\theta u_x}{v}\right)_x\theta_x\dif x&=\int_{0}^{1}\frac{R\theta}{v}u_{xx}\theta_x\dif x+\int_{0}^{1}\left(\frac{R\theta}{v}\right)_xu_x\theta_x\dif x\\
&=\int_{0}^{1}\frac{R\theta}{v}u_{xx}\theta_x\dif x+\int_{0}^{1}\left(\frac{R\theta_x}{v}-\frac{R\theta v_x}{v^2}\right)u_x\theta_x\dif x\\
&\ge\int_{0}^{1}\frac{R\theta}{v}u_{xx}\theta_x\dif x-CE^{\frac{1}{2}}(t)\mathcal{D}(t).
\end{align*}

For the right-hand side of the equation \eqref{new-hu2-2}, using \eqref{new-hu4-4}, it follows that
\begin{align*}
&\int_{0}^{1}\left(\frac{2a(\theta)}{\tau   (\theta)}q^2v+\frac{S^2v}{\mu}\right)_x\theta_x\dif x\\
&=\int_{0}^{1}\left\{\frac{4a(\theta)}{\tau   (\theta)}qq_xv+\frac{2a(\theta)}{\tau   (\theta)}q^2v_x+\frac{2a^\prime(\theta)}{\tau   (\theta)}\theta_xq^2v-\frac{\tau   ^\prime(\theta)2a(\theta)}{\tau   ^2(\theta)}\theta_xq^2v+\frac{2SS_xv}{\mu}+\frac{S^2v_x}{\mu}\right\}\theta_x\dif x\\
&\le CE^{\frac{1}{2}}(t)\mathcal{D}(t).
\end{align*}	
Therefore, we get
\begin{align}\label{u_{x2}}
\frac{1}{2}\frac{d}{\dif t}\int_{0}^{1}e_\theta \theta^2_x\dif x+\int_{0}^{1}\frac{R\theta}{v}u_{xx}\theta_x\dif x+\int_{0}^{1}q_{xx}\theta_x\dif x\le CE^{\frac{1}{2}}(t)\mathcal{D}(t).
\end{align}
Multiplying the equation $\eqref{x}_4$ by $\frac{q_x}{\kappa(\theta)}$ and integrating over $[0,1]$ with respect to $x$, we get
$$\int_{0}^{1}\frac{1}{\kappa(\theta)}\left(\tau   (\theta)q_t\right)_xq_x\dif x+\int_{0}^{1}\frac{1}{\kappa(\theta)}\left(vq+\kappa(\theta)\theta_x\right)_xq_x\dif x=0.$$
For the first term of the above equation, one has
\begin{align*}
&\int_{0}^{1}\frac{1}{\kappa(\theta)}\left(\tau   (\theta)q_t\right)_xq_x\dif x\\
&=\int_{0}^{1}Z(\theta)q_{tx}q_x\dif x+\int_{0}^{1}\frac{\tau^\prime (\theta)}{\kappa(\theta)}\theta_xq_tq_x\dif x\\
&\ge\frac{1}{2}\frac{d}{\dif t}\int_{0}^{1}Z(\theta)q^2_x\dif x-\frac{1}{2}\int_{0}^{1}Z^\prime(\theta)\theta_tq^2_x\dif x-CE^{\frac{1}{2}}(t)\mathcal{D}(t)\\
&\ge\frac{1}{2}\frac{d}{\dif t}\int_{0}^{1}Z(\theta)q^2_x\dif x-CE^{\frac{1}{2}}(t)\mathcal{D}(t).
\end{align*}
Note that $q_t(t,0)=q_t(t,1)=0$, we get 
$$vq+\kappa(\theta)\theta_x\Big|_{\partial \Omega}=0.
$$ 
Thus, it yields
\begin{align*}
&\int_{0}^{1}\frac{1}{\kappa(\theta)}\left(vq+\kappa(\theta)\theta_x\right)_xq_x\dif x\\
&=\frac{1}{\kappa(\theta)}(vq+\kappa(\theta)\theta_x)q_x\big|_{0}^{1}-\int_{0}^{1}\frac{1}{\kappa(\theta)}\left(vq+\kappa(\theta)\theta_x\right)q_{xx}\dif x+\int_{0}^{1}\frac{\kappa^\prime(\theta)}{\kappa^2(\theta)}\left(vq+\kappa(\theta)\theta_x\right)\theta_xq_x\dif x\\
&\ge-\int_{0}^{1}\frac{1}{\kappa(\theta)}vqq_{xx}\dif x-\int_{0}^{1}\theta_xq_{xx}\dif x-CE^{\frac{1}{2}}(t)\mathcal{D}(t)\\
&=-\frac{1}{\kappa(\theta)}vqq_x\big|_{0}^{1}+\int_{0}^{1}\frac{1}{\kappa(\theta)}vq^2_x\dif x+\int_{0}^{1}\left(\frac{1}{\kappa(\theta)}v_x-\frac{\kappa^\prime(\theta)}{\kappa^2(\theta)}\theta_xv\right)qq_x\dif x-\int_{0}^{1}\theta_xq_{xx}\dif x-CE^{\frac{1}{2}}(t)\mathcal{D}(t)\\
&\ge\int_{0}^{1}\frac{1}{\kappa(\theta)}vq^2_x\dif x-\int_{0}^{1}\theta_xq_{xx}\dif x-CE^{\frac{1}{2}}(t)\mathcal{D}(t).
\end{align*}
Therefore, we derive that
\begin{align}\label{u_{x3}}
\frac{1}{2}\frac{d}{\dif t}\int_{0}^{1}Z(\theta)q^2_x\dif x+\int_{0}^{1}\frac{1}{\kappa(\theta)}vq^2_x\dif x-\int_{0}^{1}\theta_xq_{xx}\dif x\le CE^{\frac{1}{2}}(t)\mathcal{D}(t).
\end{align}
Multiplying the equation $\eqref{x}_5$ by $\frac{\theta}{\mu}S_x$ and integrating over $[0,1]$ with respect to $x$, we get
$$\int_{0}^{1}\frac{\tau  \theta}{\mu}S_{tx}S_x\dif x+\int_{0}^{1}\frac{\theta}{\mu}(vS)_xS_x\dif x-\int_{0}^{1}\theta u_{xx}S_x\dif x+\frac{\tau  \epsilon}{\mu}\int_{0}^{1}(b(x)S_x)_x\theta S_x\dif x
=0.$$
Note that
\begin{align*}
\int_{0}^{1}\frac{\tau  \theta}{\mu}S_{tx}S_x\dif x&=\frac{1}{2}\frac{\dif}{\dif t}\int_{0}^{1}\frac{\tau  }{\mu}\theta S^2_x\dif x-\frac{1}{2}\int_{0}^{1}\frac{\tau  }{\mu}\theta_tS^2_x\dif x\\
&\ge\frac{1}{2}\frac{\dif}{\dif t}\int_{0}^{1}\frac{\tau  }{\mu}\theta S^2_x\dif x-CE^{\frac{1}{2}}(t)\mathcal{D}(t),
\end{align*}
and
\begin{align*}
\int_{0}^{1}\frac{\theta}{\mu}(vS)_xS_x\dif x
=\int_{0}^{1}\frac{\theta}{\mu}(v_xS+vS_x)S_x\dif x
\ge \frac{1}{2\mu}\int_{0}^{1}S^2_x\dif x-CE^{\frac{1}{2}}(t)\mathcal{D}(t).
\end{align*}
Similarly, we have
\begin{align*}
\frac{\tau  \epsilon}{\mu}\int_{0}^{1}(b(x)S_x)_x\theta S_x\dif x
&=\frac{\tau  \epsilon}{\mu}\int_{0}^{1}(2S_x+b(x)S_{xx})\theta S_x\dif x\\
&=\frac{2\tau  \epsilon}{\mu}\int_{0}^{1}\theta S^2_x\dif x+\frac{\tau  \epsilon}{2\mu}b(x)\theta S^2_x|_{0}^{1}-\frac{\tau  \epsilon}{2\mu}\int_{0}^{1}(2\theta+b(x)\theta_x)S^2_x\dif x\\
&\ge \frac{\tau  \epsilon}{\mu}\int_{0}^{1}\theta S^2_x\dif-CE^{\frac{1}{2}}(t)\mathcal{D}(t).
\end{align*}
Therefore, we get
\begin{align}\label{u_{x4}}
\frac{1}{2}\frac{\dif}{\dif t}\int_{0}^{1}\frac{\tau  }{\mu}\theta S^2_x\dif x+\int_{0}^{1}\frac{1}{4\mu}S^2_x\dif x-\int_{0}^{1}\theta u_{xx}S_x\dif x\le CE^{\frac{1}{2}}(t)\mathcal{D}(t) .
\end{align}
		
Thus, combining the estimates \eqref{u_{x1}}-\eqref{u_{x4}}, we conclude that
\begin{align*}
\frac{\dif}{\dif t} \int_0^1  \left( \frac{\theta}{2} u_x^2+\frac{R\theta^2}{v^2} v_x^2+\frac{e_\theta}{2} \theta_x^2+\frac{Z(\theta)}{2} q_x^2 +\frac{\tau \theta}{2\mu}  S_x^2\right) \dif x 
+\int_0^1 \left(\frac{v}{\kappa (\theta)} q_x^2+\frac{1}{4\mu} S_x^2 \right)\dif x \le C E^\frac{1}{2} (t) \mathcal D (t).
\end{align*}
Integrating the above inequality,  and using the a priori assumption \eqref{new-hu2-1}-\eqref{new-hu4-3}, we get  \eqref{yjgjx} immediately.
\end{proof}

Using similar methods, we can get the first order estimates with respect to $t$ as follows.
\begin{lemma}\label{yijiet}
There exists some constant $C$ such that
\begin{align}\label{yjgjt}
\int_{0}^{1}(u^2_t+v^2_t+\theta^2_t+\tau   q^2_t+\tau  S^2_t)\dif x+\int_{0}^{t}\int_{0}^{1}(q^2_t+S^2_t)\dif x\dif t
\le C\left(E_0+E^\frac{1}{2}(t)\int_{0}^{t}\mathcal{D}(s)\dif s\right).
\end{align}
\begin{proof}
Taking derivatives with respect to $t$ to the equations \eqref{approximate}, one get
\begin{align}\label{t}
\begin{cases}
v_{tt}=u_{tx},\\
u_{tt}+{p(v,\theta)}_{tx}=S_{tx},\\
(e_{\theta}\theta_t)_t-(\frac{2a(\theta)}{Z(\theta)}q\theta_x)_t+(\frac{R\theta u_x}{v})_t+q_{tx}=(\frac{2a(\theta)}{\tau   (\theta)}q^2v)_t+(\frac{S^2v}{\mu})_t,\\
\left(\tau   (\theta)q_t\right)_t+(vq)_t+(\kappa(\theta)\theta_x)_t=0,\\
\tau  S_{tt}+(vS)_t+\tau  \epsilon b(x)S_{tx}=\mu u_{tx}.
\end{cases}
\end{align} 
Multiplying the equation $\eqref{t}_2$ by $\theta u_t$ and integrating over $[0,1]$ with respect to $x$, it follows
$$\int_{0}^{1} \theta u_tu_{tt}\dif x+\int_{0}^{1} \theta \left(p(v,\theta)-S\right)_{tx}u_t\dif x=0,$$
where
\begin{align*}
\int_{0}^{1}\theta u_tu_{tt}\dif x=\frac{1}{2}\frac{\dif}{\dif t}\int_{0}^{1} \theta u^2_t\dif x-\frac{1}{2}\int_{0}^{1} \theta_t u^2_t\dif x \ge \frac{1}{2}\frac{\dif}{\dif t}\int_{0}^{1} \theta u^2_t\dif x-CE^{\frac{1}{2}}(t)\mathcal{D}(t).
\end{align*}
In view of $u_t|_{\partial \Omega}=0$, it yields
\begin{align*}
&\int_{0}^{1} \theta \left(p(v,\theta)-S\right)_{tx}u_t\dif x\\
&=\theta \left(p(v,\theta)_t-S_t\right)u_t\big|_{0}^{1}-\int_{0}^{1} \theta \left(p(v,\theta)_t-S_t\right)u_{tx}\dif x-\int_{0}^{1}\theta_x\left(p_vv_t+p_\theta\theta_t-S_t\right)u_t\dif x\\
&\ge-\int_{0}^{1} \theta p(v,\theta)_tu_{tx}\dif x+\int_{0}^{1} \theta S_tu_{tx}\dif x-CE^{\frac{1}{2}}(t)\mathcal{D}(t)\\
&=-\int_{0}^{1} \theta \left(\frac{R\theta_t}{v}-\frac{R\theta}{v^2}v_t\right) u_{tx}\dif x+\int_{0}^{1} \theta S_tu_{tx}\dif x-CE^{\frac{1}{2}}(t)\mathcal{D}(t)\\
&=-\int_{0}^{1}\frac{R\theta}{v}\theta_tu_{tx}+\int_{0}^{1}\frac{R\theta^2}{v^2}v_tv_{tt}\dif x+\int_{0}^{1}\theta S_tu_{tx}\dif x-CE^{\frac{1}{2}}(t)\mathcal{D}(t)\\
&=-\int_{0}^{1}\frac{R\theta}{v}\theta_tu_{tx}\dif x+\int_{0}^{1}\theta S_tu_{tx}\dif x+\frac{1}{2}\frac{d}{\dif t}\int_{0}^{1} \frac{R\theta^2}{v^2}v^2_t\dif x-\frac{1}{2}\int_{0}^{1}\left(\frac{R\theta^2}{v^2}\right)_tv^2_t\dif x-CE^{\frac{1}{2}}(t)\mathcal{D}(t)\\
&\ge-\int_{0}^{1}\frac{R\theta}{v}\theta_tu_{tx}\dif x+\int_{0}^{1}\theta S_tu_{tx}\dif x+\frac{1}{2}\frac{\dif}{\dif t}\int_{0}^{1} \frac{R\theta^2}{v^2}v^2_t\dif x-CE^{\frac{1}{2}}(t)\mathcal{D}(t).
\end{align*}
Combining the above estimates, we get
\begin{align}\label{u_{t1}}
\frac{1}{2}\frac{d}{\dif t}\int_{0}^{1} (\theta u^2_t+\frac{R\theta^2}{v^2}v^2_t)\dif x-\int_{0}^{1}\frac{R\theta}{v}\theta_tu_{tx}\dif x+\int_{0}^{1}\theta S_tu_{tx}\dif x\le CE^{\frac{1}{2}}(t)\mathcal{D}(t).
\end{align}
Multiplying the equation $\eqref{t}_3$ by $\theta_t$ and integrating over $[0,1]$ with respect to $x$, it follows
\begin{align*}
&\int_{0}^{1}(e_{\theta}\theta_t)_t\theta_t\dif x-\int_{0}^{1}\left(\frac{2a(\theta)}{Z(\theta)}q\theta_x\right)_t\theta_t\dif x+\int_{0}^{1}\left(\frac{R\theta u_x}{v}\right)_t\theta_t\dif x+\int_{0}^{1}q_{tx}\theta_t\dif x\\
&=\int_{0}^{1}\left(\frac{2a(\theta)}{\tau   (\theta)}q^2v\right)_t\theta_t\dif x+\int_{0}^{1}\left(\frac{S^2v}{\mu}\right)_t\theta_t\dif x.
\end{align*}
We estimate each term of the above equation individually. For the left-hand of the above equation, one have
\begin{align*}
&\int_{0}^{1}(e_{\theta}\theta_t)_t\theta_t\dif x\\
&=\int_{0}^{1}e_{\theta}\theta_{tt}\theta_t\dif x+\int_{0}^{1}\left(e_{\theta \theta}\theta_t+e_{\theta q}q_t\right)\theta^2_t\dif x\\
&=\frac{1}{2}\frac{d}{\dif t}\int_{0}^{1}e_\theta \theta^2_t\dif x-\frac{1}{2}\int_{0}^{1}\left(a^{\prime \prime}(\theta)q^2\theta_t+2a^\prime(\theta)qq_t\right)\theta^2_t\dif x\\
&\ge\frac{1}{2}\frac{d}{\dif t}\int_{0}^{1}e_\theta \theta^2_t\dif x-CE^{\frac{1}{2}}(t)\mathcal{D}(t).
\end{align*}
Using the boundary condition $q|_{\partial \Omega}=0$, we get
\begin{align*}
&-\int_{0}^{1}\left(\frac{2a(\theta)}{Z(\theta)}q\theta_x\right)_t\theta_t\dif x\\
&=-\int_{0}^{1}\frac{2a(\theta)}{Z(\theta)}q\theta_{tx}\theta_t\dif x-\int_{0}^{1}\left(\frac{2a(\theta)}{Z(\theta)}q\right)_t\theta_x\theta_t\dif x\\
&=-\frac{a(\theta)}{Z(\theta)}q\theta^2_t\big|_{0}^{1}+\int_{0}^{1}\left(\frac{a(\theta)}{Z(\theta)}q\right)_x\theta^2_t\dif x-\int_{0}^{1}\left(\frac{2a(\theta)}{Z(\theta)}q_t+\frac{2a^\prime(\theta)}{Z(\theta)}\theta_tq-\frac{2Z^\prime(\theta)a(\theta)}{Z^2(\theta)}\theta_tq\right)\theta_x\theta_t\dif x\\
&\ge\int_{0}^{1}\left(\frac{a(\theta)}{Z(\theta)}q_x+\frac{a^\prime(\theta)}{Z(\theta)}\theta_xq-\frac{Z^\prime(\theta)a(\theta)}{Z^2(\theta)}\theta_xq\right)\theta^2_t\dif x-CE^{\frac{1}{2}}(t)\mathcal{D}(t)\\
&\ge-CE^{\frac{1}{2}}(t)\mathcal{D}(t).
\end{align*}
Similarly, one has
\begin{align*}
\int_{0}^{1}(\frac{R\theta u_x}{v})_t\theta_t\dif x&=\int_{0}^{1}\frac{R\theta}{v}u_{tx}\theta_t\dif x+\int_{0}^{1}\left(\frac{R\theta}{v}\right)_tu_x\theta_t\dif x\\
&=\int_{0}^{1}\frac{R\theta}{v}u_{tx}\theta_t\dif x+\int_{0}^{1}\left(\frac{R\theta_t}{v}-\frac{R\theta v_t}{v^2}\right)u_x\theta_t\dif x\\
&\ge\int_{0}^{1}\frac{R\theta}{v}u_{tx}\theta_t\dif x-CE^{\frac{1}{2}}(t)\mathcal{D}(t).
\end{align*}
Finally, we have
\begin{align*}
&\int_{0}^{1}\left(\frac{2a(\theta)}{\tau   (\theta)}q^2v+\frac{S^2v}{\mu}\right)_t\theta_t\dif x\\
&=\int_{0}^{1}\left\{\frac{4a(\theta)}{\tau   (\theta)}qq_tv+\frac{2a(\theta)}{\tau   (\theta)}q^2v_t+\frac{2a^\prime(\theta)}{\tau   (\theta)}\theta_tq^2v-\frac{\tau   ^\prime(\theta)2a(\theta)}{\tau   ^2(\theta)}\theta_tq^2v+\frac{2SS_tv}{\mu}+\frac{S^2v_t}{\mu}\right\}\theta_t\dif x\\
&\le CE^{\frac{1}{2}}(t)\mathcal{D}(t).
\end{align*}
Therefore, 
\begin{align}\label{u_{t2}}
\frac{1}{2}\frac{d}{\dif t}\int_{0}^{1}e_\theta \theta^2_t\dif x+\int_{0}^{1}\frac{R\theta}{v}u_{tx}\theta_t\dif x+\int_{0}^{1}q_{tx}\theta_t\dif x\le CE^{\frac{1}{2}}(t)\mathcal{D}(t).
\end{align}
Multiplying the equation $\eqref{t}_4$ by $\frac{q_t}{k(\theta)}$ and integrating over $[0,1]$ with respect to $x$, it follows
$$\int_{0}^{1}\frac{1}{k(\theta)}\left(\tau   (\theta)q_t\right)_tq_t\dif x+\int_{0}^{1}\frac{1}{k(\theta)}\left(vq+k(\theta)\theta_x\right)_tq_t\dif x=0.$$
Using similar methods as before, we have
\begin{align*}
&\int_{0}^{1}\frac{1}{\kappa(\theta)}\left(\tau   (\theta)q_t\right)_tq_t\dif x\\
&=\int_{0}^{1}Z(\theta)q_{tt}q_t\dif x+\int_{0}^{1}\frac{\tau^\prime(\theta)}{\kappa(\theta)}\theta_tq^2_t\dif x\\
&\ge\frac{1}{2}\frac{d}{\dif t}\int_{0}^{1}Z(\theta)q^2_t\dif x-\frac{1}{2}\int_{0}^{1}Z^\prime(\theta)\theta_tq^2_t\dif x-CE^{\frac{1}{2}}(t)\mathcal{D}(t)\\
&\ge\frac{1}{2}\frac{d}{\dif t}\int_{0}^{1}Z(\theta)q^2_t\dif x-CE^{\frac{1}{2}}(t)\mathcal{D}(t),
\end{align*}
and
\begin{align*}
\int_{0}^{1}\frac{1}{\kappa(\theta)}(vq)_tq_t\dif x
&=\int_{0}^{1}\frac{1}{\kappa(\theta)}\left(vq_t+v_tq\right)q_t\dif x
\ge\int_{0}^{1}\frac{1}{\kappa(\theta)}vq^2_t\dif x-CE^{\frac{1}{2}}(t)\mathcal{D}(t),
\end{align*}
and
\begin{align*}
\int_{0}^{1}\frac{1}{\kappa(\theta)}(\kappa(\theta)\theta_x)_tq_t\dif x
=\int_{0}^{1}\left(\frac{\kappa^\prime(\theta)}{\kappa(\theta)}\theta_t\theta_x+\theta_{tx}\right)q_t\dif x
&\ge\theta_tq_t\big|_{0}^{1}-\int_{0}^{1}\theta_tq_{tx}\dif x-CE^{\frac{1}{2}}(t)\mathcal{D}(t)\\
&=-\int_{0}^{1}\theta_tq_{tx}\dif x-CE^{\frac{1}{2}}(t)\mathcal{D}(t).
\end{align*}
Therefore,  it yields that
\begin{align}\label{u_{t3}}
\frac{1}{2}\frac{d}{\dif t}\int_{0}^{1}Z(\theta)q^2_t\dif x+\int_{0}^{1}\frac{1}{\kappa(\theta)}vq^2_t\dif x-\int_{0}^{1}\theta_tq_{tx}\dif x\le CE^{\frac{1}{2}}(t)\mathcal{D}(t).
\end{align}
Multiplying the equation $\eqref{t}_5$ by $\frac{\theta}{\mu}S_t$ and integrating over $[0,1]$ with respect to $x$, we get
$$\int_{0}^{1}\frac{\tau  \theta}{\mu}S_{tt}S_t\dif x+\int_{0}^{1}\frac{\theta}{\mu}(vS)_tS_t\dif x-\int_{0}^{1}\theta u_{tx}S_t\dif x+\int_{0}^{1}\frac{\tau  \epsilon}{\mu} b(x)S_{tx}\theta S_t\dif x
=0.$$
Note that
\begin{align*}
\int_{0}^{1}\frac{\tau  \theta}{\mu}S_{tt}S_t\dif x&=\frac{1}{2}\frac{\dif}{\dif t}\int_{0}^{1}\frac{\tau  }{\mu}\theta S^2_t\dif x-\frac{1}{2}\int_{0}^{1}\frac{\tau  }{\mu}\theta_tS^2_t\dif x\\
&\ge\frac{1}{2}\frac{\dif}{\dif t}\int_{0}^{1}\frac{\tau  }{\mu}\theta S^2_t\dif x-CE^{\frac{1}{2}}(t)\mathcal{D}(t),
\end{align*}
and
\begin{align*}
\int_{0}^{1}\frac{\theta}{\mu}(vS)_tS_t\dif x=\int_{0}^{1}\frac{\theta}{\mu}(v_tS+vS_t)S_t\dif x\ge\int_{0}^{1}\frac{1}{2\mu}S^2_t\dif x-CE^{\frac{1}{2}}(t)\mathcal{D}(t).
\end{align*}
Choosing $\epsilon$ sufficiently small, one can get
\begin{align*}
\int_{0}^{1}\frac{\tau  \epsilon}{\mu} b(x)S_{tx}\theta S_t\dif x
=\frac{\tau  \epsilon}{2\mu}b(x)\theta S^2_t|_0^{1}-\frac{\tau  \epsilon}{2\mu} \int_{0}^{1}(2\theta-b(x)\theta_x)S^2_t\dif x
\ge -\frac{1}{8\mu}\int_{0}^{1} S^2_t\dif x-CE^{\frac{1}{2}}(t)\mathcal{D}(t).
\end{align*}
Therefore, we get 
\begin{align}\label{u_{t4}}
\frac{1}{2}\frac{\dif}{\dif t}\int_{0}^{1}\frac{\tau  }{\mu}\theta S^2_t\dif x+\int_{0}^{1}\frac{1}{4\mu}S^2_t\dif x-\int_{0}^{1}\theta u_{tx}S_t\dif x\le CE^{\frac{1}{2}}(t)\mathcal{D}(t) .
\end{align}
Thus,  combining \eqref{u_{t1}}-\eqref{u_{t4}}, we derive that
\begin{align*}
\frac{\dif}{\dif t} \int_0^1  \left( \frac{\theta}{2} u_t^2+\frac{R\theta^2}{v^2} v_t^2+\frac{e_\theta}{2} \theta_t^2+\frac{Z(\theta)}{2} q_t^2 +\frac{\tau \theta}{2\mu}  S_t^2\right) \dif x 
+\int_0^1 \left(\frac{v}{\kappa (\theta)} q_t^2+\frac{1}{4\mu} S_t^2 \right)\dif x \le C E^\frac{1}{2} (t) \mathcal D (t).
\end{align*}
Integrating the above inequality over $(0, t)$, using \eqref{new-hu2-1}-\eqref{new-hu4-3} and noticing that 
\begin{align*}
\int_0^1  \left( \frac{\theta}{2} u_t^2+\frac{R\theta^2}{v^2} v_t^2+\frac{e_\theta}{2} \theta_t^2+\frac{Z(\theta)}{2} q_t^2 +\frac{\tau \theta}{2\mu}  S_t^2\right)  (t=0, x)\dif x \le C \|V_0^1\|_{L^2}^2 \le CE_0,
\end{align*}
  \eqref{yjgjt} follows immediately.
\end{proof}
\end{lemma}

The next lemma show the dissipative estimates of $D(u,v,\theta)$
\begin{lemma}\label{yjhsth}
There exists some constant $C$ such that
\begin{align}\label{yjhs}
\int_{0}^{t}\int_{0}^{1}|D(u,v,\theta)|^2\dif x\dif t\le C\left(E_0+E^\frac{1}{2}(t)\int_{0}^{t}\mathcal{D}(s)\dif s\right).
\end{align}
\end{lemma}
\begin{proof}
From $\eqref{approximate}_5$, we can get 
\begin{align*}
\int_{0}^{t}\int_{0}^{1} (u^2_x+v^2_t)\dif x\dif t \le C \int_{0}^{t}\int_{0}^{1}(\tau^2(S_t^2+\epsilon S^2_x)+S^2)\dif x\dif t
\le C\left(E_0+E^\frac{1}{2}(t)\int_{0}^{t}\mathcal{D}(s)\dif s\right).
\end{align*}
In view of  $\eqref{approximate}_4$ and $\eqref{approximate}_3$, it follows that
\begin{align*}
\int_{0}^{t}\int_{0}^{1}\theta^2_x\dif x\dif t\le C \int_{0}^{t}\int_{0}^{1}(\tau^2 q^2_t+q^2)\dif x\dif t\le C\left(E_0+E^\frac{1}{2}(t)\int_{0}^{t}\mathcal{D}(s)\dif s\right).\end{align*}
and
\begin{align*}
\int_{0}^{t}\int_{0}^{1}\theta^2_t\dif x\dif t
&\le\int_{0}^{t}\int_{0}^{1}(u^2_x+q^2_x)\dif x\dif t+C\left(E^\frac{1}{2}(t)\int_{0}^{t}\mathcal{D}(s)\dif s+E_0\right)\\
&\le C\left(E_0+E^\frac{1}{2}(t)\int_{0}^{t}\mathcal{D}(s)\dif s\right).
\end{align*}
On the other hand, multiplying the equation $\eqref{approximate}_2$ by $u_t$, it yields
\begin{align*}
&\int_{0}^{t}\int_{0}^{1}u^2_t\dif x\dif t
=-\int_{0}^{t}\int_{0}^{1}p(v,\theta)_xu_t\dif x\dif t+\int_{0}^{t}\int_{0}^{1}S_xu_t\dif x\dif t\\
&=-\int_{0}^{t}\frac{\dif}{\dif t}\int_{0}^{1}p(v,\theta)_xu\dif x\dif t+\int_{0}^{t}\int_{0}^{1}p(v,\theta)_{tx}u\dif x\dif t+\int_{0}^{t}\int_{0}^{1}S_xu_t\dif x\dif t\\
&\le-\int_{0}^{1}p(v,\theta)_xu\dif x+CE_0-\int_{0}^{t}\int_{0}^{1}p(v,\theta)_tu_x\dif x\dif t+\int_{0}^{t}\int_{0}^{1}S_xu_t\dif x\dif t\\
&\le\frac{1}{2}\int_{0}^{t}\int_{0}^{1}u^2_t\dif x\dif t+C\int_{0}^{1}(v^2_x+\theta^2_x+u^2)\dif x+C\int_{0}^{t}\int_{0}^{1}(v^2_t+\theta^2_t+u^2_x+S^2_x)\dif x\dif t+CE_0,
\end{align*}
which gives
\begin{align*}
\int_{0}^{t}\int_{0}^{1}u^2_t\dif x\dif t
\le C\left(E_0+E^\frac{1}{2}(t)\int_{0}^{t}\mathcal{D}(s)\dif s\right).
\end{align*}
Similarly, using the equation $\eqref{approximate}_2$, one also get 
\begin{align*}
\int_{0}^{t}\int_{0}^{1}v^2_x\dif x\dif t
\le C\left(E_0+E^\frac{1}{2}(t)\int_{0}^{t}\mathcal{D}(s)\dif s\right).
\end{align*}
Thus, combining the above estimates, the estimate \eqref{yjhs} follows immediately.
\end{proof}

Combining Lemmas $\eqref{0j}-\eqref{yjhsth}$, we conclude that
\begin{lemma}\label{le3.7}
There exists some constant $C$ such that
\begin{align}
\sum_{\alpha=0}^{1}\left\|D^{\alpha}(v-1,\theta-1,u,\sqrt{\tau   }q,\sqrt{\tau  }S)\right\|_{L^{2}}^{2}+\int_{0}^{t}(\left\|D(v,u,\theta)\right\|_{L^{2}}^{2}+\sum_{\alpha=0}^{1}\left\|D^{\alpha}(q,S)\right\|_{L^{2}}^{2})\dif t\nonumber\\
\le C\left(E_0+E^\frac{1}{2}(t)\int_{0}^{t}\mathcal{D}(s)\dif s\right).
\end{align}
\end{lemma}

%Second-order estimates
Now, we are ready to give the second-order estimates of solutions.
\begin{lemma}\label{le3.8}
There exists some constant $C$ such that
\begin{align}\label{ejtt}
\int_{0}^{1} \tau^2 (u^2_{tt}+v^2_{tt}+\theta^2_{tt}+\tau q^2_{tt}+\tau  S^2_{tt})\dif x+\int_{0}^{t}\int_{0}^{1} \tau^2 (q^2_{tt}+S^2_{tt})\dif x\dif t\nonumber\\
\le C\left(E_0+E^\frac{1}{2}(t)\int_{0}^{t}\mathcal{D}(s)\dif s\right).
\end{align}
\end{lemma}
\begin{proof}
Taking derivatives with respect to $t$ twice to the equations \eqref{approximate}, we get
\begin{align}\label{tt}
\begin{cases}
v_{ttt}=u_{ttx},\\
u_{ttt}+{p(v,\theta)}_{ttx}=S_{ttx},\\
(e_{\theta}\theta_t)_{tt}-(\frac{2a(\theta)}{Z(\theta)}q\theta_x)_{tt}+(\frac{R\theta u_x}{v})_{tt}+q_{ttx}=(\frac{2a(\theta)}{\tau   (\theta)}q^2v)_{tt}+(\frac{S^2v}{\mu})_{tt},\\
\left(\tau   (\theta)q_t\right)_{tt}+(vq)_{tt}+(\kappa(\theta)\theta_x)_{tt}=0,\\
\tau  S_{ttt}+(vS)_{tt}+\tau  \epsilon b(x)S_{ttx}=\mu u_{ttx}.
\end{cases}
\end{align} 
Multiplying the equation $\eqref{tt}_2$ by $\theta u_{tt}$, we get
$$\int_{0}^{1} \theta u_{ttt}u_{tt}\dif x+\int_{0}^{1} \theta p(v,\theta)_{xtt}u_{tt}\dif x-\int_{0}^{1} \theta S_{xtt}u_{tt}\dif x=0,$$
where
\begin{align*}
\int_{0}^{1} \theta u_{ttt}u_{tt}\dif x=\frac{1}{2}\frac{\dif}{\dif t}\int_{0}^{1}\theta u^2_{tt}\dif x-\frac{1}{2}\int_{0}^{1}\theta_tu^2_{tt}\dif x\ge\frac{1}{2}\frac{\dif}{\dif t}\int_{0}^{1}\theta u^2_{tt}\dif x-CE^{\frac{1}{2}}(t)\mathcal{D}(t)
\end{align*}
and
\begin{align*}
&\int_{0}^{1} \theta p(v,\theta)_{xtt}u_{tt}\dif x\\
=&\int_{0}^{1}(p_{vvv}v^2_tv_x+2p_{vv\theta}v_tv_x\theta_t+p_{vv\theta}v^2_t\theta_x+2p_{vv}v_{tx}v_t+p_{vv}v_{tt}v_x+p_{v\theta}v_x\theta_{tt}+2p_{v\theta}v_{tx}\theta_t+p_{v\theta}v_{tt}\theta_x\\
&+2p_{v\theta}v_t\theta_{tx}+p_vv_{ttx}+p_\theta\theta_{ttx})\theta u_{tt}\dif x\\
\ge&\int_{0}^{1}(p_vv_{ttx}+p_\theta\theta_{ttx})\theta u_{tt}\dif x-CE^{\frac{1}{2}}(t)\mathcal{D}(t)\\
=&\int_{0}^{1}\theta p_vu_{txx}u_{tt}\dif x+\int_{0}^{1}\frac{R\theta}{v}\theta_{ttx}u_{tt}\dif x-CE^{\frac{1}{2}}(t)\mathcal{D}(t)\\
=&\theta p_vu_{tx}u_{tt}\big|_{0}^{1}-\int_{0}^{1}(\theta_xp_vu_{tt}+\theta p_{vv}v_xu_{tt}+\theta p_{v\theta}\theta_xu_{tt}+\theta p_vu_{ttx})u_{tx}\dif x+\int_{0}^{1}\frac{R\theta}{v}\theta_{ttx}u_{tt}\dif x-CE^{\frac{1}{2}}(t)\mathcal{D}(t)\\
\ge&-\int_{0}^{1}\theta p_vu_{ttx}u_{tx}\dif x+\int_{0}^{1}\frac{R\theta}{v}\theta_{ttx}u_{tt}\dif x-CE^{\frac{1}{2}}(t)\mathcal{D}(t)\\
=&-\frac{1}{2}\frac{\dif}{\dif t}\int_{0}^{1}\theta p_vv^2_{tt}\dif x+\frac{1}{2}\int_{0}^{1}(p_{vv}v_t\theta+p_{v\theta}\theta_t\theta+p_v\theta_t)v^2_{tt}\dif x+\int_{0}^{1}\frac{R\theta}{v}\theta_{ttx}u_{tt}\dif x-CE^{\frac{1}{2}}(t)\mathcal{D}(t)\\
\ge&-\frac{1}{2}\frac{\dif}{\dif t}\int_{0}^{1}\theta p_vv^2_{tt}\dif x+\int_{0}^{1}\frac{R\theta}{v}\theta_{ttx}u_{tt}\dif x-CE^{\frac{1}{2}}(t)\mathcal{D}(t).
\end{align*}

Note that the above estimates (also in what follows) is essentially formal because the regularity of the local solutions is insufficient to validate several of steps. In particular, we use the high-order compatibility condition $u_{tt}(t,0)=u_{tt}(t,1)=0$. However, a standard density argument will eliminate the needs for the extra regularity of local solutions.

It follows that
\begin{align*}
-\int_{0}^{1}\theta S_{xtt}u_{tt}\dif x=-\theta S_{tt}u_{tt}\big|_{0}^{1}+\int_{0}^{1}(\theta_xS_{tt}u_{tt}+\theta S_{tt}u_{ttx})\dif x
\ge\int_{0}^{1}\theta S_{tt}u_{ttx}\dif x-\frac{C}{\tau} E^{\frac{1}{2}}(t)\mathcal{D}(t).
\end{align*}
Therefore, we derive that
\begin{align}\label{u_{tt1}}
\frac{1}{2}\frac{\dif}{\dif t}\int_{0}^{1}(\theta u^2_{tt}+\frac{R\theta^2}{v^2} v^2_{tt})\dif x+\int_{0}^{1}\frac{R\theta}{v}\theta_{ttx}u_{tt}\dif x+\int_{0}^{1}\theta S_{tt}u_{ttx}\dif x\le \frac{C}{\tau}E^{\frac{1}{2}}(t)\mathcal{D}(t).
\end{align}
Multiplying the equation $\eqref{tt}_3$ by $\theta_{tt}$, we get
\begin{align}
&\int_{0}^{1}(e_{\theta}\theta_t)_{tt}\theta_{tt}\dif x-\int_{0}^{1}\left(\frac{2a(\theta)}{Z(\theta)}q\theta_x\right)_{tt}\theta_{tt}\dif x+\int_{0}^{1}\left(\frac{R\theta u_x}{v}\right)_{tt}\theta_{tt}\dif x+\int_{0}^{1}q_{ttx}\theta_{tt}\dif x \nonumber\\
&=\int_{0}^{1}\left(\frac{2a(\theta)}{\tau   (\theta)}q^2v\right)_{tt}\theta_{tt}\dif x+\int_{0}^{1}\left(\frac{S^2v}{\mu}\right)_{tt}\theta_{tt}\dif x. \label{new-hu-3}
\end{align}
We shall estimate each term of \eqref{new-hu-3} individually. 

First,  by noting that $|e_{\theta q}|_{L^\infty} \le C \tau |q|_{L^\infty} \le C \tau E^\frac{1}{2}(t)$, we have
\begin{align*}
&\int_{0}^{1}(e_{\theta}\theta_t)_{tt}\theta_{tt}\dif x\\
&=\int_{0}^{1}\left(e_{\theta\theta\theta}\theta^3_t+2e_{\theta\theta q}\theta^2_tq_t+e_{\theta qq}\theta_tq^2_t+2e_{\theta\theta}\theta_t\theta_{tt}+e_{\theta q}\theta_tq_{tt}+e_{\theta q}\theta_{tt}q_t+e_{\theta}\theta_{ttt}\right)\theta_{tt}\dif x\\
&\ge \int_{0}^{1}e_{\theta}\theta_{ttt}\theta_{tt}\dif x-CE^{\frac{1}{2}}(t)\mathcal{D}(t)\\
&=\frac{1}{2}\frac{\dif}{\dif t}\int_{0}^{1}e_{\theta}\theta^2_{tt}\dif x-\frac{1}{2}\int_{0}^{1}\left(e_{\theta\theta}\theta_t+e_{\theta q}q_t\right)\theta^2_{tt}\dif x-CE^{\frac{1}{2}}(t)\mathcal{D}(t)\\
&\ge\frac{1}{2}\frac{\dif}{\dif t}\int_{0}^{1}e_{\theta}\theta^2_{tt}\dif x-CE^{\frac{1}{2}}(t)\mathcal{D}(t),
\end{align*}
and
\begin{align*}
&-\int_{0}^{1}(\frac{2a(\theta)}{Z(\theta)}q\theta_x)_{tt}\theta_{tt}\dif x\\
=&-(\frac{2a^{\prime\prime}(\theta)}{Z(\theta)}\theta^2_tq\theta_x-\frac{4a^\prime(\theta)Z^\prime(\theta)}{Z^2(\theta)}\theta^2_tq\theta_x+\frac{2a^{\prime}(\theta)}{Z(\theta)}\theta_{tt}q\theta_x+\frac{4a^{\prime}(\theta)}{Z(\theta)}\theta_tq_t\theta_x+\frac{4a^{\prime}(\theta)}{Z(\theta)}\theta_tq\theta_{tx}\\
&-\frac{2a(\theta)Z^{\prime\prime}(\theta)}{Z^2(\theta)}\theta^2_tq\theta_x+\frac{4a(\theta){Z^\prime(\theta)}^2}{Z^3(\theta)}\theta^2_tq\theta_x-\frac{2a(\theta)Z^\prime(\theta)}{Z^2(\theta)}\theta_{tt}q\theta_x-\frac{4a(\theta)Z^\prime(\theta)}{Z^2(\theta)}\theta_tq\theta_{tx}\\
&-\frac{4a(\theta)Z^\prime(\theta)}{Z^2(\theta)}q_t\theta_t\theta_x+\frac{2a(\theta)}{Z(\theta)}q_{tt}\theta_x+\frac{4a(\theta)}{Z(\theta)}q_t\theta_{tx}+\frac{2a(\theta)}{Z(\theta)}q\theta_{ttx})\theta_{tt}\dif x\\
\ge&-\int_{0}^{1}\frac{2a(\theta)}{Z(\theta)}q\theta_{ttx}\theta_{tt}\dif x-\frac{C}{\tau}E^{\frac{1}{2}}(t)\mathcal{D}(t)\\
=&-\frac{a(\theta)}{Z(\theta)}q\theta^2_{tt}\big|_{0}^{1}+\int_{0}^{1}\theta^2_{tt}\left(\frac{a^\prime(\theta)}{Z(\theta)}\theta_xq-\frac{a(\theta)Z^\prime(\theta)}{Z^2(\theta)}\theta_xq+\frac{a(\theta)}{Z(\theta)}q_x\right)\dif x-\frac{C}{\tau}E^{\frac{1}{2}}(t)\mathcal{D}(t)\\
\ge&-\frac{C}{\tau}E^{\frac{1}{2}}(t)\mathcal{D}(t).
\end{align*}
Similarly, one has
\begin{align*}
&\int_{0}^{1}\left(\frac{R\theta u_x}{v}\right)_{tt}\theta_{tt}\dif x\\
&=\int_{0}^{1}R\left(\frac{\theta_{tt}u_x}{v}+\frac{2\theta_tu_{xt}}{v}-\frac{2\theta_tu_xv_t}{v^2}-\frac{2\theta u_{xt}v_t}{v^2}-\frac{\theta u_xv_{tt}}{v^2}+\frac{2\theta u_xv^2_t}{v^3}+\frac{\theta u_{xtt}}{v}\right)\theta_{tt}\dif x\\
&\ge\int_{0}^{1}\frac{R\theta}{v}u_{xtt}\theta_{tt}\dif x-CE^{\frac{1}{2}}(t)\mathcal{D}(t)\\
&=\frac{R\theta}{v}u_{tt}\theta_{tt}\big|_{0}^{1}-\int_{0}^{1}\left(\frac{R\theta_x\theta_{tt}}{v}+\frac{R\theta\theta_{xtt}}{v}-\frac{R\theta v_x\theta_{tt}}{v^2}\right)u_{tt}\dif x-CE^{\frac{1}{2}}(t)\mathcal{D}(t)\\
&\ge-\int_{0}^{1}\frac{R\theta}{v}\theta_{xtt}u_{tt}\dif x-CE^{\frac{1}{2}}(t)\mathcal{D}(t),
\end{align*}
and
\begin{align*}
\int_{0}^{1}q_{ttx}\theta_{tt}\dif x=q_{tt}\theta_{tt}\big|_{0}^{1}-\int_{0}^{1}q_{tt}\theta_{ttx}\dif x=-\int_{0}^{1}q_{tt}\theta_{ttx}\dif x.
\end{align*}
For the right-hand-side of \eqref{new-hu-3}, one obtain
\begin{align*}
&\int_{0}^{1}\left(\frac{2a(\theta)}{\tau   (\theta)}q^2v\right)_{tt}\theta_{tt}\dif x\\
=&\{\frac{2a^{\prime\prime}(\theta)}{\tau   (\theta)}\theta^2_tq^2v-\frac{4a^\prime(\theta)\tau   ^\prime(\theta)}{\tau   ^2(\theta)}\theta^2_tq^2v+\frac{2a^{\prime}(\theta)}{\tau   (\theta)}\theta_{tt}q^2v+\frac{4a^{\prime}(\theta)}{\tau   (\theta)}\theta_t2qq_tv+\frac{4a^{\prime}(\theta)}{\tau   (\theta)}\theta_tq^2v_t\\
&-\frac{2a(\theta)\tau   ^{\prime\prime}(\theta)}{\tau   ^2(\theta)}\theta^2_tq^2v+\frac{4a(\theta){\tau   ^\prime(\theta)}^2}{\tau   ^3(\theta)}\theta^2_tq^2v-\frac{2a(\theta)\tau   ^\prime(\theta)}{\tau   ^2(\theta)}\theta_{tt}q^2v-\frac{2a(\theta)\tau   ^\prime(\theta)}{\tau   ^2(\theta)}\theta_tq^2v_t-\frac{4a(\theta)\tau   ^\prime(\theta)}{\tau   ^2(\theta)}2qq_t\theta_tv\\
&+\frac{2a(\theta)}{\tau   (\theta)}2qq_{tt}v+\frac{2a(\theta)}{\tau   (\theta)}2q^2_tv+\frac{4a(\theta)}{\tau   (\theta)}2qq_tv_t-\frac{4a(\theta)\tau   ^\prime(\theta)}{\tau   ^2(\theta)}\theta_tq^2v_t+\frac{2a(\theta)}{\tau   (\theta)}q^2v_{tt}\}\theta_{tt}\dif x\\
\le& \frac{C}{\tau}E^{\frac{1}{2}}(t)\mathcal{D}(t)
\end{align*}
and
\begin{align*}
\int_{0}^{1}(\frac{S^2v}{\mu})_{tt}\theta_{tt}\dif x=\int_{0}^{1}\frac{1}{\mu}\left(2S^2_tv+2SS_{tt}v+4SS_tv_t+S^2v_{tt}\right)\theta_{tt}\dif x\le \frac{C}{\tau}E^{\frac{1}{2}}(t)\mathcal{D}(t).
\end{align*}
Therefore, we derive that
\begin{align}\label{u_{tt2}}
\frac{1}{2}\frac{\dif}{\dif t}\int_{0}^{1}e_{\theta}\theta^2_{tt}\dif x-\int_{0}^{1}\frac{R\theta}{v}\theta_{xtt}u_{tt}\dif x-\int_{0}^{1}q_{tt}\theta_{ttx}\dif x\le \frac{C}{\tau}E^{\frac{1}{2}}(t)\mathcal{D}(t).
\end{align}
Multiplying the equation $\eqref{tt}_4$ by $\frac{q_{tt}}{\kappa(\theta)}$, we get
\begin{align}\label{new-hu-4}
\int_{0}^{1}\left(\tau   (\theta)q_t\right)_{tt}\frac{q_{tt}}{\kappa(\theta)}\dif x+\int_{0}^{1}(vq)_{tt}\frac{q_{tt}}{\kappa(\theta)}\dif x+\int_{0}^{1}(\kappa(\theta)\theta_x)_{tt}\frac{q_{tt}}{\kappa(\theta)}\dif x=0.
\end{align}
Using similar method as before,  one can get
\begin{align*}
\int_{0}^{1}\left(\tau   (\theta)q_t\right)_{tt}\frac{q_{tt}}{\kappa(\theta)}\dif
x&=\int_{0}^{1}\frac{1}{\kappa(\theta)}\left(\tau   ^{\prime\prime}(\theta)\theta^2_tq_t+\tau   ^\prime(\theta)\theta_{tt}q_t+2\tau   ^\prime(\theta)\theta_tq_{tt}+\tau   (\theta)q_{ttt}\right)q_{tt}\dif x\\
&\ge\int_{0}^{1}Z(\theta)q_{ttt}q_{tt}\dif x-\frac{C}{\tau} E^{\frac{1}{2}}(t)\mathcal{D}(t)\\
&=\frac{1}{2}\frac{\dif}{\dif t}\int_{0}^{1}Z(\theta)q^2_{tt}\dif x-\frac{1}{2}\int_{0}^{1}Z^\prime(\theta)\theta_tq^2_{tt}\dif x-\frac{C}{\tau} E^{\frac{1}{2}}(t)\mathcal{D}(t)\\
&\ge\frac{1}{2}\frac{\dif}{\dif t}\int_{0}^{1}Z(\theta)q^2_{tt}\dif x-\frac{C}{\tau}E^{\frac{1}{2}}(t)\mathcal{D}(t),
\end{align*}
and
\begin{align*}
\int_{0}^{1}(vq)_{tt}\frac{q_{tt}}{\kappa(\theta)}\dif x
=\int_{0}^{1}\frac{1}{\kappa(\theta)}\left(v_{tt}q+2v_tq_t+vq_{tt}\right)q_{tt}\dif x
\ge\int_{0}^{1}\frac{v}{\kappa(\theta)}q^2_{tt}\dif x-\frac{C}{\tau}E^{\frac{1}{2}}(t)\mathcal{D}(t),
\end{align*}
and
\begin{align*}
\int_{0}^{1}(\kappa(\theta)\theta_x)_{tt}\frac{q_{tt}}{\kappa(\theta)}\dif x
&=\int_{0}^{1}\frac{1}{\kappa(\theta)}\left(\kappa^{\prime\prime}(\theta)\theta^2_t\theta_x+\kappa^\prime(\theta)\theta_{tt}\theta_x+2\kappa^\prime(\theta)\theta_t\theta_{tx}+\kappa(\theta)\theta_{ttx}\right)q_{tt}\dif x\\
&\ge\int_{0}^{1}\theta_{ttx}q_{tt}\dif x-\frac{C}{\tau}E^{\frac{1}{2}}(t)\mathcal{D}(t).
\end{align*}
Therefore, we derive that
\begin{align}\label{u_{tt3}}
\frac{1}{2}\frac{\dif}{\dif t}\int_{0}^{1}Z(\theta)q^2_{tt}\dif x+\int_{0}^{1}\frac{v}{\kappa(\theta)}q^2_{tt}\dif x+\int_{0}^{1}\theta_{ttx}q_{tt}\dif x\le \frac{C}{\tau} E^{\frac{1}{2}}(t)\mathcal{D}(t).
\end{align}
Multiplying the equation $\eqref{tt}_5$ by $\frac{\theta}{\mu}S_{tt}$, we get
\begin{align}
\int_{0}^{1}\frac{\theta}{\mu}S_{tt}\tau  S_{ttt}\dif x+\int_{0}^{1}\frac{\theta}{\mu}S_{tt}(vS)_{tt}\dif x-\int_{0}^{1}\theta S_{tt} u_{ttx}\dif x+\int_{0}^{1}\frac{\tau  \epsilon}{\mu} b(x)S_{ttx}\theta S_{tt}\dif x 
=0. \label{new-hu-5}
\end{align}
For the left-hand-side of \eqref{new-hu-5},  we have
\begin{align*}
\int_{0}^{1}\frac{\theta}{\mu}S_{tt}\tau  S_{ttt}\dif x
=\frac{\tau  }{2\mu}\frac{\dif}{\dif t}\int_{0}^{1}\theta S^2_{tt}\dif x-\frac{\tau  }{2\mu}\int_{0}^{1}\theta_t S^2_{tt}\dif x
\ge\frac{\tau  }{2\mu}\frac{\dif}{\dif t}\int_{0}^{1}\theta S^2_{tt}\dif x-\frac{C}{\tau} E^{\frac{1}{2}}(t)\mathcal{D}(t),
\end{align*}
and
\begin{align*}
\int_{0}^{1}\frac{\theta}{\mu}S_{tt}(vS)_{tt}\dif x
=\int_{0}^{1}\frac{\theta}{\mu}(v_{tt}S+2v_tS_t+vS_{tt})S_{tt}\dif x
\ge\int_{0}^{1}\frac{1}{2\mu}S^2_{tt}\dif x-\frac{C}{\tau}E^{\frac{1}{2}}(t)\mathcal{D}(t).
\end{align*}
Choosing $\epsilon$ small enough, one get
\begin{align*}
\int_{0}^{1}\frac{\tau  \epsilon}{\mu} b(x)S_{ttx}\theta S_{tt}\dif x
&=\frac{\tau  \epsilon}{2\mu}b(x)S^2_{tt}\big|_0^{1}-\frac{\tau  \epsilon}{2\mu} \int_{0}^{1}(2\theta-b(x)\theta_x)S^2_{tt}\dif x\\
&\ge -\frac{1}{8\mu}\int_{0}^{1}\theta S^2_{tt}\dif x-\frac{C}{\tau}E^{\frac{1}{2}}(t)\mathcal{D}(t).
\end{align*}
Therefore, we derive that
\begin{align}\label{u_{tt4}}
\frac{\tau  }{2\mu}\frac{\dif}{\dif t}\int_{0}^{1}\theta S^2_{tt}\dif x+\int_{0}^{1}\frac{1}{4\mu}S^2_{tt}\dif x-\int_{0}^{1}\theta S_{tt} u_{ttx}\dif x\le \frac{C}{\tau}E^{\frac{1}{2}}(t)\mathcal{D}(t).
\end{align}
Thus, combining the inequalities \eqref{u_{tt1}}, \eqref{u_{tt2}}, \eqref{u_{tt3}}, \eqref{u_{tt4}}, we derive that 
\begin{align*}
\frac{\dif}{\dif t} \int_0^1 \tau^2 \left( \frac{\theta}{2} u_{tt}^2+ \frac{R\theta^2}{2v^2} v_{tt}^2+\frac{e_\theta}{2}\theta_{tt}^2+\frac{Z(\theta)}{2} q_{tt}^2 +\frac{\tau}{2\mu} \theta S_{tt}^2 \right) \dif x \\
+\int_0^1 \tau^2 \left(\frac{v}{\kappa (\theta)} q_{tt}^2+\frac{1}{4\mu}  S_{tt}^2 \right) \dif x \le C E^\frac{1}{2}(t) \mathcal D(t).
\end{align*}
Integrating the above inequalities  over $(0, t)$, using \eqref{new-hu2-1}-\eqref{new-hu4-3} and noticing that 
\begin{align*}
\int_0^1 \tau^2 \left( \frac{\theta}{2} u_{tt}^2+ \frac{R\theta^2}{2v^2} v_{tt}^2+\frac{e_\theta}{2}\theta_{tt}^2+\frac{Z(\theta)}{2} q_{tt}^2 +\frac{\tau}{2\mu} \theta S_{tt}^2 \right)    (t=0, x)\dif x
 \le C \|V_0^2\|_{L^2}^2 \le CE_0,
\end{align*}
 we  get the estimate \eqref{ejtt} immediately.
\end{proof}

%Estimates of "tx"
The next lemma gives the $(x,t)$ mixed second-order estimates.
\begin{lemma}\label{le3.9}
There exists some constant $C$ such that
\begin{align}\label{ejtx}
&\int_{0}^{1}(u^2_{tx}+v^2_{tx}+\theta^2_{tx}+\tau   q^2_{tx}+\tau  S^2_{tx})\dif x+\int_{0}^{t}\int_{0}^{1}(q^2_{tx}+S^2_{tx})\dif x\dif t+\frac{3\tau  \epsilon}{8\mu}\int_{0}^{t}(S^2_{tx}(t,0)+S^2_{tx}(t,1))\dif t\nonumber\\
&\le C\left(E_0+E^\frac{1}{2}(t)\int_{0}^{t}\mathcal{D}(s)\dif s+E^\frac{3}{2}(t) \right).
\end{align}
\end{lemma}
	
\begin{proof}
Taking derivatives to the equations \eqref{approximate} with respect to $t$ and $x$, respectively, we get
\begin{align}\label{tx}
\begin{cases}
v_{ttx}=u_{txx},\\
u_{ttx}+{p(v,\theta)}_{txx}=S_{txx},\\
(e_{\theta}\theta_t)_{tx}-(\frac{2a(\theta)}{Z(\theta)}q\theta_x)_{tx}+(\frac{R\theta u_x}{v})_{tx}+q_{txx}=(\frac{2a(\theta)}{\tau   (\theta)}q^2v)_{tx}+(\frac{S^2v}{\mu})_{tx},\\
\left(\tau   (\theta)q_t\right)_{tx}+(vq)_{tx}+(\kappa(\theta)\theta_x)_{tx}=0,\\
\tau  S_{ttx}+(vS)_{tx}+\tau  \epsilon(b(x)S_x)_{tx}=\mu u_{txx}.
\end{cases}
\end{align} 
Multiplying the equation $\eqref{tx}_2$ by $\theta u_{tx}$, we get
$$\int_{0}^{1} \theta u_{ttx}u_{tx}\dif x+\int_{0}^{1} \theta \left(p(v,\theta)_{txx}-S_{txx}\right)u_{tx}\dif x=0,$$
where
\begin{align*}
\int_{0}^{1} \theta u_{ttx}u_{tx}\dif x=\frac{1}{2}\frac{\dif}{\dif t}\int_{0}^{1}\theta u^2_{tx}\dif x-\frac{1}{2}\int_{0}^{1}\theta_tu^2_{tx}\dif x\ge\frac{1}{2}\frac{\dif}{\dif t}\int_{0}^{1}\theta u^2_{tx}\dif x-CE^{\frac{1}{2}}(t)\mathcal{D}(t).
\end{align*}
Note that $u_{tt}(t,0)=u_{tt}(t,1)=0$, the momentum equation $\eqref{approximate}_2$ implies 
$$\left(p(v,\theta)_{tx}-S_{tx}\right)\Big|_{\partial \Omega}=0.
$$
Thus, it follows
\begin{align*}
&\int_{0}^{1} \theta \left(p(v,\theta)_{txx}-S_{txx}\right)u_{tx}\dif x\\
=&\theta \left(p(v,\theta)_{tx}-S_{tx}\right)u_{tx}\big|_{0}^{1}-\int_{0}^{1} (\theta u_{tx})_x\left(p(v,\theta)_{tx}-S_{tx}\right)\dif x\\
=&-\int_{0}^{1}(\theta_xu_{tx}+\theta u_{txx})\left(p_{vv}v_xv_t+p_{v\theta}v_t\theta_x+ p_{v\theta}v_x\theta_t+p_vv_{tx}+p_\theta\theta_{tx}-S_{tx}\right)\dif x\\
\ge&-\int_{0}^{1}\theta u_{txx}\left(p_{vv}v_xv_t+p_{v\theta}v_t\theta_x+ p_{v\theta}v_x\theta_t+p_vv_{tx}+p_\theta\theta_{tx}-S_{tx}\right)\dif x-CE^{\frac{1}{2}}(t)\mathcal{D}(t)\\
=&-\int_{0}^{1}\theta p_v v_{ttx}v_{tx}\dif x-\int_{0}^{1}\frac{R\theta}{v} u_{txx}\theta_{tx}\dif x+\int_{0}^{1}\theta u_{txx}S_{tx}\dif x-\frac{\dif}{\dif t}\int_{0}^{1}u_{xx}\theta A_1(t,x)\dif x\\
&+\int_{0}^{1}u_{xx}\left(\theta_tA_1(t,x)+\theta (A_1(t,x))_t\right)\dif x-CE^{\frac{1}{2}}(t)\mathcal{D}(t).
\end{align*}
Here $A_1(t,x):=p_{vv}v_xv_t+p_{v\theta}v_t\theta_x+ p_{v\theta}v_x\theta_t$. Notice that,
\begin{align*}
\int_0^1 u_{xx} \theta_t A_1(t,x) \dif x \le C E^{\frac{1}{2}}(t)\mathcal{D}(t)
\end{align*}
and
\begin{align*}
&\int_{0}^{1}u_{xx}\theta (A_1(t,x))_t\dif x\\
=&\int_{0}^{1}u_{xx}\theta(p_{vvv}v_xv^2_t+2p_{vv\theta}v_xv_t\theta_t+p_{vv\theta}v^2_t\theta_x+p_{vv}v_{tx}v_t+p_{vv}v_{tt}v_x+p_{v\theta}v_{tt}\theta_x+p_{v\theta}v_t\theta_{tx}\\
&+p_{v\theta}v_{tx}\theta_t+p_{v\theta}v_x\theta_{tt})\dif x\ge-CE^{\frac{1}{2}}(t)\mathcal{D}(t).
\end{align*}
Similarly,
\begin{align*}
-\int_{0}^{1}\theta p_vv_{ttx}v_{tx}\dif x
&=-\frac{1}{2}\frac{\dif}{\dif t}\int_{0}^{1}\theta p_vv^2_{tx}\dif x+\frac{1}{2}\int_{0}^{1}\{\theta_tp_v+\theta(p_{vv}v_t+p_{v\theta}\theta_t)\}v^2_{tx}\dif x\\
&\ge-\frac{1}{2}\frac{\dif}{\dif t}\int_{0}^{1}\theta p_vv^2_{tx}\dif x-CE^{\frac{1}{2}}(t)\mathcal{D}(t).
\end{align*}
Combining the above estimates, we derive that
\begin{align}\label{u_{tx1}}
\frac{1}{2}\frac{\dif}{\dif t}\int_{0}^{1}\theta u^2_{tx}+ \frac{R\theta^2}{v^2}v^2_{tx}\dif x+\int_{0}^{1}\theta u_{txx}S_{tx}\dif x-\int_{0}^{1}\frac{R\theta}{v} u_{txx}\theta_{tx}\dif x\nonumber\\
\le\frac{\dif}{\dif t}\int_{0}^{1}u_{xx}\theta A_1(t,x)\dif x+CE^{\frac{1}{2}}(t)\mathcal{D}(t).
\end{align}
Multiplying the equation $\eqref{tx}_3$ by $\theta_{tx}$, we get
\begin{align}
&\int_{0}^{1}(e_{\theta}\theta_t)_{tx}\theta_{tx}\dif x+\int_{0}^{1}\left(\frac{R\theta u_x}{v}\right)_{tx}\theta_{tx}\dif x-\int_{0}^{1}\left(\frac{2a(\theta)}{Z(\theta)}q\theta_x\right)_{tx}\theta_{tx}\dif x+\int_{0}^{1}q_{txx}\theta_{tx}\dif x \nonumber \\
&=\int_{0}^{1}\left(\frac{2a(\theta)}{\tau   (\theta)}q^2v\right)_{tx}\theta_{tx}\dif x+\int_{0}^{1}\left(\frac{S^2v}{\mu}\right)_{tx}\theta_{tx}\dif x. \label{new-hu2-3}
\end{align}
First,  using the fact that $|e_{\theta q}| \le C \tau q$ and  $|e_{\theta qq}| \le C \tau$, we have
\begin{align*}
&\int_{0}^{1}(e_{\theta}\theta_t)_{tx}\theta_{tx}\dif x\\
=&\int_{0}^{1}(e_{\theta\theta\theta}\theta^2_t\theta_x+e_{\theta\theta q}\theta^2_tq_x+e_{\theta\theta q}\theta_t\theta_xq_t+e_{\theta qq}\theta_tq_tq_x+2e_{\theta\theta}\theta_t\theta_{tx}+e_{\theta\theta}\theta_x\theta_{tt}\\
&+e_{\theta q}\theta_tq_{tx}+e_{\theta q}\theta_{tx}q_t+e_{\theta q}\theta_{tt}q_x+e_{\theta}\theta_{ttx})\theta_{tx}\dif x\\
\ge& \int_{0}^{1}e_{\theta}\theta_{ttx}\theta_{tx}\dif x-CE^{\frac{1}{2}}(t)\mathcal{D}(t)\\
=&\frac{1}{2}\frac{\dif}{\dif t}\int_{0}^{1}e_{\theta}\theta^2_{tx}\dif x-\frac{1}{2}\int_{0}^{1}\left(e_{\theta\theta}\theta_t+e_{\theta q}q_t\right)\theta^2_{tx}\dif x-CE^{\frac{1}{2}}(t)\mathcal{D}(t)\\
\ge&\frac{1}{2}\frac{\dif}{\dif t}\int_{0}^{1}e_{\theta}\theta^2_{tx}\dif x-CE^{\frac{1}{2}}(t)\mathcal{D}(t),
\end{align*}
and
\begin{align*}
&\int_{0}^{1}\left(\frac{R\theta u_x}{v}\right)_{tx}\theta_{tx}\dif x\\
=&\int_{0}^{1}R(\frac{\theta_{tx}u_x}{v}+\frac{\theta_tu_{xx}}{v}-\frac{\theta_tu_xv_x}{v^2}+\frac{\theta_xu_{tx}}{v}-\frac{\theta u_{xt}v_x}{v^2}-\frac{\theta_xu_xv_t}{v^2}-\frac{\theta u_{xx}v_t}{v^2}-\frac{\theta u_xv_{tx}}{v^2}\\
&+\frac{2\theta u_xv_tv_x}{v^3}+\frac{\theta u_{txx}}{v})\theta_{tx}\dif x\\
\ge&\int_{0}^{1}\frac{R\theta}{v}u_{txx}\theta_{tx}\dif x-CE^{\frac{1}{2}}(t)\mathcal{D}(t).
\end{align*}
Second, one has
\begin{align*}
&-\int_{0}^{1}\left(\frac{2a(\theta)}{Z(\theta)}q\theta_x\right)_{tx}\theta_{tx}\dif x\\
=&-\left(\frac{2a^{\prime\prime}(\theta)}{Z(\theta)}\theta^2_xq\theta_t-\frac{4a^\prime(\theta)Z^\prime(\theta)}{Z^2(\theta)}\theta^2_xq\theta_t+\frac{2a^{\prime}(\theta)}{Z(\theta)}\theta_tq_x\theta_x+\frac{4a^{\prime}(\theta)}{Z(\theta)}\theta_{tx}q\theta_x+\frac{2a^{\prime}(\theta)}{Z(\theta)}\theta_tq\theta_{xx}\right.\\
&-\frac{2a(\theta)Z^{\prime\prime}(\theta)}{Z^2(\theta)}\theta^2_xq\theta_t+\frac{4a(\theta){Z^\prime(\theta)}^2}{Z^3(\theta)}\theta^2_xq\theta_t-\frac{2a(\theta)Z^\prime(\theta)}{Z^2(\theta)}\theta_t\theta_xq_x-\frac{4a(\theta)Z^\prime(\theta)}{Z^2(\theta)}q\theta_x\theta_{tx}\\
&-\frac{2a(\theta)Z^\prime(\theta)}{Z^2(\theta)}q\theta_t\theta_{xx}+\frac{2a^\prime(\theta)}{Z(\theta)}q_t\theta^2_x-\frac{2a(\theta)Z^\prime(\theta)}{Z^2(\theta)}q_t\theta^2_x+\frac{2a(\theta)}{Z(\theta)}q_t\theta_{xx}+\frac{2a(\theta)}{Z(\theta)}q_{tx}\theta_x\\
&\left.+\frac{2a(\theta)}{Z(\theta)}q_x\theta_{tx}+\frac{2a(\theta)}{Z(\theta)}q\theta_{txx}\right)\theta_{tx}\dif x\\
\ge&-\int_{0}^{1}\frac{2a(\theta)}{Z(\theta)}q\theta_{txx}\theta_{tx}\dif x-CE^{\frac{1}{2}}(t)\mathcal{D}(t)\\
=&-\frac{a(\theta)}{Z(\theta)}q\theta^2_{tx}\big|_{0}^{1}+\int_{0}^{1}\theta^2_{tx}\left(\frac{a^\prime(\theta)}{Z(\theta)}\theta_xq-\frac{a(\theta)Z^\prime(\theta)}{Z^2(\theta)}\theta_xq+\frac{a(\theta)}{Z(\theta)}q_x\right)\dif x-CE^{\frac{1}{2}}(t)\mathcal{D}(t)\\
\ge&-CE^{\frac{1}{2}}(t)\mathcal{D}(t).
\end{align*}
For the right-hand side of the  equation \eqref{new-hu-3}, we have
\begin{align*}
&\int_{0}^{1}\left(\frac{2a(\theta)}{\tau   (\theta)}q^2v\right)_{tx}\theta_{tx}\dif x\\
=&\left\{\frac{2a^{\prime\prime}(\theta)}{\tau   (\theta)}\theta_t\theta_xq^2v-\frac{4a^\prime(\theta)\tau   ^\prime(\theta)}{\tau   ^2(\theta)}\theta_t\theta_xq^2v+\frac{2a^{\prime}(\theta)}{\tau   (\theta)}\theta_{tx}q^2v+\frac{2a^{\prime}(\theta)}{\tau   (\theta)}2\theta_tqq_xv \right.\\
&-\frac{2a(\theta)\tau   ^{\prime\prime}(\theta)}{\tau   ^2(\theta)}\theta_t\theta_xq^2v+\frac{4a(\theta){\tau   ^\prime(\theta)}^2}{\tau   ^3(\theta)}\theta_t\theta_xq^2v-\frac{2a(\theta)\tau   ^\prime(\theta)}{\tau   ^2(\theta)}\theta_{tx}q^2v-\frac{2a(\theta)\tau   ^\prime(\theta)}{\tau   ^2(\theta)}2\theta_tqq_xv\\
&-\frac{2a(\theta)\tau   ^\prime(\theta)}{\tau   ^2(\theta)}\theta_tq^2v_x+\frac{2a^{\prime}(\theta)}{\tau   (\theta)}2\theta_xqq_tv-\frac{2a(\theta)\tau   ^\prime(\theta)}{\tau   ^2(\theta)}2qq_t\theta_xv+\frac{4a(\theta)}{\tau   (\theta)}2q_xq_tv+\frac{2a(\theta)}{\tau   (\theta)}2qq_{tx}v\\
&\left.+\frac{2a(\theta)}{\tau   (\theta)}2qq_tv_x+\frac{2a^{\prime}(\theta)}{\tau   (\theta)}\theta_xq^2v_t-\frac{2a(\theta)\tau   ^\prime(\theta)}{\tau   ^2(\theta)}\theta_xq^2v_t+\frac{2a(\theta)}{\tau   (\theta)}q^2v_{tx}\right\}\theta_{tx}\dif x\\
\le& CE^{\frac{1}{2}}(t)\mathcal{D}(t),
\end{align*}
and
\begin{align*}
\int_{0}^{1}\left(\frac{S^2v}{\mu}\right)_{tx}\theta_{tx}\dif x
&=\int_{0}^{1}\frac{1}{\mu}\left(2S_tS_xv+2SS_{tx}v+2SS_tv_x+2SS_xv_t+S^2v_{tx}\right)\theta_{tx}\dif x
\le CE^{\frac{1}{2}}(t)\mathcal{D}(t).
\end{align*}
Combining the above estimates, we derive that
\begin{align}\label{u_{tx2}}
\frac{1}{2}\frac{\dif}{\dif t}\int_{0}^{1}e_{\theta}\theta^2_{tx}\dif x+\int_{0}^{1}\frac{R\theta}{v}u_{txx}\theta_{tx}\dif x+\int_{0}^{1}q_{txx}\theta_{tx}\dif x\le CE^{\frac{1}{2}}(t)\mathcal{D}(t).
\end{align}
Multiplying the equation $\eqref{tx}_4$ by $\frac{q_{tx}}{\kappa(\theta)}$, we get
$$\int_{0}^{1}\left(\tau   (\theta)q_t\right)_{tx}\frac{q_{tx}}{\kappa(\theta)}\dif x+\int_{0}^{1}(vq+\kappa(\theta)\theta_x)_{tx}\frac{q_{tx}}{\kappa(\theta)}\dif x=0.$$
We estimate each term of the above equation separately. For the left-hand side of the above equation,  
use the fact that 
$$
\left\|\frac{\tau   ^\prime(\theta) }{\kappa (\theta)} \theta_x q_{tt} q_{tx}\right\|_{L^1}\le C \|\theta_x\|_{L^\infty} \|\tau q_{tt}\|_{L^2} \|q_{tx}\|_{L^2}\le C E^{\frac{1}{2}}(t)\mathcal{D}(t),
$$
we have
\begin{align*}
&\int_{0}^{1}\left(\tau   (\theta)q_t\right)_{tx}\frac{q_{tx}}{\kappa(\theta)}\dif x\\	&=\int_{0}^{1}\frac{1}{\kappa(\theta)}\left(\tau   ^{\prime\prime}(\theta)\theta_t\theta_xq_t+\tau   ^\prime(\theta)\theta_{tx}q_t+\tau   ^\prime(\theta)\theta_tq_{tx}+\tau   ^\prime(\theta)\theta_{tx}q_t+\tau   ^\prime(\theta)\theta_xq_{tt}+\tau   (\theta)q_{ttx}\right)q_{tx}\dif x\\
&\ge\int_{0}^{1}Z(\theta)q_{ttx}q_{tx}\dif x-CE^{\frac{1}{2}}(t)\mathcal{D}(t)\\
&=\frac{1}{2}\frac{\dif}{\dif t}\int_{0}^{1}Z(\theta)q^2_{tx}\dif x-\int_{0}^{1}Z^\prime(\theta)\theta_tq^2_{tx}\dif x-CE^{\frac{1}{2}}(t)\mathcal{D}(t)\\
&\ge\frac{1}{2}\frac{\dif}{\dif t}\int_{0}^{1}Z(\theta)q^2_{tx}\dif x-CE^{\frac{1}{2}}(t)\mathcal{D}(t).
\end{align*}

Note that  the constitutive equation $\eqref{approximate}_4$ implies that $(\theta_x, \theta_{xt})\Big|_{\partial \Omega}=0$.   Then, it follows that
\begin{align*}
&\int_{0}^{1}(vq+\kappa(\theta)\theta_x)_{tx}\frac{q_{tx}}{\kappa(\theta)}\dif x\\
=&\frac{q_{tx}}{\kappa(\theta)}(vq+\kappa(\theta)\theta_x)_t\big|_{0}^{1}-\int_{0}^{1}\left(v_tq+vq_t+\kappa^\prime(\theta)\theta_t\theta_x+\kappa(\theta)\theta_{tx}\right)\left(\frac{1}{\kappa(\theta)}q_{txx}-\frac{\kappa^\prime(\theta)}{\kappa^2(\theta)}q_{tx}\theta_x\right)\dif x\\
\ge&-\frac{1}{\kappa(\theta)}q_{tx}\left(v_tq+vq_t+\kappa^\prime(\theta)\theta_t\theta_x\right)\big|_{0}^{1}+\int_{0}^{1}\left(\frac{v_tq+vq_t+\kappa^\prime(\theta)\theta_t\theta_x}{\kappa(\theta)}\right)_xq_{tx}\dif x-\int_{0}^{1}q_{txx}\theta_{tx}\dif x\\
&-CE^{\frac{1}{2}}(t)\mathcal{D}(t)\\
=&\int_{0}^{1}\left(\frac{v_tq+vq_t+\kappa^\prime(\theta)\theta_t\theta_x}{\kappa(\theta)}\right)_xq_{tx}\dif x-\int_{0}^{1}q_{txx}\theta_{tx}\dif x-CE^{\frac{1}{2}}(t)\mathcal{D}(t)\\
=&-\int_{0}^{1}q_{txx}\theta_{tx}\dif x+\int_{0}^{1}\left(\frac{v_{tx}q+v_tq_x+v_xq_t+vq_{tx}+\kappa^{\prime\prime}(\theta)\theta_t\theta^2_x+\kappa^\prime(\theta)\theta_{tx}\theta_x+\kappa^\prime(\theta)\theta_t\theta_{xx}}{\kappa(\theta)}\right.\\
&\left.-\frac{\kappa^\prime(\theta)}{\kappa^2(\theta)}\theta_x\left(v_tq+vq_t+\kappa^\prime(\theta)\theta_t\theta_x\right)\right)q_{tx}\dif x-CE^{\frac{1}{2}}(t)\mathcal{D}(t)\\
\ge&\int_{0}^{1}\frac{v}{\kappa(\theta)}q^2_{tx}\dif x-\int_{0}^{1}q_{txx}\theta_{tx}\dif x-CE^{\frac{1}{2}}(t)\mathcal{D}(t).
\end{align*}
Overall, we have
\begin{align}\label{u_{tx3}}
\frac{1}{2}\frac{\dif}{\dif t}\int_{0}^{1}Z(\theta)q^2_{tx}\dif x+\int_{0}^{1}\frac{v}{\kappa(\theta)}q^2_{tx}\dif x-\int_{0}^{1}q_{txx}\theta_{tx}\dif x
\le CE^{\frac{1}{2}}(t)\mathcal{D}(t).
\end{align}
Multiplying the equation $\eqref{tt}_5$ by $\frac{\theta}{\mu}S_{tx}$, we get
\begin{align*}
\int_{0}^{1}\frac{\theta}{\mu}S_{tx}\tau  S_{ttx}\dif x+\int_{0}^{1}\frac{\theta}{\mu}S_{tx}(vS)_{tx}\dif x-\int_{0}^{1}\theta S_{tx} u_{txx}\dif x+\frac{\tau  \epsilon}{\mu}\int_{0}^{1}(b(x)S_x)_{tx}\theta S_{tx}\dif x =0. 
\end{align*}
First, we have
\begin{align*}
\int_{0}^{1}\frac{\theta}{\mu}S_{tx}\tau  S_{ttx}\dif x
=\frac{1}{2}\frac{\dif}{\dif t}\int_{0}^{1}\frac{\tau  }{\mu}\theta S^2_{tx}\dif x-\frac{1}{2}\int_{0}^{1}\frac{\tau  }{\mu}\theta_tS^2_{tx}\dif x
\ge\frac{1}{2}\frac{\dif}{\dif t}\int_{0}^{1}\frac{\tau  }{\mu}\theta S^2_{tx}\dif x-CE^{\frac{1}{2}}(t)\mathcal{D}(t),
\end{align*}
and
\begin{align*}
\int_{0}^{1}\frac{\theta}{\mu}S_{tx}(vS)_{tx}\dif x
=\frac{1}{\mu}\int_{0}^{1}\theta\left(v_{tx}S+v_tS_x+v_xS_t+vS_{tx}\right)S_{tx}\dif x
\ge \frac{1}{2\mu}\int_{0}^{1}S^2_{tx}\dif x-CE^{\frac{1}{2}}(t)\mathcal{D}(t).
\end{align*}
Similarly, one has
\begin{align*}
\frac{\tau  \epsilon}{\mu}\int_{0}^{1}(b(x)S_x)_{tx}\theta S_{tx}\dif x 
&=\frac{\tau  \epsilon}{\mu}\int_{0}^{1}(2S_{tx}+b(x)S_{txx})\theta S_{tx}\dif x \\
&=\frac{2\tau  \epsilon}{\mu}\int_{0}^{1}\theta S^2_{tx}\dif x+\frac{\tau  \epsilon}{2\mu}b(x)\theta S^2_{tx}\big|_{0}^{1}-\frac{\tau  \epsilon}{2\mu}\int_{0}^{1}(2\theta+b(x)\theta_x)S^2_{tx}\dif x\\
&\ge \frac{\tau  \epsilon}{\mu}\int_{0}^{1}\theta S^2_{tx}\dif x+\frac{3\tau  \epsilon}{8\mu}(S^2_{tx}(t,0)+S^2_{tx}(t,1))
-CE^{\frac{1}{2}}(t)\mathcal{D}(t).
\end{align*}
Therefore, we derive that
\begin{align}\label{u_{tx4}}
\frac{1}{2}\frac{\dif}{\dif t}\int_{0}^{1}\frac{\tau  }{\mu}\theta S^2_{tx}\dif x+\frac{1}{4\mu}\int_{0}^{1}S^2_{tx}\dif x-\int_{0}^{1}\theta S_{tx} u_{txx}\dif x+\frac{3\tau  \epsilon}{8\mu}(S^2_{tx}(t,0)+S^2_{tx}(t,1))\nonumber\\
\le CE^{\frac{1}{2}}(t)\mathcal{D}(t).
\end{align}
Thus, combining the estimates \eqref{u_{tx1}}-\eqref{u_{tx4}}, we get
\begin{align*}
\frac{\dif}{\dif t} \int_0^1  \left( \frac{\theta}{2} u_{tx}^2+\frac{R\theta^2}{v^2} v_{tx}^2+\frac{e_\theta}{2} \theta_{tx}^2+\frac{Z(\theta)}{2} q_{tx}^2 +\frac{\tau \theta}{2\mu}  S_{tx}^2\right) \dif x 
+\int_0^1 \left(\frac{v}{\kappa (\theta)} q_{tx}^2+\frac{1}{4\mu} S_{tx}^2 \right)\dif x \\
+\frac{3\tau  \epsilon}{8\mu}(S^2_{tx}(t,0)+S^2_{tx}(t,1)) \le \frac{\dif }{\dif t} \int_0^1 u_{xx} \theta A_1(t,x) \dif x +CE^{\frac{1}{2}}(t)\mathcal{D}(t).
\end{align*}
Integrating the above inequality over $(0, t)$, using \eqref{new-hu2-1}-\eqref{new-hu4-3} and noticing that 
\begin{align*}
\int_0^1  \left( \frac{\theta}{2} u_{tx}^2+\frac{R\theta^2}{v^2} v_{tx}^2+\frac{e_\theta}{2} \theta_{tx}^2+\frac{Z(\theta)}{2} q_{tx}^2 +\frac{\tau \theta}{2\mu}  S_{tx}^2\right)  (t=0, x)\dif x \le C \|V_0^1\|_{H^1}^2 \le CE_0
\end{align*}
and
\begin{align*}
 \int_0^1 u_{xx} \theta A_1(t,x) \dif x\le C E^\frac{3}{2}(t).
 \end{align*}
The inequality  \eqref{ejtx} follows immediately.
\end{proof}

With the Lemmas \ref{le3.7}, \ref{le3.8} and \ref{le3.9}, using the equations \eqref{approximate}, we have the following corollary.

\begin{corollary}\label{co3.10}
There exists some constant $C$ such that
\begin{align}\label{ejhs}
&\int_{0}^{t}\int_{0}^{1}(u^2_{xx}+u^2_{tx}+u^2_{tt}+v^2_{tx}+v^2_{tt}+\theta^2_{xx}+\theta^2_{tx}+\theta^2_{tt})\dif x\dif t\nonumber\\
&\le C\left(\epsilon\int_{0}^{t}\int_{0}^{1}S_{xx}^2\dif x\dif t+E_0+E^\frac{1}{2}\int_{0}^{t}\mathcal{D}(s)\dif s+E^\frac{3}{2}(t)\right).
\end{align}
\end{corollary}

\begin{proof}
From $\eqref{approximate}_5$, we can get 
\begin{align*}
\int_{0}^{t}\int_{0}^{1} (u^2_{tx}+v^2_{tt})\dif x\dif t &\le C \int_{0}^{t}\int_{0}^{1}(\tau^2(S_{tt}^2+\epsilon S^2_{tx})+S_t^2+v_t^2)\dif x\dif t\\
&\le C\left(E_0+E^\frac{1}{2}(t)\int_{0}^{t}\mathcal{D}(s)\dif s+E^\frac{3}{2}(t)\right).
\end{align*}
and
\begin{align*}
\int_{0}^{t}\int_{0}^{1} (u^2_{xx}+v^2_{tx})\dif x\dif t & \le C \int_{0}^{t}\int_{0}^{1}(\tau^2(S_{tx}^2+\epsilon S^2_{xx})+S_x^2+v_x^2)\dif x\dif t\\
&\le C\left(\epsilon\int_{0}^{t}\int_{0}^{1}S_{xx}^2\dif x\dif t+E_0+E^\frac{1}{2}\int_{0}^{t}\mathcal{D}(s)\dif s+E^\frac{3}{2}(t)\right).
\end{align*}

In view of  $\eqref{approximate}_4$, it follows that
\begin{align*}
\int_{0}^{t}\int_{0}^{1}\theta^2_{tx}\dif x\dif t\le C \int_{0}^{t}\int_{0}^{1}(\tau^2 (q^2_{tt}+\theta_t^2 q_t^2)+q_t^2+v_t^2)\dif x\dif t\\
\le C\left(E_0+E^\frac{1}{2}(t)\int_{0}^{t}\mathcal{D}(s)\dif s+E^\frac{3}{2}(t)\right)
\end{align*}
and
\begin{align*}
\int_{0}^{t}\int_{0}^{1}\theta^2_{xx}\dif x\dif t\le C \int_{0}^{t}\int_{0}^{1}(\tau^2 (q^2_{tx}+\theta_x^2 q_t^2)+q_x^2+v_x^2)\dif x\dif t\\
\le C\left(E_0+E^\frac{1}{2}(t)\int_{0}^{t}\mathcal{D}(s)\dif s+E^\frac{3}{2}(t)\right).
\end{align*}
Similarly, using $\eqref{approximate}_3$, we have
\begin{align*}
\int_{0}^{t}\int_{0}^{1}\theta^2_{tt}\dif x\dif t
&\le\int_{0}^{t}\int_{0}^{1}(u^2_{tx}+q^2_{tx})\dif x\dif t+C E^\frac{1}{2}(t)\int_{0}^{t}\mathcal{D}(s)\dif s \\
&\le C\left(E_0+E^\frac{1}{2}(t)\int_{0}^{t}\mathcal{D}(s)\dif s+E^\frac{3}{2}(t)\right).
\end{align*}
On the other hand, using the equation $\eqref{t}_2$ and the above estimates, one has
\begin{align*}
\int_0^t \int_0^1 u_{tt}^2 \dif x \dif t &\le C( \int_0^t \int_0^1 (v_{tx}^2+\theta_{tx}^2+S_{tx}^2)\dif x \dif t+E^\frac{1}{2}(t) \int_0^t \mathcal D(s) \dif s)\\
&\le C\left(\epsilon\int_{0}^{t}\int_{0}^{1}S_{xx}^2\dif x\dif t+E_0+E^\frac{1}{2}\int_{0}^{t}\mathcal{D}(s)\dif s+E^\frac{3}{2}(t)\right).
 \end{align*}
 Thus, the proof of this corollary is finished. 
\end{proof}

%Estimates of "xx"
We are now ready to estimate the second-order derivative of the solution with respect to \( x \).
Unlike the estimates in the previous lemmas, integration by parts in this case introduces boundary terms. To handle these terms, we must carefully exploit the structure of the system, ensuring proper control and ultimately obtaining the desired estimate.

\begin{lemma}\label{le3.11}
There exists some constant $C$ such that
\begin{align}\label{ejxx}
&\int_{0}^{1}(u^2_{xx}+v^2_{xx}+\theta^2_{xx}+\tau   q^2_{xx}+\tau   S^2_{xx})\dif x+\int_{0}^{t}\int_{0}^{1}(q^2_{xx}+S^2_{xx})\dif x\dif t \nonumber\\
&\qquad+\frac{3\tau  \epsilon}{8\mu}\int_{0}^{t}(S^2_{xx}(t,0)+S^2_{xx}(t,1))\dif t
\le C\int_{0}^{t}\left( \SUM i 0 1 (u^2_{xx}+\theta^2_{xx}+u^2_{tx}+q^2_{xx} )(t, i)\right)\dif t  \nonumber \\
&\qquad\qquad\qquad+C\left( E_0+E^\frac{1}{2}(t)\int_{0}^{t}\mathcal{D}(s)\dif s+E^\frac{3}{2}(t) \right). 
\end{align}
\end{lemma}

\begin{proof}
Taking derivatives with respect to $x$ twice to the equations \eqref{approximate}, we get
\begin{align}\label{xx}
\begin{cases}
v_{txx}=u_{xxx},\\
u_{txx}+{p(v,\theta)}_{xxx}=S_{xxx},\\
(e_{\theta}\theta_t)_{xx}-(\frac{2a(\theta)}{Z(\theta)}q\theta_x)_{xx}+(\frac{R\theta u_x}{v})_{xx}+q_{xxx}=(\frac{2a(\theta)}{\tau   (\theta)}q^2v)_{xx}+(\frac{S^2v}{\mu})_{xx},\\
\left(\tau   (\theta)q_t\right)_{xx}+(vq)_{xx}+(\kappa(\theta)\theta_x)_{xx}=0,\\
\tau  S_{txx}+(vS)_{xx}+\tau  \epsilon(b(x)S_x)_{xx}=\mu u_{xxx}.
\end{cases}
\end{align}
Multiplying the equation $\eqref{xx}_2$ by $\theta u_{xx}$, we get
$$\int_{0}^{1} \theta u_{txx}u_{xx}\dif x+\int_{0}^{1} \theta \left(p(v,\theta)_{xx}-S_{xx}\right)_xu_{xx}\dif x=0.$$
For the left-hand side of the above equation, we have
\begin{align*}
\int_{0}^{1} \theta u_{txx}u_{xx}\dif x=\frac{1}{2}\frac{\dif}{\dif t}\int_{0}^{1}\theta u^2_{xx}\dif x-\frac{1}{2}\int_{0}^{1}\theta_tu^2_{xx}\dif x
\ge\frac{1}{2}\frac{\dif}{\dif t}\int_{0}^{1}\theta u^2_{xx}\dif x-CE^{\frac{1}{2}}(t)\mathcal{D}(t),
\end{align*}
and
\begin{align*}
&\int_{0}^{1} \theta \left(p(v,\theta)_{xx}-S_{xx}\right)_xu_{xx}\dif x\\
=&\theta \left(p(v,\theta)_{xx}-S_{xx}\right)u_{xx}\big|_{0}^{1}-\int_{0}^{1} (\theta u_{xx})_x\left(p(v,\theta)_{xx}-S_{xx}\right)\dif x\\
=&-\theta u_{tx}u_{xx}\big|_{0}^{1}-\int_{0}^{1}(\theta_xu_{xx}+\theta u_{xxx})\left(p_{vv}v^2_x+2p_{v\theta}v_x\theta_x+p_vv_{xx}+p_\theta\theta_{xx}-S_{xx}\right)\dif x\\
\ge&-\theta u_{tx}u_{xx}\big|_{0}^{1}-\int_{0}^{1}\theta v_{txx}\left(p_{vv}v^2_x+2p_{v\theta}v_x\theta_x+p_vv_{xx}\right)\dif x-\int_{0}^{1}\frac{R\theta}{v}\theta_{xx}u_{xxx}\dif x\\
&+\int_{0}^{1}\theta S_{xx}u_{xxx}\dif x-CE^{\frac{1}{2}}(t)\mathcal{D}(t)\\
=&-\theta u_{tx}u_{xx}\big|_{0}^{1}-\frac{\dif}{\dif t}\int_{0}^{1}\theta v_{xx}A_2(t,x)\dif x+\frac{1}{2}\frac{\dif}{\dif t}\int_{0}^{1}\frac{R\theta^2}{v^2}v^2_{xx}\dif x+\int_{0}^{1}(\theta A_2(t,x))_tv_{xx}\dif x\\
&+\frac{1}{2}\int_{0}^{1}(\theta_tp_v+\theta p_{vv}v_t+\theta p_{v\theta}\theta_t)v^2_{xx}\dif x-\int_{0}^{1}\frac{R\theta}{v}\theta_{xx}u_{xxx}\dif x+\int_{0}^{1}\theta S_{xx}u_{xxx}\dif x-CE^{\frac{1}{2}}(t)\mathcal{D}(t).
\end{align*}
Here, $A_2(t,x):=p_{vv}v^2_x+2p_{v\theta}v_x\theta_x$. It's easy to show that
\begin{align*}
\frac{1}{2}\int_{0}^{1}(\theta_tp_v+\theta p_{vv}v_t+\theta p_{v\theta}\theta_t)v^2_{xx}\dif x \ge - CE^{\frac{1}{2}}(t)\mathcal{D}(t),
\end{align*}
and
\begin{align*}
&\int_{0}^{1}(\theta A_2(t,x))_tv_{xx}\dif x\\
=&\int_{0}^{1}\left\{\theta_t(p_{vv}v^2_x+2p_{v\theta}v_x\theta_t)+\theta(p_{vvv}v^2_xv_t+p_{vv\theta}v^2_x\theta_t\right.\\
&\left.+2p_{vv\theta}v_xv_t\theta_t+p_{vv}2v_xv_{tx}+2p_{v\theta}\theta_tv_{tx}+2p_{v\theta}\theta_{tt}v_x)\right\}v_{xx}\dif x\\
\ge&-CE^{\frac{1}{2}}(t)\mathcal{D}(t).
\end{align*}
Combining the above estimates, it yields
\begin{align*}
&\int_{0}^{1} \theta \left(p(v,\theta)_{xx}-S_{xx}\right)_xu_{xx}\dif x\ge-\theta u_{tx}u_{xx}\big|_{0}^{1}-\frac{\dif}{\dif t}\int_{0}^{1}\theta v_{xx}A_2(t,x)\dif x\\
&\qquad\qquad\qquad+\frac{1}{2}\frac{\dif}{\dif t}\int_{0}^{1}\frac{R\theta^2}{v^2}v^2_{xx}\dif x-\int_{0}^{1}\frac{R\theta}{v}\theta_{xx}u_{xxx}\dif x+\int_{0}^{1}\theta S_{xx}u_{xxx}\dif x-CE^{\frac{1}{2}}(t)\mathcal{D}(t).
\end{align*}
Thus, we have
\begin{align}\label{u_{xx1}}
\frac{1}{2}\frac{\dif}{\dif t}\int_{0}^{1}(\theta u^2_{xx}+\frac{R\theta^2}{v^2}v^2_{xx})\dif x-\int_{0}^{1}\frac{R\theta}{v}\theta_{xx}u_{xxx}\dif x+\int_{0}^{1}\theta S_{xx}u_{xxx}\dif x\nonumber\\
\le \theta u_{tx}u_{xx}\big|_{0}^{1}+\frac{\dif}{\dif t}\int_{0}^{1}\theta v_{xx}A_2(t,x)\dif x+CE^{\frac{1}{2}}(t)\mathcal{D}(t).
\end{align}
Multiplying the equation $\eqref{xx}_3$ by $\theta_{xx}$, we get
\begin{align*}
&\int_{0}^{1}(e_{\theta}\theta_t)_{xx}\theta_{xx}\dif x-\int_{0}^{1}\left(\frac{2a(\theta)}{Z(\theta)}q\theta_x\right)_{xx}\theta_{xx}\dif x+\int_{0}^{1}\left(\frac{R\theta u_x}{v}\right)_{xx}\theta_{xx}\dif x+\int_{0}^{1}q_{xxx}\theta_{xx}\dif x\\
&=\int_{0}^{1}\left(\frac{2a(\theta)}{\tau   (\theta)}q^2v\right)_{xx}\theta_{xx}\dif x+\int_{0}^{1}\left(\frac{S^2v}{\mu}\right)_{xx}\theta_{xx}\dif x.
\end{align*}
First, we have
\begin{align*}
&\int_{0}^{1}(e_{\theta}\theta_t)_{xx}\theta_{xx}\dif x\\
=&\int_{0}^{1}(e_{\theta\theta\theta}\theta^2_x\theta_t+2e_{\theta\theta q}\theta_xq_x\theta_t+e_{\theta qq}\theta_tq^2_x+2e_{\theta\theta}\theta_x\theta_{tx}+e_{\theta\theta}\theta_t\theta_{xx}+e_{\theta q}\theta_tq_{xx}\\
&+2e_{\theta q}\theta_{tx}q_x+e_{\theta}\theta_{txx})\theta_{xx}\dif x\\
&\ge \int_{0}^{1}e_{\theta}\theta_{txx}\theta_{xx}\dif x-CE^{\frac{1}{2}}(t)\mathcal{D}(t)\\
&=\frac{1}{2}\frac{\dif}{\dif t}\int_{0}^{1}e_{\theta}\theta^2_{xx}\dif x-\frac{1}{2}\int_{0}^{1}\left(e_{\theta\theta}\theta_t+e_{\theta q}q_t\right)\theta^2_{xx}\dif x-CE^{\frac{1}{2}}(t)\mathcal{D}(t)\\
&\ge\frac{1}{2}\frac{\dif}{\dif t}\int_{0}^{1}e_{\theta}\theta^2_{xx}\dif x-CE^{\frac{1}{2}}(t)\mathcal{D}(t).
\end{align*}
Second, one has
\begin{align*}
&-\int_{0}^{1}\left(\frac{2a(\theta)}{Z(\theta)}q\theta_x\right)_{xx}\theta_{xx}\dif x\\
=&-\left(\frac{2a^{\prime\prime}(\theta) Z(\theta)+4a^\prime(\theta) Z^\prime(\theta)}{Z^2(\theta)}q\theta^3_x+\frac{4a^{\prime}(\theta)}{Z(\theta)}q_x\theta^2_x+\frac{6a^{\prime}(\theta)}{Z(\theta)}q\theta_x\theta_{xx}-\frac{2a(\theta)(Z^{\prime\prime}(\theta)-2Z^\prime(\theta)^2)}{Z^3(\theta)}q\theta^3_x \right.\\
&\left.-\frac{4a(\theta)Z^\prime(\theta)}{Z^2(\theta)}q_x\theta^2_x-\frac{6a(\theta)Z^\prime(\theta)}{Z^2(\theta)}q\theta_x\theta_{xx}+\frac{2a(\theta)}{Z(\theta)}q_{xx}\theta_x+\frac{4a(\theta)}{Z(\theta)}q_x\theta_{xx}+\frac{2a(\theta)}{Z(\theta)}q\theta_{xxx}\right)\theta_{xx}\dif x\\
\ge&-\int_{0}^{1}\frac{2a(\theta)}{Z(\theta)}q\theta_{xxx}\theta_{xx}\dif x-CE^{\frac{1}{2}}(t)\mathcal{D}(t)\\
=&-\frac{a(\theta)}{Z(\theta)}q\theta^2_{xx}\big|_{0}^{1}+\int_{0}^{1}\theta^2_{xx}\left(\frac{a^\prime(\theta)}{Z(\theta)}\theta_xq-\frac{a(\theta)Z^\prime(\theta)}{Z^2(\theta)}\theta_xq+\frac{a(\theta)}{Z(\theta)}q_x\right)\dif x-CE^{\frac{1}{2}}(t)\mathcal{D}(t)\\
\ge&-CE^{\frac{1}{2}}(t)\mathcal{D}(t),
\end{align*}
and
\begin{align*}
&\int_{0}^{1}\left(\frac{R\theta u_x}{v}\right)_{xx}\theta_{xx}\dif x\\
&=\int_{0}^{1}R\left(\frac{\theta_{xx}u_x}{v}+\frac{2\theta_xu_{xx}}{v}-\frac{2\theta_xu_xv_x}{v^2}-\frac{2\theta u_{xx}v_x}{v^2}-\frac{\theta u_xv_{xx}}{v^2}+\frac{2\theta u_xv^2_x}{v^3}+\frac{\theta u_{xxx}}{v}\right)\theta_{xx}\dif x\\
&\ge\int_{0}^{1}\frac{R\theta}{v}u_{xxx}\theta_{xx}\dif x-CE^{\frac{1}{2}}(t)\mathcal{D}(t).
\end{align*}
Similarly, we have
\begin{align*}
&\int_{0}^{1}\left(\frac{2a(\theta)}{\tau   (\theta)}q^2v\right)_{xx}\theta_{xx}\dif x\\
=&\left\{\frac{2a^{\prime\prime}(\theta) \tau   (\theta)-4a^\prime(\theta)\tau   ^\prime(\theta)}{\tau   ^2(\theta)}\theta^2_xq^2v+\frac{2a^{\prime}(\theta)}{\tau   (\theta)}\theta_{xx}q^2v+\frac{4a^{\prime}(\theta)}{\tau   (\theta)}\theta_x2qq_xv+\frac{4a^{\prime}(\theta)}{\tau   (\theta)}\theta_xq^2v_x\right.\\
&-\frac{2a(\theta)(\tau   ^{\prime\prime}(\theta)\tau   (\theta)-2\tau   ^\prime(\theta)^2)}{\tau   ^3(\theta)}\theta^2_xq^2v-\frac{2a(\theta)\tau   ^\prime(\theta)}{\tau   ^2(\theta)}q^2(\theta_{xx}v+\theta_x v_x)-\frac{4a(\theta)\tau   ^\prime(\theta)}{\tau   ^2(\theta)}2qq_x\theta_xv\\
&\left.+\frac{2a(\theta)}{\tau   (\theta)}2qq_{xx}v+\frac{2a(\theta)}{\tau   (\theta)}2q^2_xv+\frac{4a(\theta)}{\tau   (\theta)}2qq_xv_x-\frac{4a(\theta)\tau   ^\prime(\theta)}{\tau   ^2(\theta)}\theta_xq^2v_x+\frac{2a(\theta)}{\tau   (\theta)}q^2v_{xx}\right\}\theta_{xx}\dif x\\
\le& CE^{\frac{1}{2}}(t)\mathcal{D}(t),
\end{align*}
and
\begin{align*}
\int_{0}^{1}(\frac{S^2v}{\mu})_{xx}\theta_{xx}\dif x
=\int_{0}^{1}\frac{1}{\mu}\left(2S^2_xv+2SS_{xx}v+4SS_xv_x+S^2v_{xx}\right)\theta_{xx}\dif x\le CE^{\frac{1}{2}}(t)\mathcal{D}(t).
\end{align*}
Therefore, we derive that
\begin{align}\label{u_{xx2}}
\frac{1}{2}\frac{\dif}{\dif t}\int_{0}^{1}e_{\theta}\theta^2_{xx}\dif x+\int_{0}^{1}\frac{R\theta}{v}u_{xxx}\theta_{xx}\dif x+\int_{0}^{1}q_{xxx}\theta_{xx}\dif x
\le CE^{\frac{1}{2}}(t)\mathcal{D}(t).
\end{align}
Multiplying the equation $\eqref{xx}_4$ by $\frac{q_{xx}}{\kappa(\theta)}$, we get
$$\int_{0}^{1}\left(\tau   (\theta)q_t\right)_{xx}\frac{q_{xx}}{\kappa(\theta)}\dif x+\int_{0}^{1}(vq)_{xx}\frac{q_{xx}}{\kappa(\theta)}\dif x+\int_{0}^{1}(\kappa(\theta)\theta_x)_{xx}\frac{q_{xx}}{\kappa(\theta)}\dif x=0.$$
First, we have
\begin{align*}
\int_{0}^{1}\left(\tau   (\theta)q_t\right)_{xx}\frac{q_{xx}}{\kappa(\theta)}\dif x
&=\int_{0}^{1}\frac{1}{\kappa(\theta)}\left(\tau   ^{\prime\prime}(\theta)\theta^2_xq_t+\tau   ^\prime(\theta)\theta_{xx}q_t+2\tau   ^\prime(\theta)\theta_xq_{xt}+\tau   (\theta)q_{txx}\right)q_{xx}\dif x\\
&\ge\int_{0}^{1}Z(\theta)q_{txx}q_{xx}\dif x-CE^{\frac{1}{2}}(t)\mathcal{D}(t)\\
&=\frac{1}{2}\frac{\dif}{\dif t}\int_{0}^{1}Z(\theta)q^2_{xx}\dif x-\frac{1}{2}\int_{0}^{1}Z^\prime(\theta)\theta_tq^2_{xx}\dif x-CE^{\frac{1}{2}}(t)\mathcal{D}(t)\\
&\ge\frac{1}{2}\frac{\dif}{\dif t}\int_{0}^{1}Z(\theta)q^2_{xx}\dif x-CE^{\frac{1}{2}}(t)\mathcal{D}(t),
\end{align*}

\begin{align*}
\int_{0}^{1}(vq)_{xx}\frac{q_{xx}}{\kappa(\theta)}\dif x
=\int_{0}^{1}\frac{1}{\kappa(\theta)}\left(v_{xx}q+2v_xq_x+vq_{xx}\right)q_{xx}\dif x
\ge\int_{0}^{1}\frac{v}{\kappa(\theta)}q^2_{xx}\dif x-CE^{\frac{1}{2}}(t)\mathcal{D}(t),
\end{align*}
and
\begin{align*}
\int_{0}^{1}(\kappa(\theta)\theta_x)_{xx}\frac{q_{xx}}{\kappa(\theta)}\dif x
&=\int_{0}^{1}\frac{1}{\kappa(\theta)}\left(\kappa^{\prime\prime}(\theta)\theta^3_x+3\kappa^\prime(\theta)\theta_x\theta_{xx}+\kappa(\theta)\theta_{xxx}\right)q_{xx}\dif x\\
&\ge\int_{0}^{1}\theta_{xxx}q_{xx}\dif x-CE^{\frac{1}{2}}(t)\mathcal{D}(t)\\
&=\theta_{xx}q_{xx}\big|_{0}^{1}-\int_{0}^{1}\theta_{xx}q_{xxx}\dif x-CE^{\frac{1}{2}}(t)\mathcal{D}(t).
\end{align*}
Therefore, we derive that
\begin{align}\label{u_{xx3}}
\frac{1}{2}\frac{\dif}{\dif t}\int_{0}^{1}Z(\theta)q^2_{xx}\dif x+\int_{0}^{1}\frac{v}{\kappa(\theta)}q^2_{xx}\dif x+\theta_{xx}q_{xx}\big|_{0}^{1}-\int_{0}^{1}\theta_{xx}q_{xxx}\dif x
\le CE^{\frac{1}{2}}(t)\mathcal{D}(t).
\end{align}
Multiplying the equation $\eqref{xx}_5$ by $\frac{\theta}{\mu}S_{xx}$, we get
\begin{align*}
\int_{0}^{1}\frac{\theta}{\mu}S_{xx}\tau  S_{txx}\dif x+\int_{0}^{1}\frac{\theta}{\mu}S_{xx}(vS)_{xx}\dif x-\int_{0}^{1}\theta S_{xx} u_{xxx}\dif x
+\frac{\tau  \epsilon}{\mu}\int_{0}^{1}(b(x)S_x)_{xx}\theta S_{xx}\dif x
=0.
\end{align*}
Similarly, we have
\begin{align*}
\int_{0}^{1}\frac{\theta}{\mu}S_{xx}\tau  S_{txx}\dif x
=\frac{\tau  }{2\mu}\frac{\dif}{\dif t}\int_{0}^{1}\theta S^2_{xx}\dif x-\frac{\tau  }{2\mu}\int_{0}^{1}\theta_t S^2_{xx}\dif x
\ge\frac{\tau  }{2\mu}\frac{\dif}{\dif t}\int_{0}^{1}\theta S^2_{xx}\dif x-CE^{\frac{1}{2}}(t)\mathcal{D}(t),
\end{align*}

\begin{align*}
\int_{0}^{1}\frac{\theta}{\mu}S_{xx}(vS)_{xx}\dif x
=\int_{0}^{1}\frac{\theta}{\mu}(v_{xx}S+2v_xS_x+vS_{xx})S_{xx}\dif x
\ge\int_{0}^{1}\frac{1}{2\mu}S^2_{xx}\dif x-CE^{\frac{1}{2}}(t)\mathcal{D}(t)
\end{align*}
and
\begin{align*}
\frac{\tau  \epsilon}{\mu}\int_{0}^{1}(b(x)S_x)_{xx}\theta S_{xx}\dif x
&=\frac{\tau  \epsilon}{\mu}\int_{0}^{1}(4S_{xx}+b(x)S_{xxx})\theta S_{xx}\dif x\\
&=\frac{4\tau  \epsilon}{\mu}\int_{0}^{1}\theta S_{xx}^2\dif x+\frac{\tau  \epsilon}{2\mu}b(x)\theta S^2_{xx}\big|_{0}^{1}-\frac{\tau  \epsilon}{2\mu}\int_{0}^{1}(2\theta+b(x)\theta_x)S^2_{xx}\dif x\\
&\ge \frac{3\tau  \epsilon}{\mu}\int_{0}^{1}\theta S^2_{xx}\dif x+\frac{3\tau  \epsilon}{8\mu}(S^2_{xx}(t,0)+S^2_{xx}(t,1))-CE^{\frac{1}{2}}(t)\mathcal{D}(t).
\end{align*}
Therefore, we derive that
\begin{align}\label{u_{xx4}}
\frac{\tau  }{2\mu}\frac{\dif}{\dif t}\int_{0}^{1}\theta S^2_{xx}\dif x+\int_{0}^{1}\frac{1}{4\mu}S^2_{xx}\dif x-\int_{0}^{1}\theta S_{xx} u_{xxx}\dif x+\frac{3\tau  \epsilon}{8\mu}(S^2_{xx}(t,0)+S^2_{xx}(t,1))
\le CE^{\frac{1}{2}}(t)\mathcal{D}(t) .
\end{align}
Thus, combining the estimates \eqref{u_{xx1}}-\eqref{u_{xx4}} , we get
\begin{align*}
\frac{\dif}{\dif t} \int_0^1  \left( \frac{\theta}{2} u_{xx}^2+\frac{R\theta^2}{v^2} v_{xx}^2+\frac{e_\theta}{2} \theta_{xx}^2+\frac{Z(\theta)}{2} q_{xx}^2 +\frac{\tau \theta}{2\mu}  S_{xx}^2\right) \dif x 
+\int_0^1 \left(\frac{v}{\kappa (\theta)} q_{xx}^2+\frac{1}{4\mu} S_{xx}^2 \right)\dif x \\
+\frac{3\tau  \epsilon}{8\mu}(S^2_{xx}(t,0)+S^2_{xx}(t,1)) \le \theta u_{tx} u_{xx}\Big |_0^1-\theta_{xx} q_{xx} \Big|_0^1+\frac{\dif }{\dif t} \int_0^1 v_{xx} \theta A_2(t,x) \dif x +CE^{\frac{1}{2}}(t)\mathcal{D}(t).
\end{align*}
Integrating the above inequality over $(0, t)$, and noting that
\begin{align*}
\int_0^1 v_{xx} \theta A_2(t,x) \dif x  \le CE^\frac{3}{2}(t)
\end{align*}
and
\begin{align*}
\theta u_{tx} u_{xx}\Big |_0^1-\theta_{xx} q_{xx} \Big|_0^1 \le C\left(   \SUM i 0 1 (u^2_{xx}+\theta^2_{xx}+u^2_{tx}+q^2_{xx} )(t, i)\right),
\end{align*}
the inequality   \eqref{ejxx} follows  immediately.
\end{proof}

%Boundary estimates

Now, we prove the following key lemmas which deal with the boundary term on the right-hand-side of \eqref{ejxx}.
\begin{lemma}\label{le-hu-2}
There exists some constant $C$ such that
\begin{align}
&\int_{0}^{t}(u^2_{xx}(t,1)+u^2_{xx}(t,0)+\theta^2_{xx}(t,1)+\theta^2_{xx}(t,0))\dif t
\le \eta \left(\int_{0}^{1}\theta^2_{xx}\dif x +\int_0^t\int_0^1 (q_{xx}^2+S_{xx}^2)\dif x \dif t\right)  \nonumber
\\
&+C(\eta) \epsilon \int_{0}^{t}\int_{0}^{1}S^2_{xx}\dif x \dif t+\frac{\tau  \epsilon}{8\mu} \eta \int_{0}^{t}(S^2_{xx}(t,0)+S^2_{xx}(t,1))\dif t+\frac{2\tau  \epsilon}{\mu} C(\eta)(u^2_{xx}(t,0)+u^2_{xx}(t,1))\nonumber\\
&+C\left(E_0+E^\frac{1}{2}(t)\int_{0}^{t}\mathcal{D}(s)\dif s+E^\frac{3}{2}(t) \right),  \label{new-hu-9}
\end{align}
where $\eta$ is any small  constant to be chosen later.
\end{lemma}
	
\begin{proof}
For simplicity, denote $\overline{b}(x):=-b(x)$, with $\overline{b}(0)=1, \overline{b}(1)=-1$.
		
Multiplying the equation $\eqref{xx}_4$ by $\frac{\tau  b(x)}{\mu\tau   (\theta)\theta}\theta_{xx}$ and integrating the result,  we get
\begin{align}\label{new-hu-6}
\int_{0}^{1}\frac{\tau  b(x)}{\mu\tau   (\theta)\theta}\left(\tau   (\theta)q_t\right)_{xx}\theta_{xx}\dif x+\int_{0}^{1}\frac{\tau  b(x)}{\mu\tau   (\theta)\theta}(vq)_{xx}\theta_{xx}\dif x+\int_{0}^{1}\frac{\tau  b(x)}{\mu\tau   (\theta)\theta}(\kappa(\theta)\theta_x)_{xx}\theta_{xx}\dif x=0.
\end{align}
For the first term of \eqref{new-hu-6}, we have
\begin{align*}
&\int_{0}^{1}\frac{\tau  b(x)}{\mu\tau   (\theta)\theta}\left(\tau   (\theta)q_t\right)_{xx}\theta_{xx}\dif x\\
&=\int_{0}^{1}\frac{\tau  b(x)}{\mu\tau   (\theta)\theta}\left(\tau   ^{\prime\prime}(\theta)\theta^2_xq_t+\tau   ^\prime(\theta)\theta_{xx}q_t+2\tau   ^\prime(\theta)\theta_xq_{xt}+\tau   (\theta)q_{xxt}\right)\theta_{xx}\dif x\\
&\ge \int_{0}^{1}\frac{\tau  b(x)}{\mu\theta}q_{txx}\theta_{xx}\dif x-CE^{\frac{1}{2}}(t)\mathcal{D}(t).
\end{align*}
Note that, by using $\eqref{tx}_3$, we have
$$q_{txx}=(\frac{2a(\theta)}{\tau   (\theta)}q^2v)_{tx}+(\frac{S^2v}{\mu})_{tx}-(e_{\theta}\theta_t)_{tx}+(\frac{2a(\theta)}{Z(\theta)}q\theta_x)_{tx}-(\frac{R\theta u_x}{v})_{tx},$$
which gives
\begin{align*}
&\int_{0}^{1}\frac{\tau  b(x)}{\mu\theta}q_{txx}\theta_{xx}\dif x\\
\ge&\int_{0}^{1}\frac{\tau  b(x)}{\mu\theta}\left(-e_{\theta}\theta_{ttx}+\frac{2a(\theta)}{Z(\theta)}q\theta_{txx}-\frac{R\theta}{v}u_{txx}\right)\theta_{xx}\dif x-CE^{\frac{1}{2}}(t)\mathcal{D}(t)\\
\ge&-\frac{\dif}{\dif t}\int_{0}^{1}\frac{\tau  b(x)}{\mu\theta}\left(e_{\theta}\theta_{tx}-\frac{a(\theta)}{Z(\theta)}q\theta_{xx}+\frac{R\theta}{v}u_{xx}\right)\theta_{xx}\dif x\\
&+\int_{0}^{1}\left\{-\frac{\tau  b(x)\theta_t}{\mu\theta^2}\left(e_{\theta}\theta_{tx}-\frac{a(\theta)}{Z(\theta)}q\theta_{xx}+\frac{R\theta}{v}u_{xx}\right)\theta_{xx}\right.\\
&+\frac{\tau  b(x)}{\mu\theta}\left(e_{\theta\theta}\theta_t\theta_{tx}+e_{\theta q}q_t\theta_{tx}-\frac{a^\prime(\theta)Z(\theta)-a(\theta)Z^\prime(\theta)}{Z^2(\theta)}q\theta_t\theta_{xx} -\frac{a(\theta)}{Z(\theta)}q_t\theta_{xx}+\frac{R(v\theta_t-\theta v_t}{v^2}u_{xx}\right)\theta_{xx}\\
&\left.+\frac{\tau  b(x)}{\mu\theta}e_\theta\theta_{tx}\theta_{txx}+\frac{R\tau  b(x)}{v\mu}u_{xx}\theta_{txx}\right\}\dif x-CE^{\frac{1}{2}}(t)\mathcal{D}(t)\\
\ge&-\frac{\dif}{\dif t}\int_{0}^{1}\frac{\tau  b(x)}{\mu\theta}\left(e_{\theta}\theta_{tx}-\frac{a(\theta)}{Z(\theta)}q\theta_{xx}+\frac{R\theta}{v}u_{xx}\right)\theta_{xx}\dif x\\
&+\int_{0}^{1}\frac{\tau  b(x)}{\mu\theta}e_\theta\theta_{tx}\theta_{txx}\dif x+\int_{0}^{1}\frac{R\tau  b(x)}{v\mu}u_{xx}\theta_{txx}\dif x-CE^{\frac{1}{2}}(t)\mathcal{D}(t),
\end{align*}
where
\begin{align*}
&\int_{0}^{1}\frac{\tau  b(x)}{\mu\theta}e_\theta\theta_{tx}\theta_{txx}\dif x\\
&=\frac{1}{2}\frac{\tau  b(x)}{\mu\theta}e_\theta\theta^2_{tx}\big|_{0}^{1}-\int_{0}^{1}\frac{\tau  }{2\mu}\left(\frac{b(x)e_\theta}{\theta}\right)_x\theta^2_{tx}\dif x\\
&\ge-\int_{0}^{1}\frac{\tau  }{2\mu}\left(\frac{2e_\theta}{\theta}-\frac{\theta_xb(x)e_\theta}{\theta^2}+\frac{b(x)e_{\theta\theta}\theta_x}{\theta}+\frac{b(x)e_{q\theta}q_x}{\theta}\right)\theta^2_{tx}\dif x\\
&\ge -\int_{0}^{1}\frac{\tau  e_\theta}{\theta\mu}\theta^2_{tx}\dif x-CE^{\frac{1}{2}}(t)\mathcal{D}(t).
\end{align*}
For the second term of \eqref{new-hu-6}, we get
\begin{align*}
\int_{0}^{1}\frac{\tau  b(x)}{\mu\tau   (\theta)\theta}(vq)_{xx}\theta_{xx}\dif x
&=\int_{0}^{1}\frac{\tau  b(x)}{\mu\tau   (\theta)\theta}\left(v_{xx}q+2v_xq_x+vq_{xx}\right)\theta_{xx}\dif x\\
&\ge\int_{0}^{1}\frac{\tau  b(x)}{\mu\tau   (\theta)\theta}vq_{xx}\theta_{xx}\dif x-CE^{\frac{1}{2}}(t)\mathcal{D}(t).
\end{align*}
The last term of \eqref{new-hu-6} can be estimates as follows:
\begin{align*}
&\int_{0}^{1}\frac{\tau  b(x)}{\mu\tau   (\theta)\theta}(\kappa(\theta)\theta_x)_{xx}\theta_{xx}\dif x\\
=&\int_{0}^{1}\frac{\tau  b(x)}{\mu\tau   (\theta)\theta}\left(\kappa^{\prime\prime}(\theta)\theta^3_x+3\kappa^\prime(\theta)\theta_x\theta_{xx}+\kappa(\theta)\theta_{xxx}\right)\theta_{xx}\dif x\\
\ge&\int_{0}^{1}\frac{\tau  b(x) \kappa (\theta)}{\mu \tau   (\theta)\theta}\theta_{xxx}\theta_{xx}\dif x-CE^{\frac{1}{2}}(t)\mathcal{D}(t)\\
=&\frac{\tau  }{2\mu}\frac{b(x) \kappa (\theta)}{\tau   (\theta)\theta}\theta^2_{xx}\bigg|_{0}^{1}-\frac{\tau  }{2\mu}\int_{0}^{1}\left(\frac{2}{Z(\theta)\theta}-\frac{Z^\prime(\theta)b(x)}{Z^2(\theta)\theta}\theta_x-\frac{b(x)}{Z(\theta)\theta^2}\theta_x\right)\theta^2_{xx}\dif x-CE^{\frac{1}{2}}(t)D(t)\\
\ge&\frac{\tau  \kappa(\theta(t,1))\theta^2_{xx}(t,1)}{2\mu \tau   (\theta(t,1))\theta(t,1)}+\frac{\tau   \kappa (\theta(t,0))\theta^2_{xx}(t,0)}{2\mu \tau   (\theta(t,0))\theta(t,0)}-\int_{0}^{1}\frac{\tau  }{\mu Z(\theta)\theta}\theta^2_{xx}\dif x-CE^{\frac{1}{2}}(t)D(t).
\end{align*}
Combining the above estimates,  using the fact $\tau   (\theta)=\tau    g(\theta)$, we derive that
\begin{align}\label{bj1}
&C_0\left(\theta^2_{xx}(t,1)+\theta^2_{xx}(t,0)\right)\le\frac{\dif}{\dif t}\int_{0}^{1}\frac{\tau  b(x)}{\mu\theta}\left(e_{\theta}\theta_{tx}-\frac{a(\theta)}{z(\theta)}q\theta_{xx}+\frac{R\theta}{v}u_{xx}\right)\theta_{xx}\dif x+\int_{0}^{1}\frac{\tau  }{\mu Z(\theta)\theta}\theta^2_{xx}\dif x\nonumber\\
&-\int_{0}^{1}\frac{\tau  b(x)}{\mu\tau   (\theta)\theta}vq_{xx}\theta_{xx}\dif x+\int_{0}^{1}\frac{\tau  e_\theta}{\theta\mu}\theta^2_{tx}\dif x-\int_{0}^{1}\frac{R\tau  b(x)}{v\mu}u_{xx}\theta_{txx}\dif x+CE^{\frac{1}{2}}(t)D(t)
\end{align}
where the constant $C_0$  is positive and  independent of $\tau   , \tau  $ and $\epsilon$. Note that a third-order term still appears in \eqref{bj1}, but it will be canceled out in conjunction with the estimates for \( u^2_{xx}(t,1) \) and \( u^2_{xx}(t,0) \) in the following part (see \eqref{bj2} below).		

Multiplying the equation $\eqref{xx}_5$ by $\frac{\overline{b}(x)}{\mu}u_{xx}$, we have
{\small 
\begin{align}
\int_{0}^{1}\frac{\tau  }{\mu}S_{txx}\overline{b}(x)u_{xx}\dif x+\frac{1}{\mu}\int_{0}^{1}(vS)_{xx}\overline{b}(x)u_{xx}\dif x-\int_{0}^{1}\overline{b}(x)u_{xxx}u_{xx}\dif x  =-\frac{\tau  \epsilon}{\mu}\int_{0}^{1}(b(x)S_x)_{xx}\overline{b}(x)u_{xx}\dif x
 . \label{new-hu-7}
\end{align}
}
We estimate each term of \eqref{new-hu-7} separately. Firstly, we have
\begin{align*}
&\int_{0}^{1}\frac{\tau  }{\mu}S_{txx}\overline{b}(x)u_{xx}\dif x\\
&=\frac{\dif}{\dif t}\int_{0}^{1}\frac{\tau  }{\mu}\overline{b}(x)S_{xx}u_{xx}\dif x-\int_{0}^{1}\frac{\tau  }{\mu}\overline{b}(x)S_{xx}u_{txx}\dif x\\
&=\frac{\dif}{\dif t}\int_{0}^{1}\frac{\tau  }{\mu}\overline{b}(x)S_{xx}u_{xx}\dif x-\int_{0}^{1}\frac{\tau  }{\mu}\overline{b}(x)\left(u_{tx}+p(v,\theta)_{xx}\right)u_{txx}\dif x\\
&=\frac{\dif}{\dif t}\int_{0}^{1}\frac{\tau  }{\mu}\overline{b}(x)S_{xx}u_{xx}\dif x-\int_{0}^{1}\frac{\tau  }{\mu}\overline{b}(x)p(v,\theta)_{xx}u_{txx}\dif x-\frac{\tau  }{2\mu}\overline{b}(x)u^2_{tx}\big|_{0}^{1}-\int_{0}^{1}\frac{\tau  }{\mu}u^2_{tx}\dif x\\
&\ge \frac{\dif}{\dif t}\int_{0}^{1}\frac{\tau  }{\mu}\overline{b}(x)S_{xx}u_{xx}\dif x-\int_{0}^{1}\frac{\tau  }{\mu}\overline{b}(x)p(v,\theta)_{xx}u_{txx}\dif x-\int_{0}^{1}\frac{\tau  }{\mu}u^2_{tx}\dif x,
\end{align*}
where
\begin{align*}
&-\int_{0}^{1}\frac{\tau  }{\mu}\overline{b}(x)p(v,\theta)_{xx}u_{txx}\dif x\\
=&-\frac{\dif}{\dif t}\int_{0}^{1}\frac{\tau  }{\mu}\overline{b}(x)p(v,\theta)_{xx}u_{xx}\dif x+\int_{0}^{1}\frac{\tau  }{\mu}\overline{b}(x)p(v,\theta)_{txx}u_{xx}\dif x\\
=&-\frac{\dif}{\dif t}\int_{0}^{1}\frac{\tau  }{\mu}\overline{b}(x)p(v,\theta)_{xx}u_{xx}\dif x+\int_{0}^{1}\frac{\tau  }{\mu}\overline{b}(x)\left\{p_{vvv}v^2_xv_t+2p_{vv\theta}v_xv_t\theta_x+p_{vv}v_{xx}v_t+2p_{vv}v_xv_{tx}\right.\\
&\left.+2p_{v\theta}v_{tx}\theta_x+p_{v\theta}v_t\theta_{xx}+p_{vv\theta}v^2_x\theta_t+p_{v\theta}v_{xx}\theta_t+2p_{v\theta}v_x\theta_{tx}+p_vv_{txx}+p_\theta\theta_{txx}\right\}u_{xx}\dif x\\
\ge& -\frac{\dif}{\dif t}\int_{0}^{1}\frac{\tau  }{\mu}\overline{b}(x)p(v,\theta)_{xx}u_{xx}\dif x+\int_{0}^{1}\frac{\tau  }{\mu}\overline{b}(x)\left(p_vv_{txx}+p_\theta\theta_{txx}\right)u_{xx}\dif x-CE^{\frac{1}{2}}(t)\mathcal{D}(t)\\
=& -\frac{\dif}{\dif t}\int_{0}^{1}\frac{\tau  }{\mu}\overline{b}(x)p(v,\theta)_{xx}u_{xx}\dif x+\frac{\tau  }{2\mu}\overline{b}(x)p_vu^2_{xx}\big|_{0}^{1}-\int_{0}^{1}\frac{\tau  }{2\mu}\left(-2p_v+\overline{b}(x)p_{vv}v_x+\overline{b}(x)p_{v\theta}\theta_x\right)u^2_{xx}\dif x\\
&+\int_{0}^{1}\frac{R\tau  }{v\mu}\overline{b}(x)\theta_{txx}u_{xx}\dif x-CE^{\frac{1}{2}}(t)\mathcal{D}(t)\\
\ge& -\frac{\dif}{\dif t}\int_{0}^{1}\frac{\tau  }{\mu}\overline{b}(x)p(v,\theta)_{xx}u_{xx}\dif x+\int_{0}^{1}\frac{\tau  }{\mu}p_vu^2_{xx}\dif x+\int_{0}^{1}\frac{R\tau  }{v\mu}\overline{b}(x)\theta_{txx}u_{xx}\dif x-CE^{\frac{1}{2}}(t)\mathcal{D}(t).
\end{align*}
since $p_v(v,\theta)<0$.

The second and third term of \eqref{new-hu-7} can be estimated as 
\begin{align*}
\frac{1}{\mu}\int_{0}^{1}(vS)_{xx}\overline{b}(x)u_{xx}\dif x
&=\frac{1}{\mu}\int_{0}^{1}\left(v_{xx}S+2v_xS_x+vS_{xx}\right)\overline{b}(x)u_{xx}\dif x\\
&\ge \int_{0}^{1}\frac{v}{\mu}\overline{b}(x)S_{xx}u_{xx}\dif x-CE^{\frac{1}{2}}(t)\mathcal{D}(t),
\end{align*}
and
\begin{align*}
-\int_{0}^{1}\overline{b}(x)u_{xxx}u_{xx}\dif x
=\frac{b(x)}{2}u^2_{xx}\big|_{0}^{1}-\int_{0}^{1}u^2_{xx}\dif x
=\frac{1}{2}\left(u^2_{xx}(t,1)+u^2_{xx}(t,0)\right)-\int_{0}^{1}u^2_{xx}\dif x.
\end{align*}
For the right-hand-side of \eqref{new-hu-7}, one has
\begin{align*}
&-\frac{\tau  \epsilon}{\mu}\int_{0}^{1}(b(x)S_x)_{xx}\overline{b}(x)u_{xx}\dif x\\
=&\frac{\tau  \epsilon}{\mu}\int_{0}^{1}(4S_{xx}+b(x)S_{xxx})b(x)u_{xx}\dif x\\
=&\frac{4\tau  \epsilon}{\mu}\int_{0}^{1}b(x)S_{xx}u_{xx}\dif x+\frac{\tau  \epsilon}{\mu}\int_{0}^{1}b^2(x)S_{xxx}u_{xx}\dif x\\
=&\frac{4\tau  \epsilon}{\mu}\int_{0}^{1}b(x)S_{xx}u_{xx}\dif x+\frac{\tau  \epsilon}{\mu}b^2(x)S_{xx}u_{xx}\big|_{0}^{1}-\frac{\tau  \epsilon}{\mu}\int_{0}^{1}(4b(x)u_{xx}+b^2(x)u_{xxx})S_{xx}\dif x\\
\le& \frac{\tau  \epsilon}{8\mu}\eta (S^2_{xx}(t,0)+S^2_{xx}(t,1))+\frac{2\tau  \epsilon}{\mu}C(\eta)(u^2_{xx}(t,0)+u^2_{xx}(t,1))\\
&-\frac{\tau  \epsilon}{\mu^2}\int_{0}^{1}b^2(x)(\tau  S_{txx}+(vS)_{xx}+\tau  \epsilon(b(x)S_x)_{xx})S_{xx}\dif x\\
\le&\frac{\tau  \epsilon}{8\mu} \eta (S^2_{xx}(t,0)+S^2_{xx}(t,1))+\frac{2\tau  \epsilon}{\mu} C(\eta) (u^2_{xx}(t,0)+u^2_{xx}(t,1))\\
&-\frac{\tau^2  \epsilon}{\mu^2}\frac{\dif}{\dif t}\int_{0}^{1}\frac{b^2(x)}{2}S^2_{xx}\dif x-\frac{\tau  \epsilon}{\mu^2}\int_{0}^{1}b^2(x)vS^2_{xx}\dif x
-(\frac{\tau  \epsilon}{\mu})^2\int_{0}^{1}b^2(x)(4S^2_{xx}+b(x)S_{xxx}S_{xx})\dif x\\
&+\int_{0}^{1}\eta S^2_{xx}\dif x
+CE^{\frac{1}{2}}(t)\mathcal{D}(t)\\
\le& \frac{\tau  \epsilon}{8\mu}\eta (S^2_{xx}(t,0)+S^2_{xx}(t,1))+\frac{2\tau  \epsilon}{\mu}C(\eta) (u^2_{xx}(t,0)+u^2_{xx}(t,1))-\frac{\tau^2  \epsilon}{\mu^2}\frac{\dif}{\dif t}\int_{0}^{1}\frac{b^2(x)}{2}S^2_{xx}\dif x\\
&-(\frac{\tau  \epsilon}{\mu})^2\frac{b^3(x)}{2}S^2_{xx}\big|_{0}^{1}
+\int_{0}^{1}\eta S^2_{xx}\dif x 
+CE^{\frac{1}{2}}(t)\mathcal{D}(t)\\
\le&\frac{\tau  \epsilon}{8\mu}\eta (S^2_{xx}(t,0)+S^2_{xx}(t,1))+\frac{2\tau  \epsilon}{\mu} C(\eta)(u^2_{xx}(t,0)+u^2_{xx}(t,1))-\frac{\tau^2  \epsilon}{\mu^2}\frac{\dif}{\dif t}\int_{0}^{1}\frac{b^2(x)}{2}S^2_{xx}\dif x\\
&+\int_{0}^{1}\eta S^2_{xx}\dif x 
+CE^{\frac{1}{2}}(t)\mathcal{D}(t),
\end{align*}
where $\eta$ is any sufficiently  small constant to be chosen later.

Combining the above estimates, we derive that
\begin{align}\label{bj2}
&\frac{1}{2}\left(u^2_{xx}(t,1)+u^2_{xx}(t,0)\right)
\le \frac{\dif}{\dif t}\int_{0}^{1} \left(\frac{\tau  }{\mu}\overline{b}(x)\left(p(v,\theta)_{xx}-S_{xx}\right)u_{xx}-\frac{\tau^2  \epsilon}{\mu^2}\frac{b^2(x)}{2}S^2_{xx}\right)\dif x \nonumber\\\nonumber
&+\int_{0}^{1}\frac{\tau  }{\mu}\left(u^2_{tx}-p_vu^2_{xx}\right)\dif x+\int_{0}^{1}\frac{R\tau  b(x)}{v\mu}\theta_{txx}u_{xx}\dif x+\int_{0}^{1}\frac{v}{\mu}b(x)S_{xx}u_{xx}\dif x+\int_{0}^{1}u^2_{xx}\dif x\\
&+\int_{0}^{1}\eta S^2_{xx}\dif x+\frac{\tau  \epsilon}{8\mu}(S^2_{xx}(t,0)+S^2_{xx}(t,1))+\frac{2\tau  \epsilon}{\mu} C(\eta)(u^2_{xx}(t,0)+u^2_{xx}(t,1))+CE^{\frac{1}{2}}(t)\mathcal{D}(t) .
\end{align}
Thus, combining inequalities \eqref{bj1} and \eqref{bj2} and noting the only two three-order terms in \eqref{bj1} and \eqref{bj2} cancels each other, we finally derive that
{\small{
\begin{align}
&C_0\left(u^2_{xx}(t,1)+u^2_{xx}(t,0)+\theta^2_{xx}(t,1)+\theta^2_{xx}(t,0)\right) \nonumber\\
&\le \frac{\dif}{\dif t}\int_{0}^{1}\left(\frac{\tau  }{\mu}\overline{b}(x)\left(p(v,\theta)_{xx}-S_{xx}\right)u_{xx}
-\frac{\tau^2  \epsilon}{\mu^2}\frac{b^2(x)}{2}S^2_{xx}+\frac{\tau  b(x)}{\mu\theta}\left(e_{\theta}\theta_{tx}-\frac{a(\theta)}{Z(\theta)}q\theta_{xx}+\frac{R\theta}{v}u_{xx}\right)\theta_{xx}\right)\dif x\nonumber\\
&+\int_{0}^{1}(\frac{\tau  }{\mu}\left(u^2_{tx}-p_vu^2_{xx}\right)+\frac{v}{\mu}b(x)S_{xx}u_{xx}+u^2_{xx}+\eta S^2_{xx}+\frac{\tau  \kappa(\theta)}{\mu \tau   (\theta)\theta}\theta^2_{xx}-\frac{\tau  b(x)}{\mu \tau   (\theta)\theta}vq_{xx}\theta_{xx}+\frac{\tau  e_\theta}{\theta\mu}\theta^2_{tx})\dif x \nonumber \\
& +\frac{\tau  \epsilon}{8\mu}\eta(S^2_{xx}(t,0)+S^2_{xx}(t,1))+\frac{2\tau  \epsilon}{\mu} C(\eta)(u^2_{xx}(t,0)+u^2_{xx}(t,1))+CE^{\frac{1}{2}}(t)\mathcal{D}(t).\label{new-hu-10}
\end{align}
}}
Here,  the constant $C_0$  is positive and  independent of $\tau   , \tau  $ and $\epsilon$. Without loss of generality, let $C_0=1$.  Now, we estimate the integral over $0$ to $t$ of the right-hand-side of the equation \eqref{new-hu-10}.

Firstly, notice that $p(v,\theta)_{xx}-S_{xx}=-u_{tx}$ and $v_{tx}=u_{xx}$, and using Lemma \ref{le3.9}, we have
\begin{align*}
&\int_{0}^{1}\left(\frac{\tau  }{\mu}\overline{b}(x)\left(p(v,\theta)_{xx}-S_{xx}\right)u_{xx}-\frac{\tau^2  \epsilon}{\mu^2}\frac{b^2(x)}{2}S^2_{xx}+\frac{\tau  b(x)}{\mu\theta}\left(e_{\theta}\theta_{tx}-\frac{a(\theta)}{Z(\theta)}q\theta_{xx}+\frac{R\theta}{v}u_{xx}\right)\theta_{xx}\right)\dif x \\
&\le \eta \int_0^1 \theta_{xx}^2 \dif x +C(\eta) \int_0^1 (u_{tx}^2+v_{tx}^2+\theta_{tx}^2 )\dif x +E^\frac{3}{2}(t)\\
&\le \eta \int_0^1 \theta_{xx}^2 \dif x+C(\eta)(E(0)+E^\frac{1}{2}(t) \int_0^t \mathcal D(s)\dif s +E^\frac{3}{2}(t)).
\end{align*}
Similarly,  remembering Corollary \ref{co3.10}, one have
\begin{align*}
&\int_0^t \int_{0}^{1}(\frac{\tau  }{\mu}\left(u^2_{tx}-p_vu^2_{xx}\right)+\frac{v}{\mu}b(x)S_{xx}u_{xx}+u^2_{xx}+\eta S^2_{xx})\dif x \dif t\\
&\le \eta \int_0^t \int_0^1 S_{xx}^2 \dif x \dif t +C(\eta) \epsilon \int_0^t \int_0^1 S^2_{xx}\dif x \dif t
\end{align*}
and
\begin{align*}
&\int_0^t \int_{0}^{1}(\frac{\tau  \kappa(\theta)}{\mu \tau   (\theta)\theta}\theta^2_{xx}-\frac{\tau  b(x)}{\mu \tau   (\theta)\theta}vq_{xx}\theta_{xx}+\frac{\tau  e_\theta}{\theta\mu}\theta^2_{tx})\dif x \dif t\\
&\le \eta \int_0^t \int_0^1 q_{xx}^2\dif x \dif t+ C(\eta)\epsilon \int_0^t \int_0^1 S^2_{xx}\dif x \dif t.
\end{align*}

Therefore, integrating the inequality \eqref{new-hu-10} over $(0,t)$, once obtain the inequality \eqref{new-hu-9} immediately and the proof is finished.
\end{proof}

Using the similar methods, we can estimate the mixed second-order estimates on the boundary. 

% estimates of $u_{tx}^2(0)$

\begin{lemma}\label{le-hu-1}
There exists some constant $C$ such that
\begin{align}\label{new-hu-8}
&\int_{0}^{t}(u^2_{tx}(t,1)+u^2_{tx}(t,0)+\theta^2_{tx}(t,1)+\theta^2_{tx}(t,0))\dif t
\le C\left(E_0+E^\frac{1}{2}\int_{0}^{t}\mathcal{D}(s)\dif s \right)+C\epsilon \int_{0}^{t}\int_{0}^{1} S^2_{xx}\dif x \dif t \nonumber\\
&+\frac{\tau  \epsilon}{8\mu}\eta \int_{0}^{t}(S^2_{tx}(t,0)+S^2_{tx}(t,1))\dif t+ \frac{2\tau  \epsilon}{\mu}C(\eta) \int_0^t(u^2_{tx}(t,0)+u^2_{tx}(t,1))\dif t .
\end{align}
\end{lemma}

\begin{proof}
Multiplying the equation $\eqref{tx}_4$ by$\frac{\tau  b(x)}{\mu\tau   (\theta)\theta}\theta_{tx}$, we have
\begin{align}\label{new-hu-11}
\int_{0}^{1}\frac{\tau  b(x)}{\mu\tau   (\theta)\theta}\left(\tau   (\theta)q_t\right)_{tx}\theta_{tx}\dif x+\int_{0}^{1}\frac{\tau  b(x)}{\mu\tau   (\theta)\theta}(vq)_{tx}\theta_{tx}\dif x+\int_{0}^{1}\frac{\tau  b(x)}{\mu\tau   (\theta)\theta}(\kappa(\theta)\theta_x)_{tx}\theta_{tx}\dif x=0.
\end{align}
We estimate each term of the equation \eqref{new-hu-11} separately. 
First, we have
\begin{align*}
&\int_{0}^{1}\frac{\tau  b(x)}{\mu\tau   (\theta)\theta}\left(\tau   (\theta)q_t\right)_{tx}\theta_{tx}\dif x\\
&=\int_{0}^{1}\frac{\tau  b(x)}{\mu\tau   (\theta)\theta}\left(\tau   ^{\prime\prime}(\theta)\theta_t\theta_xq_t+\tau   ^\prime(\theta)\theta_{tx}q_t+\tau   ^\prime(\theta)\theta_tq_{tx}+\tau   ^\prime(\theta)\theta_{tx}q_t+\tau   ^\prime(\theta)\theta_xq_{tt}+\tau   (\theta)q_{ttx}\right)\theta_{tx}\dif x\\
&\ge \int_{0}^{1}\frac{\tau  b(x)}{\mu\theta}q_{ttx}\theta_{tx}\dif x-CE^{\frac{1}{2}}(t)\mathcal{D}(t).
\end{align*}
Note that, by using $\eqref{tt}_3$, one get
\begin{align*}
q_{ttx}=(\frac{2a(\theta)}{\tau   (\theta)}q^2v)_{tt}+(\frac{S^2v}{\mu})_{tt}-(e_{\theta}\theta_t)_{tt}+(\frac{2a(\theta)}{Z(\theta)}q\theta_x)_{tt}-(\frac{R\theta u_x}{v})_{tt}.
\end{align*}
So, we have
{\small{
\begin{align*}
&\int_{0}^{1}\frac{\tau  b(x)}{\mu\theta}q_{ttx}\theta_{tx}\dif x\\
\ge&\int_{0}^{1}\frac{\tau  b(x)}{\mu\theta}\left(-e_{\theta}\theta_{ttt}+\frac{2a(\theta)}{Z(\theta)}q\theta_{ttx}-\frac{R\theta}{v}u_{ttx}\right)\theta_{tx}\dif x-CE^{\frac{1}{2}}(t)\mathcal{D}(t)\\
\ge&-\frac{\dif}{\dif t}\int_{0}^{1}\frac{\tau  b(x)}{\mu\theta}\left(e_{\theta}\theta_{tt}-\frac{a(\theta)}{Z(\theta)}q\theta_{tx}+\frac{R\theta}{v}u_{tx}\right)\theta_{tx}\dif x\\
&+\int_{0}^{1}\left\{-\frac{\tau  b(x)\theta_t}{\mu\theta^2}\left(e_{\theta}\theta_{tt}-\frac{a(\theta)}{Z(\theta)}q\theta_{tx}+\frac{R\theta}{v}u_{tx}\right)\theta_{tx}\right.\\
&+\frac{\tau  b(x)}{\mu\theta}\left(e_{\theta\theta}\theta_t\theta_{tt}+e_{\theta q}q_t\theta_{tt}-\frac{a^\prime(\theta)Z(\theta)-a(\theta)Z^\prime(\theta)}{Z^2(\theta)}q\theta_t\theta_{tx} -\frac{a(\theta)}{Z(\theta)}q_t\theta_{tx}+\frac{R\theta_t}{v}u_{tx}-\frac{R\theta}{v^2}v_tu_{xx}\right)\theta_{tx}\\
&\left.+\frac{\tau  b(x)}{\mu\theta}e_\theta\theta_{tt}\theta_{ttx}+\frac{R\tau  b(x)}{\mu v}u_{tx}\theta_{ttx}\right\}\dif x-CE^{\frac{1}{2}}(t)\mathcal{D}(t)\\
\ge&-\frac{\dif}{\dif t}\int_{0}^{1}\frac{\tau  b(x)}{\mu\theta}\left(e_{\theta}\theta_{tt}-\frac{a(\theta)}{Z(\theta)}q\theta_{tx}+\frac{R\theta}{v}u_{tx}\right)\theta_{tx}\dif x\\
&+\int_{0}^{1}\frac{\tau  b(x)}{\mu\theta}e_\theta\theta_{tt}\theta_{ttx}\dif x+\int_{0}^{1}\frac{R\tau  b(x)}{\mu v}u_{tx}\theta_{ttx}\dif x-CE^{\frac{1}{2}}(t)\mathcal{D}(t),
\end{align*}
}}
where
\begin{align*}
&\int_{0}^{1}\frac{\tau  b(x)}{\mu\theta}e_\theta\theta_{tt}\theta_{ttx}\dif x\\
&=\frac{1}{2}\frac{\tau  b(x)}{\mu\theta}e_\theta\theta^2_{tt}\big|_{0}^{1}-\int_{0}^{1}\frac{\tau  }{2\mu}\left(\frac{b(x)e_\theta}{\theta}\right)_x\theta^2_{tt}\dif x\\
&\ge-\int_{0}^{1}\frac{\tau  }{2\mu}\left(\frac{2e_\theta}{\theta}-\frac{\theta_xb(x)e_\theta}{\theta^2}+\frac{b(x)e_{\theta\theta}\theta_x}{\theta}+\frac{b(x)e_{q\theta}q_x}{\theta}\right)\theta^2_{tt}\dif x\\
&\ge -\int_{0}^{1}\frac{\tau  e_\theta}{\mu\theta}\theta^2_{tt}\dif x-CE^{\frac{1}{2}}(t)\mathcal{D}(t).
\end{align*}

The second left term of \eqref{new-hu-11} can be estimated as follows:
\begin{align*}
\int_{0}^{1}\frac{\tau  b(x)}{\mu\tau   (\theta)\theta}(vq)_{tx}\theta_{tx}\dif x
&=\int_{0}^{1}\frac{\tau  b(x)}{\mu\tau   (\theta)\theta}\left(v_{tx}q+v_xq_t+v_tq_x+vq_{tx}\right)\theta_{tx}\dif x\\
&\ge\int_{0}^{1}\frac{\tau  b(x)}{\mu\tau   (\theta)\theta}vq_{tx}\theta_{tx}\dif x-CE^{\frac{1}{2}}(t)\mathcal{D}(t),
\end{align*}
and the last left term of \eqref{new-hu-11} as 
\begin{align*}
&\int_{0}^{1}\frac{\tau  b(x)}{\mu\tau   (\theta)\theta}(\kappa(\theta)\theta_x)_{tx}\theta_{tx}\dif x\\
=&\int_{0}^{1}\frac{\tau  b(x)}{\mu\tau   (\theta)\theta}\left(\kappa^{\prime\prime}(\theta)\theta_t\theta^2_x+2\kappa^\prime(\theta)\theta_x\theta_{tx}+\kappa^\prime(\theta)\theta_t\theta_{xx}+\kappa(\theta)\theta_{txx}\right)\theta_{tx}\dif x\\
\ge&\int_{0}^{1}\frac{\tau  b(x)\kappa (\theta)}{\mu \tau   (\theta)\theta}\theta_{txx}\theta_{tx}\dif x-CE^{\frac{1}{2}}(t)\mathcal{D}(t)\\
=&\frac{\tau  }{2\mu}\frac{b(x)\kappa(\theta)}{\tau   (\theta)\theta}\theta^2_{tx}\bigg|_{0}^{1}-\frac{\tau  }{2\mu}\int_{0}^{1}\left(\frac{2}{Z(\theta)\theta}-\frac{Z^\prime(\theta)b(x)}{Z^2(\theta)\theta}\theta_x-\frac{b(x)}{Z(\theta)\theta^2}\theta_x\right)\theta^2_{tx}\dif x-CE^{\frac{1}{2}}(t)D(t)\\
\ge& C_0(\theta^2_{tx}(t,1)+\theta^2_{tx}(t,0)-\int_{0}^{1}\frac{\tau  }{\mu Z(\theta)\theta}\theta^2_{tx}\dif x-CE^{\frac{1}{2}}(t)D(t).
\end{align*}
where we use the fact $\tau   (\theta)=\tau    g(\theta)$ and the assumption $\tau   =\tau  $. $C_0$ is a positive constant. 

Combining the above estimates,   we derive that
\begin{align}
&C_0(\theta^2_{tx}(t,1)+\theta^2_{tx}(t,0)\nonumber
\\\le\nonumber
& \frac{\dif}{\dif t}\int_{0}^{1}\frac{\tau  b(x)}{\mu\theta}\left(e_{\theta}\theta_{tt}-\frac{a(\theta)}{z(\theta)}q\theta_{tx}+\frac{R\theta}{v}u_{tx}\right)\theta_{tx}\dif x-\int_{0}^{1}\frac{R\tau  b(x)}{\mu v}u_{tx}\theta_{ttx}\dif x+\int_{0}^{1}\frac{\tau  e_\theta}{\mu\theta}\theta^2_{tt}\dif x
\\
&-\int_{0}^{1}\frac{\tau  b(x)}{\mu\tau   (\theta)\theta}vq_{tx}\theta_{tx}\dif x+\int_{0}^{1}\frac{\tau  }{\mu Z(\theta)\theta}\theta^2_{tx}\dif x+CE^{\frac{1}{2}}(t)D(t). \label{bj3}
\end{align}
Notice that, there is still a third-order term in above inequality. As we did in Lemma \ref{le-hu-2}, these high-order terms will be canceled if one combines \eqref{bj3} with the estimates \eqref{bj4} below.

Now, we control the boundary term $u_{tx}|_{\partial \Omega}$.
To do this, multiplying the equation $\eqref{tx}_5$ by $\frac{\overline{b}(x)}{\mu}u_{tx}$, we have
\begin{align}
\int_{0}^{1}\frac{\tau  }{\mu}S_{ttx}\overline{b}(x)u_{tx}\dif x+\frac{1}{\mu}\int_{0}^{1}(vS)_{tx}\overline{b}(x)u_{tx}\dif x-\int_{0}^{1}\overline{b}(x)u_{txx}u_{tx}\dif x 
=-\frac{\tau  \epsilon}{\mu}\int_{0}^{1}(b(x)S_x)_{tx}\overline{b}(x)u_{tx}\dif x . \label{new-hu-12}
\end{align}

For the first left term of \eqref{new-hu-12}, we have
\begin{align*}
&\int_{0}^{1}\frac{\tau  }{\mu}S_{ttx}\overline{b}(x)u_{tx}\dif x\\
&=\frac{\dif}{\dif t}\int_{0}^{1}\frac{\tau  }{\mu}\overline{b}(x)S_{tx}u_{tx}\dif x-\int_{0}^{1}\frac{\tau  }{\mu}\overline{b}(x)S_{tx}u_{ttx}\dif x\\
&=\frac{\dif}{\dif t}\int_{0}^{1}\frac{\tau  }{\mu}\overline{b}(x)S_{tx}u_{tx}\dif x-\int_{0}^{1}\frac{\tau  }{\mu}\overline{b}(x)\left(u_{tt}+p(v,\theta)_{tx}\right)u_{ttx}\dif x\\
&=\frac{\dif}{\dif t}\int_{0}^{1}\frac{\tau  }{\mu}\overline{b}(x)S_{tx}u_{tx}\dif x-\int_{0}^{1}\frac{\tau  }{\mu}\overline{b}(x)p(v,\theta)_{tx}u_{ttx}\dif x-\frac{\tau  }{2\mu}\overline{b}(x)u^2_{tt}\big|_{0}^{1}-\int_{0}^{1}\frac{\tau  }{\mu}u^2_{tt}\dif x\\
&\ge \frac{\dif}{\dif t}\int_{0}^{1}\frac{\tau  }{\mu}\overline{b}(x)S_{tx}u_{tx}\dif x-\int_{0}^{1}\frac{\tau  }{\mu}\overline{b}(x)p(v,\theta)_{tx}u_{ttx}\dif x-\int_{0}^{1}\frac{\tau  }{\mu}u^2_{tt}\dif x,
\end{align*}
where
\begin{align*}
&-\int_{0}^{1}\frac{\tau  \overline{b}(x)}{\mu}p(v,\theta)_{tx}u_{ttx}\dif x\\
=&-\frac{\dif}{\dif t}\int_{0}^{1}\frac{\tau  \overline{b}(x)}{\mu}p(v,\theta)_{tx}u_{tx}\dif x+\int_{0}^{1}\frac{\tau  \overline{b}(x)}{\mu}p(v,\theta)_{ttx}u_{tx}\dif x\\
=&-\frac{\dif}{\dif t}\int_{0}^{1}\frac{\tau  \overline{b}(x)}{\mu}p(v,\theta)_{tx}u_{tx}\dif x+\int_{0}^{1}\frac{\tau  \overline{b}(x)}{\mu}\{p_{vvv}v^2_tv_x+2p_{vv\theta}v_tv_x\theta_t+p_{vv\theta}v^2_t\theta_x+2p_{vv}v_{tx}v_t\\
&+p_{vv}v_{tt}v_x+p_{v\theta}v_x\theta_{tt}+2p_{v\theta}v_{tx}\theta_t+p_{v\theta}v_{tt}\theta_x+2p_{v\theta}v_t\theta_{tx}+p_vv_{ttx}+p_\theta\theta_{ttx}\}u_{tx}\dif x\\
\ge& -\frac{\dif}{\dif t}\int_{0}^{1}\frac{\tau  \overline{b}(x)}{\mu}p(v,\theta)_{tx}u_{tx}\dif x+\int_{0}^{1}\frac{\tau  \overline{b}(x)}{\mu}\left(p_vv_{ttx}+p_\theta\theta_{ttx}\right)u_{tx}\dif x-CE^{\frac{1}{2}}(t)\mathcal{D}(t)\\
=& -\frac{\dif}{\dif t}\int_{0}^{1}\frac{\tau  \overline{b}(x)}{\mu}p(v,\theta)_{tx}u_{tx}\dif x+\frac{\tau  \overline{b}(x)}{2\mu}p_vu^2_{tx}\big|_{0}^{1}-\int_{0}^{1}\frac{\tau  }{2\mu}\left(-2p_v+\overline{b}(x)p_{vv}v_x+\overline{b}(x)p_{v\theta}\theta_x\right)u^2_{tx}\dif x\\
&+\int_{0}^{1}\frac{R\tau  }{\mu v}\overline{b}(x)\theta_{ttx}u_{tx}\dif x-CE^{\frac{1}{2}}(t)\mathcal{D}(t)\\
\ge& -\frac{\dif}{\dif t}\int_{0}^{1}\frac{\tau  \overline{b}(x)}{\mu}p(v,\theta)_{tx}u_{tx}\dif x+\int_{0}^{1}\frac{\tau  }{\mu}p_vu^2_{tx}\dif x+\int_{0}^{1}\frac{R\tau  }{\mu v}\overline{b}(x)\theta_{ttx}u_{tx}\dif x-CE^{\frac{1}{2}}(t)\mathcal{D}(t),
\end{align*}
where we use the fact $p_v(v,\theta)<0$.

The second and third left terms of \eqref{new-hu-12} can be estimates as
\begin{align*}
\frac{1}{\mu}\int_{0}^{1}(vS)_{tx}\overline{b}(x)u_{tx}\dif x
&=\int_{0}^{1}\frac{\overline{b}(x)}{\mu}\left(v_{tx}S+v_xS_t+v_tS_x+vS_{tx}\right)u_{tx}\dif x\\
&\ge \int_{0}^{1}\frac{v\overline{b}(x)}{\mu}S_{tx}u_{tx}\dif x-CE^{\frac{1}{2}}(t)\mathcal{D}(t),
\end{align*}
and 
\begin{align*}
-\int_{0}^{1}\overline{b}(x)u_{txx}u_{tx}\dif x
=\frac{b(x)}{2}u^2_{tx}\big|_{0}^{1}-\int_{0}^{1} u^2_{tx}\dif x
=\frac{1}{2}\left(u^2_{tx}(t,1)+u^2_{tx}(t,0)\right)-\int_{0}^{1} u^2_{tx}\dif x.
\end{align*}
For the right-hand of the equation \eqref{new-hu-12}, we derive that
{\small
\begin{align*}
&-\frac{\tau  \epsilon}{\mu}\int_{0}^{1}(b(x)S_x)_{tx}\overline{b}(x)u_{tx}\dif x\\
=&\frac{\tau  \epsilon}{\mu}\int_{0}^{1}(2S_{tx}+b(x)S_{txx})b(x)u_{tx}\dif x\\
=&\frac{2\tau  \epsilon}{\mu}\int_{0}^{1}b(x)S_{tx}u_{tx}\dif x+\frac{\tau  \epsilon}{\mu} b^2(x)S_{tx}u_{tx}\big|_{0}^{1}-\frac{\tau  \epsilon}{\mu}\int_{0}^{1}(4b(x)u_{tx}+b^2(x)u_{txx})S_{tx}\dif x\\
\le&\frac{-2\tau  \epsilon}{\mu}\int_{0}^{1}b(x)S_{tx}u_{tx}\dif x+ \frac{\tau  \epsilon}{8\mu} \eta (S^2_{tx}(t,0)+S^2_{tx}(t,1))+\frac{2\tau  \epsilon}{\mu}C(\eta)(u^2_{tx}(t,0)+u^2_{tx}(t,1))\\
&-\frac{\tau  \epsilon}{\mu^2}\int_{0}^{1}b^2(x)(\tau  S_{ttx}+(vS)_{tx}+\tau  \epsilon(b(x)S_x)_{tx} )S_{tx}\dif x\\
\le&-\frac{2\tau  \epsilon}{\mu}\int_{0}^{1}b(x)S_{tx}u_{tx}\dif x+\frac{\tau  \epsilon}{8\mu} \eta (S^2_{tx}(t,0)+S^2_{tx}(t,1))+\frac{2\tau  \epsilon}{\mu}C(\eta) (u^2_{tx}(t,0)+u^2_{tx}(t,1))\\
&-\frac{\tau^2  \epsilon}{\mu^2}\frac{\dif}{\dif t}\int_{0}^{1}\frac{b^2(x)}{2}S^2_{tx}\dif x-\frac{\tau  \epsilon}{\mu^2}\int_{0}^{1}b^2(x)vS^2_{tx}\dif x-(\frac{\tau  \epsilon}{\mu})^2\int_{0}^{1}b^2(x)(2S^2_{tx}+b(x)S_{txx}S_{tx})\dif x\\
&+\int_{0}^{1}S^2_{tx}\dif x +CE^{\frac{1}{2}}(t)\mathcal{D}(t)\\
\le&-\frac{2\tau  \epsilon}{\mu}\int_{0}^{1}b(x)S_{tx}u_{tx}\dif x+ \frac{\tau  \epsilon}{8\mu}\eta (S^2_{tx}(t,0)+S^2_{tx}(t,1))+\frac{2\tau  \epsilon}{\mu}C(\eta)(u^2_{tx}(t,0)+u^2_{tx}(t,1))\\
&-\frac{\tau^2  \epsilon}{\mu^2}\frac{\dif}{\dif t}\int_{0}^{1}\frac{b^2(x)}{2}S^2_{tx}\dif x-(\frac{\tau  \epsilon}{\mu})^2\frac{b^3(x)}{2}S^2_{tx}\big|_{0}^{1}+\int_{0}^{1} S^2_{tx}\dif x 
+CE^{\frac{1}{2}}(t)\mathcal{D}(t)\\
\le&-2\tau  \epsilon\int_{0}^{1}b(x)S_{tx}u_{tx}\dif x+ \frac{\tau  \epsilon}{8\mu}\eta (S^2_{tx}(t,0)+S^2_{tx}(t,1))+\frac{2\tau  \epsilon}{\mu}C(\eta)(u^2_{tx}(t,0)+u^2_{tx}(t,1))\\
&-\frac{\tau^2  \epsilon}{\mu}\frac{\dif}{\dif t}\int_{0}^{1}\frac{b^2(x)}{2}S^2_{tx}\dif x+\int_{0}^{1} S^2_{tx}\dif x 
+CE^{\frac{1}{2}}(t)\mathcal{D}(t).
\end{align*}
}
Combining the above estimates, we derive that
\begin{align}\label{bj4}
&\frac{1}{2}\left(u^2_{tx}(t,1)+u^2_{tx}(t,0)\right)
\le \frac{\dif}{\dif t}\int_{0}^{1}\frac{\tau  \overline{b}(x)}{\mu}\left(p(v,\theta)_{tx}-S_{tx}\right)u_{tx}-\frac{\tau^2  \epsilon}{\mu^2}\frac{b^2(x)}{2}S^2_{tx}\dif x\nonumber\\
&+\int_{0}^{1}\frac{R\tau  }{\mu v}b(x)\theta_{ttx}u_{tx}\dif x+\int_{0}^{1}((1+\frac{R\theta}{\mu v^2}\tau  )u^2_{tx}+\frac{\tau  }{\mu}u^2_{tt}+\frac{1}{\mu}S^2_{tx})\dif x \nonumber\\
&+\int_{0}^{1}\frac{(v-2\tau  \epsilon)b(x)}{\mu}S_{tx}u_{tx}\dif x+\frac{\tau  \epsilon}{8\mu}\eta(S^2_{tx}(t,0)+S^2_{tx}(t,1))+\frac{2\tau  \epsilon}{\mu}C(\eta)(u^2_{tx}(t,0)+u^2_{tx}(t,1))\nonumber\\
&+ CE^{\frac{1}{2}}(t)\mathcal{D}(t).
\end{align}
Thus, combining $\eqref{bj3}-\eqref{bj4}$, we can get 
\begin{align*}
&C_0\left(u^2_{tx}(t,1)+u^2_{tx}(t,0)+\theta^2_{tx}(t,1)+\theta^2_{tx}(t,0)\right)\\
\le&\frac{\dif}{\dif t}\int_{0}^{1}(\frac{\tau  \overline{b}(x)}{\mu}\left(p(v,\theta)_{tx}-S_{tx}\right)u_{tx}-\frac{\tau^2  \epsilon}{\mu^2}\frac{b^2(x)}{2}S^2_{tx}+\frac{\tau  b(x)}{\mu\theta}\left(e_{\theta}\theta_{tt}-\frac{a(\theta)}{Z(\theta)}q\theta_{tx}+\frac{R\theta}{v}u_{tx}\right)\theta_{tx})\dif x\\ 
&+\int_{0}^{1}(\frac{\tau  }{\mu}u^2_{tt}+(1+\frac{R\theta}{\mu v^2}\tau  )u^2_{tx}+\frac{(v-2\tau  \epsilon)b(x)}{\mu}S_{tx}u_{tx}+\frac{ 1}{\mu}S_{tx}^2)\dif x\\
&+\int_{0}^{1}(\frac{\tau  e_\theta}{\mu\theta}\theta^2_{tt}-\frac{\tau  b(x)}{\mu \tau   (\theta)\theta}vq_{tx}\theta_{tx}
+\frac{\tau  \kappa(\theta)}{\mu \tau   (\theta)\theta}\theta^2_{tx})\dif x+\frac{\tau  \epsilon}{8\mu}\eta(S^2_{tx}(t,0)+S^2_{tx}(t,1))\\
&+\frac{2\tau  \epsilon}{\mu}C(\eta)(u^2_{tx}(t,0)+u^2_{tx}(t,1))
+ CE^{\frac{1}{2}}(t)\mathcal{D}(t).
\end{align*}
Integrating the above result over $(0,t)$,  using similar method as we did in \eqref{new-hu-10}, once obtain the inequality \eqref{new-hu-8} immediately and the proof is finished.
\end{proof}

Until now, the only left boundary term is $q_{xx}\big|_{0}^{1}$, which can be shown to be controlled by $u_{xx}|_{\partial \Omega}$ by using the equation $\eqref{x}_3$. Indeed,
multiplying the equation $\eqref{x}_3$ by $q_{xx}$, and evaluated at the point $x=1$, we have
\begin{align*}
q^2_{xx}(t,1)=&[(\frac{2a(\theta)}{\tau   (\theta)}q^2v)_xq_{xx}](t,1)+[(\frac{S^2v}{\mu})_xq_{xx}](t,1)-[(e_{\theta}\theta_t)_xq_{xx}](t,1)\\
&+[(\frac{2a(\theta)}{Z(\theta)}q\theta_x)_xq_{xx}](t,1)-[(\frac{R\theta u_x}{v})_xq_{xx}](t,1).
\end{align*}
Note that the first term on the right-hand side of the above equation vanishes:
\begin{align*}
&[(\frac{2a(\theta)}{\tau   (\theta)}q^2v)_xq_{xx}](t,1)\\
=&[\left(\frac{2a^\prime(\theta)}{\tau   (\theta)}\theta_xq^2v-\frac{2a(\theta)\tau   ^\prime(\theta)}{\tau   (\theta)}\theta_xq^2v+\frac{2a(\theta)}{\tau   (\theta)}2qq_xv+\frac{2a(\theta)}{\tau   (\theta)}q^2v_x\right)q_{xx}](t,1)=0,
\end{align*}
For the second term, from $\eqref{approximate}_5$, we have
\begin{align*}
[(\frac{S^2v}{\mu})_xq_{xx}](t,1)&=[\left(2SS_xv+S^2v_x\right)\frac{q_{xx}}{\mu}](t,1)\\
\le &\frac{1}{8} q_{xx}^2(t,1)+C(2SvS_x+S^2v_x)^2(t,1)\\
=&\frac{1}{8} q_{xx}^2(t,1)+C(4S^2v^2S_x^2+2S^3vS_xv_x+S^4v_x^2)(t,1)\\
\le & \frac{1}{8} q_{xx}^2(t,1)+CE^\frac{1}{2} \mathcal D(t)
 \end{align*}
 where we use the fact $|S|_{L^\infty} \le C |\mu u_x-\tau   (S_t+\epsilon S_x)|_{L^\infty}\le C E^\frac{1}{2}$.
 
Similarly, for the remaining three terms on the right-hand of the equation, by  noting that $(q, \theta_x) \big|_{\partial\Omega}=0$, one has
\begin{align*}
&\left[-(e_{\theta}\theta_t)_xq_{xx}+(\frac{2a(\theta)}{Z(\theta)}q\theta_x)_xq_{xx}-(\frac{R\theta u_x}{v})_xq_{xx}\right](t,1)\\
=&\left[(e_{\theta\theta}\theta_x\theta_t+e_{\theta q}q_x\theta_t+e_\theta\theta_{tx}+\frac{2a^\prime(\theta)}{Z(\theta)}q\theta^2_x-\frac{2a(\theta)Z^\prime(\theta)}{Z^2(\theta)}q\theta^2_x+\frac{2a(\theta)}{Z(\theta)}q_x\theta_x+\frac{2a(\theta)}{Z(\theta)}q\theta_{xx}\right.\\
&\left.+\frac{R\theta_xu_x}{v}+\frac{R\theta u_{xx}}{v}-\frac{R\theta u_xv_x}{v^2})q_{xx}\right](t,1)\\
\le&\frac{1}{8}  q_{xx}^2(t,1)+C\left(\frac{2a(\theta)}{Z(\theta)}q_x\theta_x+\frac{R\theta_xu_x}{v}+\frac{R\theta u_{xx}}{v}-\frac{R\theta u_xv_x}{v^2}\right)^2(t,1)\\
\le&\frac{1}{8} q_{xx}^2(t,1)+C\left(u^2_{xx}(t,1)+E^{\frac{1}{2}}(t)\mathcal{D}(t)\right).
\end{align*}
Combining the above estimates, we get
\begin{align}\label{q_{xx}(t,1)}
q^2_{xx}(t,1)\le C\left(u^2_{xx}(t,1)+E^{\frac{1}{2}}(t)\mathcal{D}(t)\right) .
\end{align}
In the same way, we have
\begin{align}\label{q_{xx}(t,0)}
q^2_{xx}(t,0)\le C\left(u^2_{xx}(t,0)+E^{\frac{1}{2}}(t)\mathcal{D}(t)\right).
\end{align}

Therefore, combining Lemmas \ref{le-hu-2}, \ref{le-hu-1} and \eqref{q_{xx}(t,1)}-\eqref{q_{xx}(t,0)}, we get 
\begin{lemma}\label{le3.14}
There exists some constant $C$ such that
\begin{align}
&\int_{0}^{t}(u^2_{xx}(t,1)+u^2_{xx}(t,0)+\theta^2_{xx}(t,1)+\theta^2_{xx}(t,0)+u^2_{tx}(t,1)+u^2_{tx}(t,0)+q^2_{xx}(t,1)+q^2_{xx}(t,0))\dif t\nonumber\\
&\le C \eta \left(\int_{0}^{1}\theta^2_{xx}\dif x +\int_0^t\int_0^1 (q_{xx}^2+S_{xx}^2)\dif x \dif t\right) +C(\eta) \epsilon \int_{0}^{t}\int_{0}^{1}S^2_{xx}\dif x \dif t  \nonumber\\
&+C\frac{\tau  \epsilon}{8\mu} \eta \int_{0}^{t}(S^2_{xx}(t,0)+S^2_{xx}(t,1)+S_{tx}^2(t,0)+S_{tx}^2(t,1))\dif t \nonumber\\
&+\frac{2\tau  \epsilon}{\mu} C(\eta)\int_0^t (u^2_{xx}(t,0)+u^2_{xx}(t,1)+u_{tx}^2(t,0)+u_{tx}^2(t,1))\dif t \nonumber\\
&+ C\left(E_0+E^\frac{1}{2}\int_{0}^{t}\mathcal{D}(s)\dif s+E^\frac{3}{2}(t) \right).  \label{bjtxxx}
\end{align}
 \end{lemma}

% combining Lemma 3.11, 3.14.

Now, following Lemmas \ref{le3.11} and \ref{le3.14}, we have the following corollary
\begin{corollary}\label{co3.15}
There exists a constant $C$ such that
\begin{align}
 &\int_{0}^{1}(u^2_{xx}+v^2_{xx}+\theta^2_{xx}+\tau   q^2_{xx}+\tau   S^2_{xx})\dif x+\int_{0}^{t}\int_{0}^{1}(q^2_{xx}+S^2_{xx})\dif x\dif t \nonumber\\
 &+\int_{0}^{t}(u^2_{xx}(t,1)+u^2_{xx}(t,0)+\theta^2_{xx}(t,1)+\theta^2_{xx}(t,0)+u^2_{tx}(t,1)+u^2_{tx}(t,0)+q^2_{xx}(t,1)+q^2_{xx}(t,0))\dif t  \nonumber\\
 &\le   C\left(E_0+E^\frac{1}{2}\int_{0}^{t}\mathcal{D}(s)\dif s+E^\frac{3}{2}(t) \right).  \label{new-hu-16}
\end{align}
\end{corollary}
\begin{proof}
Indeed, multiplying \eqref{bjtxxx} by $C+1$, choosing $\eta$ small such that $(C+1)C\eta<\frac{1}{2}$ and fix $\eta$, we have
\begin{align}
 &\int_{0}^{1}(u^2_{xx}+v^2_{xx}+\frac{1}{2}\theta^2_{xx}+\tau   q^2_{xx}+\tau   S^2_{xx})\dif x+\frac{1}{2}\int_{0}^{t}\int_{0}^{1}(q^2_{xx}+S^2_{xx})\dif x\dif t \nonumber\\
 &+\int_{0}^{t}(u^2_{xx}(t,1)+u^2_{xx}(t,0)+\theta^2_{xx}(t,1)+\theta^2_{xx}(t,0)+u^2_{tx}(t,1)+u^2_{tx}(t,0)+q^2_{xx}(t,1)+q^2_{xx}(t,0))\dif t  \nonumber\\
 &\le  C(\eta) \epsilon \int_{0}^{t}\int_{0}^{1}S^2_{xx}\dif x \dif t 
 +\frac{2\tau  \epsilon}{\mu} C(\eta)\int_0^t (u^2_{xx}(t,0)+u^2_{xx}(t,1)+u_{tx}^2(t,0)+u_{tx}^2(t,1))\dif t \nonumber\\
&+ C\left(E_0+E^\frac{1}{2}\int_{0}^{t}\mathcal{D}(s)\dif s+E^\frac{3}{2}(t) \right).  \label{new-hu-15}
\end{align}

For fixed $\eta$, choosing $\epsilon$ sufficiently small, we get the desired result \eqref{new-hu-16} immediately. 
 \end{proof}

On the other hand, using $\eqref{approximate}_2$, Corollaries \ref{co3.10}, \ref{co3.15} , we can easily get the second order dissipation of $v$. 
\begin{corollary}\label{co3.16}
There exists some constant $C$ such that
\begin{align}
\int_{0}^{t}\int_{0}^{1}v^2_{xx}\dif x\dif t \le   C(E_0+E^\frac{1}{2}\int_{0}^{t}\mathcal{D}(s)\dif s+E^\frac{3}{2}(t)).  \label{new-hu-17}
\end{align}
\end{corollary}
Combining Lemmas \ref{le3.8}, \ref{le3.9} and Corollaries \ref{co3.10}, \ref{co3.15}, \ref{co3.16}, we get
\begin{lemma}\label{second}
There exists some constant $C$ such that
\begin{align}
\left\|D(v, u, \theta,  \sqrt{\tau   }q, \sqrt{\tau  }S)\right\|_{H^{1}}^{2}+\tau^2 \| \partial_t^2(v, u, \theta, \sqrt{\tau   } q, \sqrt{\tau  } S)\|_{L^2}^2 \nonumber \\
+\int_{0}^{t}(\left\|D^2(v,u,\theta )\right\|_{L^{2}}^{2} + \|D(q, S)\|_{H^1}^2+ \tau^2 \|\partial_t^2(q,S)\|_{L^2}^2 )  \dif t  \nonumber\\
\le C\left(E(0)+E^\frac{1}{2}(t)\int_{0}^{t}\mathcal{D}(s)\dif s+E^\frac{3}{2}(t) \right).
\end{align}
\end{lemma}

Therefore, combining Lemma $\ref{le3.7}$ and $\ref{second}$,  the  Proposition \ref{zong} follows immediately.

\section{Passing to the limit and proof of main theorems}
\textbf{Proof of Theorem \ref{global}:} According to Proposition \ref{zong}, the local solution $(v^\epsilon,u^\epsilon,\theta^\epsilon,q^\epsilon,S^\epsilon)$ can be extend to $[0,\infty)$ by usual continuation methods. Thus, for fixed $\epsilon$, we get a global solution $(v^\epsilon,u^\epsilon,\theta^\epsilon,q^\epsilon,S^\epsilon)$ to system \eqref{approximate}-\eqref{aboundary} satisfying 
{\small
\begin{align}
		&\sup_{0\le t<\infty} \left(\SUM k 0 1 \left\|\partial_t^k\left(v^\epsilon-1, u^\epsilon, \theta^\epsilon-1, \sqrt{\tau   }q^\epsilon, \sqrt{\tau  }S^\epsilon\right)(t, \cdot)\right\|_{H^{2-k}}^{2}+\tau^2 \left\| \partial_t^2 \left(v^\epsilon, u^\epsilon, \theta^\epsilon, \sqrt{\tau} q^\epsilon, \sqrt{\tau} S^\epsilon\right(t, \cdot))\right\|_{L^2}^2\right),\nonumber\\
		&+\int_{0}^{\infty}\left( \sum_{|\alpha|=1}^{2}\left\|D^{\alpha}(v^\epsilon, u^\epsilon, \theta^\epsilon)(t, \cdot)\right\|_{L^{2}}^{2}+ \SUM k 0 1 \left\|\partial_t^k(q^\epsilon,S^\epsilon)(t ,\cdot)\right\|_{H^{2-k}}^{2}+\tau^2 \|(q^\epsilon_{tt}, S^\epsilon_{tt})(t, \cdot)\|_{L^2}^2\right)\dif t\le CE_0,  \label{p1.1}
	\end{align}
	}
where $C$ is a constant independent of $\epsilon  $ and $\tau  $. Therefore, for fixed $\tau$, the uniform bounds of  $(v^\epsilon-1,u^\epsilon,\theta^\epsilon-1,\sqrt{\tau   }q^\epsilon,\sqrt{\tau  }S^\epsilon)$ in $L^\infty([0,\infty),H^2(\Omega))$ implies that there exists $(v,u,\theta,q,S) \in L^\infty([0,\infty),H^2(\Omega))$ such that
$$(v^\epsilon,u^\epsilon,\theta^\epsilon,\sqrt{\tau   }q^\epsilon,\sqrt{\tau  }S^\epsilon)\rightharpoonup (v,u,\theta,\sqrt{\tau   }q,\sqrt{\tau  }S)\quad
\mbox{\rm  weakly - $\ast$ \quad in} \quad L^\infty(\mathbb{R}^+;H^2(\Omega)).$$

On the other hand, since $\partial_t^k (v^\epsilon,u^\epsilon,  \theta^\epsilon,   q^\epsilon, S^\epsilon), k=1, 2$ are uniformly bounded with respect to $\epsilon$  in $L^\infty(0, T;H^{2-k})\cap L^2((0,T; H^{2-k})$ for any $T>0$, by classical compactness theorem, $(v^\epsilon,u^\epsilon,\theta^\epsilon,q^\epsilon,S^\epsilon)$ are relative compact in 
 $C^0([0,T],H^{2-\delta_0} )\cap  C^1([0,T], H^{1-\delta_0})$ for any $\delta_0>0$.
  As a consequence, as $\epsilon\rightarrow 0$ and up to subsequences, 
$$
(v^\epsilon,u^\epsilon,\theta^\epsilon,q^\epsilon,S^\epsilon) \rightarrow (v,u,\theta,q,S), \quad \mbox{\rm  strongly\quad in } \quad C^0([0,T],H^{2-\delta_0} )\cap  C^1([0,T], H^{1-\delta_0}).
$$
This is sufficient to passing to the limit in \eqref{approximate} and \eqref{p1.1}. Thus, we derive that $(v, u, \theta, q, S)$ are classical solutions of  the system \eqref{yuan}-\eqref{boundary} satisfying \eqref{new-hu5-1}. Since  
$$
(v, u, \theta, q, S)\in C([0,\infty),H^{2-\delta_0}(\Omega))\cap L^\infty((0,\infty), H^2(\Omega)) \cap L^2((0,\infty), H^1(\Omega)),
$$ the uniqueness follows immediately. Thus, the proof of Theorem \ref{1.1} is finished.

\textbf{Proof of Theorem \ref{weak convergence}:}  Let $(v^\tau,u^\tau,\theta^\tau,q^\tau,S^\tau)$ be the global solutions obtained in Theorem \ref{1.1} satisfying \eqref{new-hu5-1}. 
Then, there exists $(v^0,u^0,\theta^0) \in L^\infty([0,\infty),H^2(\Omega))$ and $(q^0,S^0) \in L^2([0,\infty),H^2(\Omega))$ such that
$$(v^\tau,u^\tau,\theta^\tau)\rightharpoonup(v^0,u^0,\theta^0) \quad \mbox{\rm  weakly- $\ast$ \quad in} \quad L^\infty([0,\infty);H^2(\Omega)),$$
and
$$(q^\tau,S^\tau)\rightharpoonup(q^0,S^0) \quad \mbox{\rm  weakly - \quad in} \quad L^2([0,\infty);H^2(\Omega)).$$

Since $\partial_t^k(v^\tau, u^\tau, \theta^\tau), k=1, 2$ are uniform bounded  with respect to $\tau$ in $L^2([0,\infty);H^{2-k}(\Omega))$, by classical compactness theorem, for any $\delta_0>0$, $(v^\tau,u^\tau,\theta^\tau)$ are relatively compact in $C([0,T],H^{2-\delta_0}) \cap C^1([0,T], H^{1-\delta_0})$ for any $T>0$. As a consequence, as $\tau  \rightarrow 0$ and up to subsequences, 
$$
(v^\tau,u^\tau,\theta^\tau) \rightarrow (v^0,u^0,\theta^0), \quad \mbox{\rm  strongly\quad in } \quad C([0,T],H^{2-\delta_0}(\Omega))\cap C^1([0,T], H^{1-\delta_0}).
$$

On the other hand, from the estimate \eqref{new-hu5-1} and the constitutive equations $\eqref{yuan}_4, \eqref{yuan}_5$, we have
\begin{align}\label{new-hu3-2}
\sup_{0\le t <\infty} \|(q^\tau, S^\tau)\|_{H^1} \le CE_0. 
\end{align}
Since $(q^\tau_t, S^\tau_t)$ are bounded in $L^2([0,\infty); H^1(\Omega)$, the sequence $(q^\tau, S^\tau)$ are relatively compact in $C([0,T]; H^{1-\delta}(\Omega))$ for any $T>0$.  Therefore, as 
$\tau   \rightarrow 0 $ and up to subsequences, 
$$(q^\tau, S^\tau) \rightarrow (q^0,S^0), \quad \mbox{\rm  strongly\quad in } \quad C([0,T],H^{1-\delta_0}(\Omega)).$$

The uniform boundedness of $(\sqrt{\tau   }q^\tau,\sqrt{\tau  }S^\tau)$ yields $(\tau   q^\tau,\tau  S^\tau)\rightarrow (0,0)$ in $L^\infty((0,\infty),H^2(\Omega))$ , which leads to $(\tau   \partial_tq^\tau,\tau  \partial_tS^\tau)\rightarrow (0,0)$ in $D^\prime((0,\infty)\times \Omega)$ as $\tau  \rightarrow 0$. Thus, taking limit in the equations
 $\eqref{yuan}_4, \eqref{yuan}_5$, we get
\begin{align}\label{new-hu5-2}
q^0=\frac{\kappa(\theta^0)(\theta^0)_x}{v^0},\quad S^0=\mu\frac{(u^0)_x}{v^0}. \quad
\mbox{\rm a.e.} \quad (0,\infty)\times \Omega.
\end{align}
Then, passing to the limit in $\eqref{yuan}_1-\eqref{yuan}_3$ and using \eqref{new-hu5-2}, the functions $(v^0, u^0, \theta^0)$ satisfy the classical compressible Navier-Stokes-Fourier system \eqref{classical}. Thus, the proof of Theorem \ref{1.2} is finished.

{\bf Acknowledgement:} Yuxi Hu's Research is supported by the Fundamental Research Funds for the Central Universities (No. 2023ZKPYLX01).
 

\begin{thebibliography}{aaaaa}

 \bibitem{Bae023}
J. B\"arlin, Formation of singularities in solutions to nonlinear hyperbolic systems with general sources, {\it Nonlinear Anal. Real World Appl.} {\bf 73 } (2023), 103901.


\bibitem{CA48} C. Cattaneo, Sulla conduzione del calore. {\it Atti. Sem. Mat. Fis. Univ. Modena.} {\bf 3}(83) (1948), 83-101.

\bibitem{ChSa015} Chakraborty and J.E. Sader, Constitutive models for linear compressible viscoelastic flows of simple liquids at nanometer length scales, {\it Phys. Fluids} {\bf 27} (2015), 052002.

\bibitem{Chen2007} S.X. Chen, Initial boundary value problems for quasilinear symmetric hyperbolic systems with characteristic boundary, {\it Front. Math. China} {\bf 2}(1) (2007), 87-102. Translated from {\it 
Chinese Annals of Mathematics} Ser. A {\bf 3}(2) (1982), 223-232. 

\bibitem{CG} P.J. Chen and M.E. Gurtin, On second sound in materials with memory, {\it Z. Ang. Math. Phys.}  {\bf 21} (1970), 232-241.


\bibitem{CW03} {G. Q. Chen, D. H. Wang, Global solutions for nonlinear magnetohydrodynamics with large initial data, J. Differential Equations  {\bf 182} (2002) 344-376.}


\bibitem{CHE63} M. Chester, Second sound in solids, {it Physical Review} {\bf 131}(5) (1963), 2013-2015.


 \bibitem{CHR09} C.I. Christov, On frame indifferent formulation of the Maxwell-Cattaneo model of finite-speed heat conduction, {\it Mech. Res. Comm.}  {\bf 36} (2009), 481-486.

\bibitem{CFO} B.D. Coleman, M. Fabrizio, and D.R. Owen, On the thermodynamics of second sound in dielectric crystals, {\it Arch. Rational Mech. Anal.}  {\bf 80} (1986), 135-158.

    \bibitem{CHO} B.D. Coleman, W.J. Hrusa, and D.R. Owen, Stability of Equilibrium for a Nonlinear Hyperbolic System Describing Heat Propagation by Second Sound in Solids, {\it Arch. Rational Mech. Anal.}  {\bf 94} (1986), 267-289.



\bibitem{Frei1} H. Freist\"uhler, A Galilei invariant version of Yong’s model. arXiv 2012.09059 (2020).

%\bibitem{Frei2} H. Freist\"uhler, Time-Asymptotic Stability for First-Order Symmetric Hyperbolic Systems of Balance Laws in Dissipative Compressible Fluid Dynamics, {\it  Quart. Appl. Math.}  {\bf 80}(2022), 597-606.

%\bibitem{Frei3}
% H. Freist\"uhler, Formation of singularities in solutions to Ruggeri's hyperbolic Navier-Stokes equations, arXiv:2305.05426 (2023).


\bibitem{HL24} Y. Hu and Y.C. Li, Initial boundary value problem for relaxed isentropic compressible Navier-Stokes equations, preprint, 2024.

%\bibitem {HR2} Y. Hu and R. Racke, Formation of singularities in one-dimensional thermoelasticity with second sound, {\it Quart. Appl. Math.}  {\bf 72} (2014), 311-321.

\bibitem{HR17} Y. Hu and R. Racke, Compressible Navier-Stokes equations with revised Maxwell's law, {\it J. Math.Fluid Mech.}  {\bf 19} (2017), 77-90.


\bibitem{HR20} Y. Hu and R. Racke, Hyperbolic compressible Navier-Stokes equations, {\it J. Differential Equations} {\bf 269} (2020), 3196-3220.

\bibitem{HR23}
Y. Hu and R. Racke, Global existence versus blow-up for multi-dimensional hyperbolized compressible Navier-Stokes equations, {\it SIAM J. Math. Anal.}  {\bf 55} (5)(2023), 4788-4815.

\bibitem{HR24} Y. Hu and R. Racke, Blow-up of solutions for relaxed compressible Navier-Stokes equations, {\it J. Hyper. Diff. Eqs.} {\bf 21}(1) (2024), 129-141.

\bibitem{HRW22} Y. Hu, R. Racke and N. Wang, Formation of singularities for one-dimensional relaxed compressible Navier-Stokes equations, {\it J. Differential Equations} {\bf 327} (2022), 145-165.

\bibitem{HW2019} Y. Hu and N. Wang, Global existence versus blow-up results for one dimensional compressible Navier-Stokes equations with Maxwell's law, {\it Math. Nachr.} {\bf 292} (2019), 826-840.


 %\bibitem{JR2000} S. Jiang and R. Racke,  Evolution equations in thermoelasticity.
 %{\it $\pi$ Monographs Surveys Pure Appl. Math.} {\bf 112}. Chapman \& Hall/CRC, Boca Raton (2000).

%\bibitem{Kaw83} S. Kawashima, Systems of a hyperbolic-parabolic composite type, with applications to the equations of magnetohydrodynamics. Thesis, Kyoto University (1983).

\bibitem{KA} A.V. Kazhikhov, Cauchy problem for viscous gas equations, {\it Siberian Mathematical Journal} {\bf 23} (1982), 44-49.

\bibitem{Max1867} J.C. Maxwell, On the dynamical theory of gases, {\it Phil. Trans. Roy. Soc. London}, {\bf 157} (1867), 49-88.

\bibitem{Mu67} I. M\"uller, Zum paradoxen der W\"armeleitungstheorie, {\it Zeitschrift f\"ur Physik} {\bf 198} (1967), 329-344.


\bibitem{OLD50} J.G. Oldroyd, On the formulation of rheological equations of state, {\it Proceedings of the Royal Society of London. Series A} {\bf 200} (1950), 523-541.


\bibitem{Peetal} M. Pelton, D. Chakraborty, E. Malachosky, P. Guyot-Sionnest, and J.E. Sader, Viscoelastic flows in simple liquids generated by vibrating nanostructures, {\it Phys. Rev. Letters} {\bf 111} (2013), 244502.

\bibitem{PY2020} Y.-J. Peng, Relaxed Euler systems and convergence to Navier–Stokes equations, {\it Ann. Inst. H. Poinca\'re Anal. Non
Lin\'eaire} {\bf 38}(2) (2021), 369-401. 

\bibitem{Peng-Zhao-22} Y.-J. Peng and L. Zhao, Global convergence to compressible full Navier–Stokes equations by approximation with Oldroyd-type constitutive laws, {\it J. Math. Fluid Mech.} {\bf 24} (2022),  29.

\bibitem{Peng-Zhao-24} Y.-J. Peng and L. Zhao, Global convergence rates from relaxed Euler equations to Navier–Stokes equations with Oldroyd-type constitutive laws, {\it Nonlinearity} {\bf 37}(9) (2024), 095032. 



\bibitem{Rug83}T. Ruggeri, Symmetric-hyperbolic system of conservative equations for a viscous heat conducting fluid. {\it Acta Mech.} {\bf 47} (1983), 167-183.

\bibitem{SCH86} S. Schochet, The compressible Euler equations in a bounded domain: Existence of solutions and the incompressible limit, {\it Commun. Math. Phys} {\bf 104}, (1986) 49-75.


%\bibitem{S85}{ T.C. Sideris, Formation of singularities in three-dimensional
%compressible fluids, {\it Commun. Math. Phys.} {\bf 101} (1985),
%475-485.}


%\bibitem{TAY} M.E. Taylor, Pseudodifferential Operators and Nonlinear PDE, Progress Math., vol. 100, Birkhäuser, Boston, 1991.

   \bibitem{TA}{M.A. Tarabek, On the existence of smooth solutions in one-dimensional
nonlinear thermoelasticity with second sound, {\it Quart. Appl. Math.} {\bf 50}
 (1992), 727--742.}

\bibitem{HW2020}  N. Wang and Y. Hu, Blow-up of solutions for compressible Navier-Stokes equations with revised Maxwell's law, {\it Applied Mathematics Letters}, {\bf 103} (2020), 106221.


%\bibitem{WH25} N. Wang and Y. Hu, Global solvability and relaxation limit of hyperbolic compressible Navier-Stokes equations on semi-axis, {\it Acta Math. Appl. Sinica} (in Chinese), to appear, 2025.

\bibitem{Yong14} W.A. Yong, Newtonian limit of Maxwell fluid flows, {\it Arch. Rational Mech. Anal.} {\bf 214} (2014), 913-922.



\end{thebibliography}
\end{document}